\documentclass[11pt]{article}
\usepackage{smile}

\usepackage{fullpage}
\usepackage{lscape}
\usepackage{bigints}
\usepackage{framed}
\usepackage{mdframed}
\usepackage{enumerate}
\usepackage[inline]{enumitem}
\usepackage[T1]{fontenc}
\usepackage{moresize}
\usepackage{bm}
\usepackage{bbm}
\usepackage{dsfont}
\usepackage{amsmath}
\usepackage{amssymb}
\usepackage{amsthm}
\usepackage{amsfonts}
\usepackage{stmaryrd}
\usepackage{array}
\usepackage{mathrsfs}
\usepackage{mathtools} 
\usepackage{extarrows}
\usepackage{stackrel}
\usepackage{relsize,exscale}
\usepackage{scalerel}
\usepackage[nodisplayskipstretch]{setspace}
\usepackage{color}
\usepackage[usenames,dvipsnames]{xcolor}
\usepackage{cancel}
\usepackage{soul}
\usepackage{undertilde}
\usepackage{xfrac}
\usepackage{siunitx}
\usepackage{graphicx}
\usepackage{float}
\usepackage{rotating}
\usepackage{subcaption}
\usepackage{overpic}
\usepackage[all]{xy}
\usepackage{tikz}
\usetikzlibrary{arrows,matrix,positioning,calc,automata,patterns}
\usepackage{booktabs}
\usepackage{dcolumn}
\usepackage{multirow}
\usepackage{diagbox}
\usepackage{tabularx}
\usepackage{verbatim}
\usepackage{listings}
\usepackage[ruled,vlined]{algorithm2e}
\usepackage{fancyvrb}
\usepackage{hyperref}
\usepackage[round]{natbib}
\usepackage{sectsty}

\hypersetup{
    bookmarks=true,         
    unicode=false,          
    pdftoolbar=true,        
    pdfmenubar=true,        
    pdffitwindow=false,     
    pdfstartview={FitH},    
    pdftitle={My title},    
    pdfauthor={Author},     
    pdfsubject={Subject},   
    pdfcreator={Creator},   
    pdfproducer={Producer}, 
    pdfkeywords={key1, key2}, 
    pdfnewwindow=true,      
    colorlinks=true,        
    linkcolor=blue,         
    citecolor=blue,         
    filecolor=blue,         
    urlcolor=cyan           
}

\def\scaleint#1{\vcenter{\hbox{\scaleto[3ex]{\displaystyle\int}{#1}}}}

\usepackage{stackengine}
\stackMath
\newcommand\tenq[2][1]{%
\def\useanchorwidth{T}%
\ifnum#1>1%
\stackunder[0pt]{\tenq[\numexpr#1-1\relax]{#2}}{\!\scriptscriptstyle\thicksim}%
\else%
\stackunder[1pt]{#2}{\!\scriptstyle\thicksim}%
\fi%
}

\makeatletter
\DeclareRobustCommand\widecheck[1]{{\mathpalette\@widecheck{#1}}}
\def\@widecheck#1#2{%
    \setbox\z@\hbox{\m@th$#1#2$}%
    \setbox\tw@\hbox{\m@th$#1%
       \widehat{%
          \vrule\@width\z@\@height\ht\z@
          \vrule\@height\z@\@width\wd\z@}$}%
    \dp\tw@-\ht\z@
    \@tempdima\ht\z@ \advance\@tempdima2\ht\tw@ \divide\@tempdima\thr@@
    \setbox\tw@\hbox{%
       \raise\@tempdima\hbox{\scalebox{1}[-1]{\lower\@tempdima\box
\tw@}}}%
    {\ooalign{\box\tw@ \cr \box\z@}}}
\makeatother

\def\given{\,|\,}

\def\Biggiven{\,\Big{|}\,}
\def\tr{\mathop{\text{tr}}\kern.2ex}

\def\tX{{\tilde X}}
\def\tY{{\tilde Y}}

\def\tN{{\tilde N}}
\def\tT{{\tilde T}}
\def\P{{\mathrm P}}

\def\E{{\mathrm E}}

\def\N{{\mathbbm N}}

\def\d{{\mathrm d}}

\newcommand{\zahl}[1]{\llbracket #1\rrbracket}
\newcommand\yestag{\addtocounter{equation}{1}\tag{\theequation}}
\newcolumntype{L}[1]{>{\raggedright\let\newline\\\arraybackslash\hspace{0pt}}m{#1}}
\newcolumntype{C}[1]{>{  \centering\let\newline\\\arraybackslash\hspace{0pt}}m{#1}}
\newcolumntype{R}[1]{>{ \raggedleft\let\newline\\\arraybackslash\hspace{0pt}}m{#1}}
\newcolumntype{d}[1]{D{.}{.}{#1}}
\newcolumntype{H}{>{\setbox0=\hbox\bgroup}c<{\egroup}@{}}
\newcolumntype{Z}{>{\setbox0=\hbox\bgroup}c<{\egroup}@{\hspace*{-\tabcolsep}}}
\newcolumntype{b}{X}
\newcolumntype{s}{>{\hsize=.5\hsize}X}

\numberwithin{equation}{section}

\newtheorem{theorem}{Theorem}[section]
\newtheorem{lemma}{Lemma}[section]
\newtheorem{proposition}{Proposition}[section]

\newtheorem{innercustomgeneric}{\customgenericname}
\providecommand{\customgenericname}{}
\newcommand{\newcustomtheorem}[2]{%
  \newenvironment{#1}[1]
  {%
   \renewcommand\customgenericname{#2}%
   \renewcommand\theinnercustomgeneric{##1}%
   \innercustomgeneric
  }
  {\endinnercustomgeneric}
}
\newcustomtheorem{customdefinition}{Definition}
\newcustomtheorem{customdefinitions}{Definitions}
\newcustomtheorem{customtheorem}{Theorem}
\newcustomtheorem{customassumption}{Assumption}
\newcustomtheorem{customlemma}{Lemma}
\newcustomtheorem{customexample}{Example}
\theoremstyle{definition}

\newtheorem{remark}{Remark}[section]

\usepackage{enumitem}
\makeatletter
\newcommand{\mylabel}[2]{#2\def\@currentlabel{#2}\label{#1}}
\makeatother

\graphicspath{{./fig3/}}

\newcommand{\nb}[1]{{{#1}}}

\allowdisplaybreaks

\begin{document}

\setlength{\abovedisplayskip}{5pt}
\setlength{\belowdisplayskip}{5pt}
\setlength{\abovedisplayshortskip}{5pt}
\setlength{\belowdisplayshortskip}{5pt}
\hypersetup{colorlinks,breaklinks,urlcolor=blue,linkcolor=blue}

\title{\LARGE Limit theorems of Chatterjee's rank correlation}

\author{Zhexiao Lin\thanks{Department of Statistics, University of California, Berkeley, CA 94720, USA; e-mail: {\tt zhexiaolin@berkeley.edu}}~~~and~
Fang Han\thanks{Department of Statistics, University of Washington, Seattle, WA 98195, USA; e-mail: {\tt fanghan@uw.edu}}
}

\date{\today}

\maketitle

\vspace{-1em}

\begin{abstract}
Establishing the limiting distribution of Chatterjee's rank correlation for a general, possibly non-independent, pair of random variables has been eagerly awaited by many. This paper shows that (a) Chatterjee's rank correlation is asymptotically normal as long as one variable is not a measurable function of the other, (b) the corresponding asymptotic variance is uniformly bounded by 36, and (c) a consistent variance estimator exists. Similar results also hold for Azadkia-Chatterjee's graph-based correlation coefficient, a multivariate analogue of Chatterjee's original proposal. The proof is given by appealing to H\'ajek representation and Chatterjee's nearest-neighbor CLT.
\end{abstract}


\section{Introduction}\label{sec:intro}

Let $Y$ be a random variable in $\bR$ and $X$ be a random vector in $\bR^d$ that are defined on the same probability space and of joint and marginal distribution functions $F_{X,Y}$ and $F_X,F_Y$, respectively. Throughout the paper, we consider $F_{X,Y}$ to be {\it fixed} and {\it continuous}. 

To measure the dependence strength between $X$ and $Y$, \cite{MR3024030} introduced the following population quantity,
\begin{align}\label{eq:xi}
\xi=\xi(X,Y):=\;&  \frac{\scaleint{4.5ex}\,\Var\big\{\E\big[\ind\big(Y\geq y\big) \given X \big] \big\} \d F_{Y}(y)}{\scaleint{4.5ex}\,\Var\big\{\ind\big(Y\geq y\big)\big\}\d F_{Y}(y)},
\end{align}
with $\ind(\cdot)$ representing the indicator function. This quantity, termed the {\it Dette-Siburg-Stoimenov's dependence measure} in literature, enjoys desirable properties of being between 0 and 1 and being (a) 0 if and only if $Y$ is independent of $X$; and (b) 1 if and only if $Y$ is a measurable function of $X$.

Consider $(X_1,Y_1), \ldots, (X_n,Y_n)$ to be $n$ independent copies of $(X,Y)$.  For any $i\in\{1,\ldots,n\}$, let $R_i:=\sum_{j=1}^n\ind(Y_j\leq Y_i)$ denote the rank of $Y_i$, and 
\nb{let $N_k(i)$ and $\overline{N}_k(i)$ index the $k$-th nearest neighbor (NN) of $X_i$ among $\{X_j\}_{j=1}^n$ (under the Euclidean metric $\|\cdot\|$) and the right $k$-th NN of $X_i$ among $\{X_j\}_{j=1}^n$ (when $d=1$, with $\overline{N}_k(i) :=i$ if $X_i$ is among the $k$ largest).} To estimate $\xi$ based only on $(X_i,Y_i)$'s, \cite{azadkia2019simple} and \cite{chatterjee2020new} introduced the following two correlation coefficients:
\begin{align}
\text{(Azadkia-Chatterjee)}~~~ \xi_n &:= \frac{6}{n^2-1} \sum_{i=1}^n \min\big\{R_i, R_{N_1(i)}\big\} - \frac{2n+1}{n-1},~~~\text{for $d\geq 1$}; \label{eq:xin}\\
\text{(Chatterjee)}~~~\overline{\xi}_n &:= 1 - \frac{3}{n^2-1} \sum_{i=1}^n \Big\lvert R_{\overline{N}_1(i)} - R_i \Big\rvert, ~~~\text{when }d=1. \label{eq:barxin}
\end{align}

\citet[Theorem 2.2]{azadkia2019simple} and \citet[Theorem 1.1]{chatterjee2020new} showed that, under some very mild conditions, both $\xi_n$ and $\overline\xi_n$ constitute strongly consistent estimators of $\xi$. However, deriving the limiting distributions of $\xi_n$ and $\overline\xi_n$ is also of interest to statisticians. Unfortunately, unless $X$ and $Y$ are independent --- implying that $N_1(i)$ and $\overline N_1(i)$'s are independent of $Y_1,\ldots,Y_n$ --- this is apparently still an open problem. 

The following two theorems answer this call, and are the main results of this paper.

\begin{theorem}[Asymptotic normality]\label{thm:main}
For any fixed and continuous $F_{X,Y}$ such that $Y$ is not a measurable function of $X$ almost surely, we have
  \begin{align}\label{eq:main1}
    \big(\xi_n - \E[\xi_n]\big)/\sqrt{\Var[\xi_n]} \longrightarrow N(0,1) ~~ {\rm in~distribution},
  \end{align}
  and
  \begin{align*}
    \big(\overline{\xi}_n - \E[\overline{\xi}_n]\big)/\sqrt{\Var[\overline{\xi}_n]} \longrightarrow N(0,1) ~~ {\rm in~distribution}~~({\rm if}~d=1).
  \end{align*}
\end{theorem}

For any $a,b \in \bR$, write $a \vee b = \max\{a,b\}$ and $a \wedge b = \min\{a,b\}$. Define 
\begin{align*}
   &\hat{\sigma}^2:=\\
   & 36\Big\{\frac{1}{n^3} \sum_{i=1}^n \Big(R_i \wedge R_{N_1(i)}\Big)^2 \Big(1+ \ind \Big( i = N_1(N_1(i)) \Big) \Big) \\
  &+ \frac{1}{n^3} \sum_{i=1}^n \Big(R_i \wedge R_{N_1(i)}\Big)\Big(R_i \wedge R_{N_2(i)} \Big) \Big( 2\ind \Big( i \neq N_1(N_1(i)) \Big) + \Big\lvert \Big\{j: j \neq i, N_1(j) = N_1(i)\Big\} \Big\rvert \Big)\\
  &- \frac{1}{n^3} \sum_{i=1}^n \Big(R_i \wedge R_{N_1(i)}\Big)  \Big(R_{N_2(i)} \wedge R_{N_3(i)}\Big) \Big( 1 + \ind \Big( i \neq N_1(N_1(i)) \Big) + \Big\lvert \Big\{j: j \neq i, N_1(j) = N_1(i)\Big\} \Big\rvert \Big)\\
  &+ \frac{4}{n^2(n-1)} \sum_{\substack{i,j=1\\i\neq j}}^n \ind\Big(R_i \le R_j \wedge R_{N_1(j)}\Big)\Big(R_i \wedge R_{N_1(i)}\Big)\\
  & - \frac{2}{n^2(n-1)} \sum_{\substack{i,j=1\\i\neq j}}^n \ind\Big(R_i \le R_j \wedge R_{N_1(j)}\Big)\Big(R_{N_1(i)} \wedge R_{N_2(i)}\Big)\\
  &+ \frac{1}{n^2(n-1)} \sum_{\substack{i,j=1\\i\neq j}}^n \Big(R_i \wedge R_{N_1(i)} \wedge R_j \wedge R_{N_1(j)}\Big) - 4 \Big[\frac{1}{n^2} \sum_{i=1}^n \Big(R_i \wedge R_{N_1(i)}\Big)\Big]^2\Big\},
\end{align*}
and
\begin{align*}
  & \hat{\overline{\sigma}}^2:= \\
  &36\Big\{\frac{1}{n^3} \sum_{i=1}^n \Big(R_i \wedge R_{\overline{N}_1(i)}\Big)^2 + \frac{2}{n^3} \sum_{i=1}^n \Big(R_i \wedge R_{\overline{N}_1(i)}\Big)\Big(R_i \wedge R_{\overline{N}_2(i)} \Big) \\
  &- \frac{2}{n^3} \sum_{i=1}^n \Big(R_i \wedge R_{\overline{N}_1(i)}\Big)  \Big(R_{\overline{N}_2(i)} \wedge R_{\overline{N}_3(i)}\Big)+ \frac{4}{n^2(n-1)} \sum_{\substack{i,j=1\\i\neq j}}^n \ind\Big(R_i \le R_j \wedge R_{\overline{N}_1(j)}\Big)\Big(R_i \wedge R_{\overline{N}_1(i)}\Big) \\
  &- \frac{2}{n^2(n-1)} \sum_{\substack{i,j=1\\i\neq j}}^n \ind\Big(R_i \le R_j \wedge R_{\overline{N}_1(j)}\Big)\Big(R_{\overline{N}_1(i)} \wedge R_{\overline{N}_2(i)}\Big)\\
  &+ \frac{1}{n^2(n-1)} \sum_{\substack{i,j=1\\i\neq j}}^n \Big(R_i \wedge R_{\overline{N}_1(i)} \wedge R_j \wedge R_{\overline{N}_1(j)}\Big) - 4 \Big[\frac{1}{n^2} \sum_{i=1}^n \Big(R_i \wedge R_{\overline{N}_1(i)}\Big)\Big]^2\Big\}.
\end{align*}

\begin{theorem}[Variance estimation]\label{thm:var}
  For any fixed continuous $F_{X,Y}$, it holds true that
  \begin{align}\label{eq:var-est1}
    \hat{\sigma}^2 - n \Var[\xi_n] \longrightarrow 0 ~~ {\rm in~probability},
  \end{align}
  and
  \begin{align}\label{eq:var-est2}
    \hat{\overline{\sigma}}^2 - n \Var[\overline{\xi}_n] \longrightarrow 0 ~~ {\rm in~probability}.
  \end{align}
\end{theorem}

The following two propositions further complement Theorems \ref{thm:main} and \ref{thm:var}. 

\begin{proposition}[Asymptotic bias, \cite{azadkia2019simple}] \phantomsection \label{prop:bias} Assume $F_{X,Y}$ to be fixed and continuous.
  \begin{enumerate}[itemsep=-.5ex,label=(\roman*)]
    \item\label{prop:bias,null} If $X$ and $Y$ are independent, then
    \begin{align*}
      \E[\xi_n] = -\frac{1}{n-1}~~~{\rm and}~~~ \E[\overline{\xi}_n] = 0~~({\rm if}~d=1).
    \end{align*}
    \item\label{prop:bias,alter} If there exist fixed constants $\beta,C,C_1,C_2>0$ such that for any $t \in \bR$ and $x,x' \in \bR^d$,
    \begin{align*}
      &\Big\lvert \P\big(Y \ge t \given X = x\big) - \P\big(Y \ge t \given X=x'\big) \Big\rvert \le C(1+\lVert x \rVert^\beta + \lVert x' \rVert^\beta)\lVert x-x' \rVert\\
      {\rm and}~~~&\P(\lVert X \rVert \ge t) \le C_1 e^{-C_2 t},
    \end{align*}
    we then have
    \begin{align*}
      \Big\lvert \E[\xi_n] - \xi \Big\rvert = O\Big(\frac{(\log n)^{d+\beta+1+\ind(d=1)}}{n^{1/d}}\Big) ~~{\rm and}~~ \Big\lvert \E[\overline{\xi}_n] - \xi \Big\rvert = O\Big(\frac{(\log n)^{\beta+3}}{n}\Big)~~({\rm if}~d=1).
    \end{align*}
  \end{enumerate}
\end{proposition}

\begin{proposition}[Asymptotic variance] \phantomsection \label{prop:variance} Assume $F_{X,Y}$ to be fixed and continuous.
  \begin{enumerate}[itemsep=-.5ex,label=(\roman*)]
    \item\label{prop:variance,exist} \nb{The limits of $n \Var[\xi_n]$ and $n \Var[\overline{\xi}_n]$ exist.}
    \item\label{prop:variance,lower} If $Y$ is not a measurable function of $X$ almost surely, 
    \begin{align*}
      \nb{\lim_{n \to \infty}} \left\{n \Var[\xi_n]\right\} > 0~~~{\rm and}~~~ \nb{\lim_{n \to \infty}}  \left\{n \Var[\overline{\xi}_n] \right\} > 0 ~~({\rm if}~d=1).
    \end{align*}
    On the other hand, if $Y$ is a measurable function of $X$ almost surely, then
    \begin{align*}
      \lim_{n \to \infty}  \left\{n \Var[\xi_n] \right\} = 0~~~{\rm and}~~~ \lim_{n \to \infty}  \left\{n\Var[\overline{\xi}_n] \right\} = 0 ~~({\rm if}~d=1).
    \end{align*}
    \item\label{prop:variance,upper} It holds true that 
    \begin{align}\label{eq:upper-bound}
      \nb{\lim_{n \to \infty}}  \left\{n \Var[\xi_n] \right\} <\infty~~{\rm and}~~\nb{\lim_{n \to \infty}}  \left\{n \Var[\overline{\xi}_n] \right\} \le 36~~({\rm if}~d=1).
    \end{align}
    If in addition $F_X$ is absolutely continuous, then
    \begin{align}\label{eq:upper-bound-2}
      \nb{\lim_{n \to \infty}}  \left\{n \Var[\xi_n] \right\} \le 36 - 9 \kq_d + 9 \ko_d,
    \end{align}
    where $\kq_{d}$ and $\ko_{d}$ are two positive constants depending only on $d$, with explicit values:
    \begin{align*}
    &\kq_d:=\Big\{2-I_{3/4}\Big(\frac{d+1}{2},\frac{1}{2}\Big)\Big\}^{-1},
    ~~~~
    I_{x}(a,b):=\frac{\int_{0}^{x}t^{a-1}(1-t)^{b-1} \d t}
                     {\int_{0}^{1}t^{a-1}(1-t)^{b-1} \d t},
    \yestag\label{eq:defn_kqd}\\
    &\ko_{d}:=\int_{\Gamma_{d;2}}\exp\Big[-\lambda\Big\{B(\mw_1,\lVert\mw_1\rVert_{})\cup B(\mw_2,\lVert\mw_2\rVert_{})\Big\}\Big]\d(\mw_1,\mw_2),
    \yestag\label{eq:defn_kod}\\
    &\Gamma_{d;2}:=\Big\{(\mw_1,\mw_2)\in(\bR^d)^2: \max(\lVert\mw_1\rVert_{},\lVert\mw_2\rVert_{})<\lVert\mw_1-\mw_2\rVert_{}\Big\},
    \end{align*} 
    $B(\mw_1,r)$ denotes the ball of radius $r$ centered at $\mw_1$, and $\lambda(\cdot)$ denotes the Lebesgue measure. 
  \end{enumerate}
\end{proposition}

\begin{remark}
It is worth noting that \eqref{eq:main1} and \eqref{eq:var-est1} hold without requiring $F_X$ to be absolutely continuous (with regard to the Lebesgue measure). In particular, $\xi_n$ is still asymptotically normal even when $X$ is supported on a low-dimensional manifold in $\bR^d$, e.g., the $(d-1$)-dimensional unit sphere. 
\end{remark}

\begin{remark}\label{remark:degenerate}
For establishing asymptotic normality, Theorem \ref{thm:main} requires $Y$ to be not a measurable function of $X$. When $Y$ is perfectly dependent on $X$, Proposition \ref{prop:variance} suggests that $\xi_n$ and $\overline \xi_n$ are degenerate; indeed, \citet[Remark 9 after Theorem 1.1]{chatterjee2020new} showed that when $Y$ is an increasing transformation of $X$, $\overline\xi_n=(n-2)/(n+1)$, which reduces to a deterministic constant. \nb{The general forms of \(\xi_n\) and \(\overline{\xi}_n\) when \(Y\) is perfectly dependent on \(X\) are currently still open problems.}
\end{remark}

\begin{remark}
The assumptions in Proposition \ref{prop:bias}\ref{prop:bias,alter} correspond to Assumptions A1 and A2 in \cite{azadkia2019simple}. Its proof is a minor twist to that of \citet[Theorem 4.1]{azadkia2019simple}, which we credit this proposition to. On the other hand, Proposition \ref{prop:variance} is genuinely new, although the constants in \eqref{eq:defn_kqd} and \eqref{eq:defn_kod} can be traced to \cite{MR937563}, \cite{MR914597}, and in particular, \citet[Theorem 3.1]{shi2021ac}.
\end{remark}


Combining Theorems \ref{thm:main}, \ref{thm:var} with Propositions \ref{prop:bias} and \ref{prop:variance}, when $d=1$, one could immediately establish confidence intervals for $\xi$ using either $\xi_n$ or $\overline\xi_n$ since the asymptotic bias in this case is root-$n$ ignorable. For instance, as $d=1$ and $n$ large enough, an $1-\alpha$ confidence interval of $\xi$ can be constructed as
\[
(\overline\xi_n-z_{1-\alpha/2}\cdot\hat{\overline\sigma}/\sqrt{n}, ~~\overline\xi_n+z_{1-\alpha/2}\cdot\hat{\overline\sigma}/\sqrt{n}),
\]
where for any $\beta\in(0,1)$, $z_\beta$ represents the $\beta$-quantile of a standard normal distribution. One could similarly construct large-sample tests for the following null hypothesis
\[
H_0: \xi \leq \kappa,~~~\text{ (for a given and fixed }\kappa<1)
\]
using, e.g., \nb{the test with significance level $\alpha \in (0,1)$ is 
\begin{align}\label{eq:test}
\overline{T}:=\ind(\overline\xi_n>\kappa+z_{1-\alpha}\hat{\overline\sigma}/\sqrt{n})
\end{align}
 and the p-value is $1 - \Phi(\sqrt{n}(\overline\xi_n - \kappa)/\hat{\overline\sigma})$, where $\Phi$ is the CDF of the standard normal distribution. The size validity, consistency and local power analysis of the test are established in the following proposition.
}

\nb{
\begin{proposition}\phantomsection \label{prop:test} Assume $F_{X,Y}$ to be continuous and $Y$ is not a measurable function of $X$ almost surely. Assume $d=1$ and the assumptions of Proposition~\ref{prop:bias}\ref{prop:bias,alter} hold.  
  \begin{enumerate}[itemsep=-.5ex,label=(\roman*)]
    \item\label{prop:size} For any fix $F_{X,Y}$ satisfying $H_0: \xi \leq \kappa$, denoting $\P_{H_0}$ as the corresponding probability measure, we have $\limsup_{n \to \infty} \P_{H_0}(T=1) \le \alpha$.
    \item\label{prop:power} For any fix $F_{X,Y}$ violating $H_0: \xi \leq \kappa$, denoting $\P_{H_1}$ as the corresponding probability measure, we have $\lim_{n \to \infty} \P_{H_1}(T=1)=1$.
    \item\label{prop:local} For a sequence $F_{X,Y}$ satisfying $\xi^{(n)} = \kappa + n^{-1/2} h$ for a fixed $h>0$, denoting $\P_{H_{1,n}}$ as the corresponding probability measure, we have 
    \[
    \lim_{n \to \infty} \P_{H_{1,n}}(T=1)=1 - \Phi(z_{1-\alpha} - h/\overline\sigma),
    \] 
    where $\overline\sigma^2 = \lim_{n \to \infty} \{n \Var[\overline{\xi}_n] \}$ and $\Phi$ is the CDF of the standard normal distribution.
  \end{enumerate}
\end{proposition}
}

\begin{remark}\label{remark:debias}
Checking Proposition \ref{prop:bias}, when $d>1$, an asymptotically non-ignorable bias term may appear in the central limit theorem (CLT) and thus confidence intervals can only be established for $\E \xi_n$ instead of $\xi$. To further debias $\xi_n$, enforcing more assumptions on $F_{X,Y}$ seems inevitable to us. A possible approach is to follow the similar derivations made in \cite{MR3909934}, who studied the problem of multivariate entropy estimation using NN methods. \nb{As long as we can find an estimator \(\Delta_n\) of \(\mathbb{E}[\xi_n] - \xi\) such that the difference is negligible at the \(\sqrt{n}\) rate, all results in Proposition~\ref{prop:test} apply directly to the test statistic \(\xi_n - \Delta_n\) using the same variance estimator \(\hat{\sigma}^2\).}
\end{remark}

\nb{
\begin{remark}
It is worth noting that in the case of \(\kappa = 0\), Proposition~\ref{prop:bias}\ref{prop:local} does not contradict the findings of \citet{shi2020power} and \citet{cao2020correlations}, who showed that Chatterjee's rank correlation exhibits zero local power under the standard root-\(n\) asymptotic framework. We refer readers to \citet[Theorem 2.2]{auddy2021exact} for related results and discussion. Proposition~\ref{prop:bias}\ref{prop:local} extends their conclusions to settings beyond independence.

\end{remark}
}

\nb{\begin{remark}
The codes for computing $\overline{\xi}_n$ and $\hat{\overline{\sigma}}^2$ are available at \url{https://github.com/zhexiaolin/Limit-theorems-of-Chatterjee-s-rank-correlation}. The codes for the empirical studies are also in the repository.
\end{remark}
}

\subsection{Related literature}\label{sec:liter}

The study of Dette-Siburg-Stoimenov's dependence measure \citep{MR3024030} is receiving considerably increasing attention, partly due to the introduction of Chatterjee's rank correlation \citep{chatterjee2020new} as an elegant approach to estimating it. Nowadays, this growing literature has included \cite{azadkia2019simple}, \cite{cao2020correlations}, \cite{shi2020power}, \cite{gamboa2022global}, \cite{deb2020kernel}, \cite{huang2020kernel}, \cite{auddy2021exact}, \cite{shi2021ac}, \cite{lin2021boosting}, \cite{fuchs2021bivariate}, \cite{azadkia2021fast}, \cite{griessenberger2022multivariate}, \cite{strothmann2022rearranged}, \cite{zhang2022asymptotic}, \cite{bickel2022measures}, and \cite{chatterjee2022estimating}, among many others. We also refer the readers to \cite{han2021extensions} for a short survey on some most recent progress.

Below we outline the results in literature that are most relevant to Theorem \ref{thm:main}.

\begin{enumerate}[itemsep=-.5ex,label=(\arabic*)]
\item In his original paper, Chatterjee established the asymptotic normality of $\overline{\xi}_n$ under an important additional assumption that $X$ is independent of $Y$. In particular, he showed
\begin{align}\label{eq:chatterjee-var}
\sqrt{n} \overline{\xi}_n \longrightarrow N(0,2/5)~~\text{in distribution},
\end{align}
if $Y$ is continuous and independent of $X$ \citep[Theorem 2.1]{chatterjee2020new}.
\item Although Azadkia and Chatterjee introduced $\xi_n$ as an extension of $\overline\xi_n$ to multivariate $X$, their results did not include a CLT for $\xi_n$, which was listed as an open problem in \cite{azadkia2019simple}. Notable progress was later made by \cite{deb2020kernel} and \cite{shi2021ac}, which we shall detail below.
\item In \cite{deb2020kernel}, the authors generalized Azadkia and Chatterjee's original proposal to arbitrary metric space via combining the graph- and kernel-based methods. In particular, under independence between $X$ and $Y$ and some additional assumptions on $F_{X,Y}$, \citet[Theorem 4.1]{deb2020kernel} established the following CLT for $\xi_n$,
\[
\xi_n/S_n \longrightarrow N(0,1) ~~\text{in distribution},
\]
where $S_n$ is a data-dependent normalizing statistic.
\item In \cite{shi2021ac}, the authors re-investigated the proof of \cite{deb2020kernel} and, in particular, derived the closed form of the limit of $\Var[\xi_n]$. More specifically, \citet[Theorem 3.1(ii)]{shi2021ac} showed that, under independence between $X$ and $Y$ and some additional assumptions on $F_{X,Y}$, 
\begin{align}\label{eq:AC-var}
\sqrt{n} \xi_n \longrightarrow N\Big(0,\frac{2}{5}+\frac{2}{5} \kq_d + \frac{4}{5}\ko_d\Big) ~~\text{in distribution}, 
\end{align}
where $\kq_d$ and $\ko_d$ are two positive constants that only depend on $d$ and were explicitly defined in Proposition \ref{prop:variance}\ref{prop:variance,upper}.
\item In a related study, in order to boost the power of independence testing, \cite{lin2021boosting} revised $\overline{\xi}_n$ via incorporating more than one right nearest neighbor to its construction. Assuming independence between $X$ and $Y$ and some assumptions on $F_{X,Y}$, \citet[Theorem 3.2]{lin2021boosting} established the following CLT for their correlation coefficient $\overline\xi_{n,M}$ (with $M$ representing the number of right NNs to be included): 
\[
\sqrt{nM} \overline\xi_{n,M} \longrightarrow N(0,2/5) ~~\text{in distribution}, 
\]
as long as $M$ is increasing at a certain rate.
\end{enumerate}
 
All the above CLTs only hold when {\it $Y$ is independent of $X$}. The following papers, on the other hand, studied the statistics' behavior when $Y$ is possibly dependent on $X$. They, however, can only handle {\it local alternatives}, i.e., such distributions where the dependence between $X$ and $Y$ is so weak that $F_{X,Y}$ is very close to $F_XF_Y$. 
\begin{enumerate}[itemsep=-.5ex,label=(\arabic*)]
\setcounter{enumi}{6}
\item Assuming $\xi = \xi^{(n)} \to 0$ as $n \to \infty$ at a certain rate, \citet[Theorem 2.3]{auddy2021exact} showed 
\[
\sqrt{n} (\xi_n - \xi^{(n)}) \longrightarrow N(0,2/5)~~\text{in distribution}. 
\]
\item\label{remark-enum-8} For quadratic mean differentiable (QMD) classes of alternatives to the null independence one, \citet[Section 4.4]{cao2020correlations} and \citet[Proof of Theorem 1]{shi2020power} (the latter is focused on the special mixture and rotation type alternatives) established CLTs for Chatterjee's rank correlation $\overline{\xi}_n$ via Le Cam's third lemma. 
\item Under similar local dependence conditions as \ref{remark-enum-8}, \citet[Proof of Theorem 4.1]{shi2021ac} established the CLTs for Azadkia-Chatterjee's graph-based correlation coefficient $\xi_n$.
\end{enumerate}

\subsection{Proof sketch}\label{sec:proof-sketch}

To establish Theorem~\ref{thm:main}, the first and most important step is to find the correct forms of H\'ajek representations \citep{MR1680991} for $\xi_n$ and $\overline{\xi}_n$ with regard to a general distribution function $F_{X,Y}$ that is not necessarily equal to $F_XF_Y$. This step is technically highly challenging as we have to carefully monitor the dependence between $X$ and $Y$; it shall occupy the most of the rest paper. Interestingly, the newly found H\'ajek representation is distinct from that used in \cite{deb2020kernel}, \cite{cao2020correlations}, \cite{shi2020power}, \cite{auddy2021exact}, and \cite{lin2021boosting}, although reducing to it under independence; see Remark \ref{eq:hajek} ahead for more discussions about this point.

For sketching the proof of Theorem \ref{thm:main}, let us first introduce some necessary notation. For any $t \in \bR$, define
\begin{align}\label{eq:h}
  G_X(t):= \P\big(Y \ge t \given X\big)~~~{\rm and}~~~h(t): = \E \Big[G_X^2(t)\Big].
\end{align}
Ahead we will show that the H\'ajek representations of $\xi_n$ and $\overline{\xi}_n$ take the forms
\begin{align}\label{eq:xin*}
\xi_n^* := \frac{6n}{n^2-1} \Big(\sum_{i=1}^n \min\big\{F_Y(Y_i), F_Y(Y_{N_1(i)})\big\} + \sum_{i=1}^n h(Y_i) \Big)
\end{align}
and
\begin{align}\label{eq:barxin*}
  \overline{\xi}_n^* := \frac{6n}{n^2-1} \Big(\sum_{i=1}^n \min\big\{F_Y(Y_i), F_Y(Y_{\overline{N}_1(i)})\big\} + \sum_{i=1}^n h(Y_i) \Big).
\end{align}

Why so? Below we give some intuition. Let us use ``$\wedge$'' to represent the minimum of two numbers and focus on $\xi_n$ as the analysis for $\overline{\xi}_n$ is identical. From \eqref{eq:xin}, $\xi_n$ takes the form 
\begin{align}\label{eq:hajak-0}
n^{-2} \sum_{i=1}^n [R_i \wedge R_{N_1(i)}]
\end{align}
and a natural component of its H\'ajek representation shall be
\begin{align}\label{eq:hajak-1}
n^{-1} \sum_{i=1}^n [F_Y(Y_i) \wedge F_Y(Y_{N_1(i)})], 
\end{align}
which is via replacing the empirical distribution by the population one. We use H\'ajek projection \cite[Lemma 11.10]{MR1652247} to find the remaining component via checking the difference between \eqref{eq:hajak-0} and \eqref{eq:hajak-1}.

Fix an integer $k\in[1,n]$ and consider the projection of \eqref{eq:hajak-0} on $(X_k,Y_k)$. From the definition of ranks, we have 
\begin{align*}
R_i \wedge R_{N_1(i)} &= \sum_{j=1}^n \ind(Y_j \le Y_i \wedge Y_{N_1(i)})\\ 
&= \ind(Y_k \le Y_i \wedge Y_{N_1(i)}) + \sum_{j=1,j \neq k}^n \ind(Y_j \le Y_i \wedge Y_{N_1(i)}). 
\end{align*}
Then $\xi_n$, of the form $n^{-2} \sum_{i=1}^n [R_i \wedge R_{N_1(i)}]$, can be decomposed as the summation of the following two terms:
\begin{align}\label{eq:hajak-2}
n^{-2} \sum_{i=1}^n \sum_{j=1,j \neq k}^n \ind(Y_j \le Y_i \wedge Y_{N_1(i)}) ~~~{\rm and}~~~n^{-2} \sum_{i=1}^n \ind(Y_k \le Y_i \wedge Y_{N_1(i)}).
\end{align}
For the first term, since $j \neq k$, $(X_j, Y_j)$ is independent of $(X_k,Y_k)$ and hence 
\[
\E \Big[n^{-2} \sum_{i=1}^n \sum_{j=1,j \neq k}^n \ind(Y_j \le Y_i \wedge Y_{N_1(i)}) \given X_k,Y_k\Big] \approx \E \Big[ n^{-1} \sum_{i=1}^n F_Y(Y_i \wedge Y_{N_1(i)}) \given X_k,Y_k \Big],
\]
which corresponds exactly to the ``natural component of the H\'ajek representation''  \eqref{eq:hajak-1} when projected to $(X_k,Y_k)$.

What about the second term in \eqref{eq:hajak-2}? Notice that when the sample size is sufficiently large, the NN distance is small, and hence for any $k\ne 1$,
\begin{align*}
\E \Big[n^{-2} \sum_{i=1}^n \ind(Y_k \le Y_i \wedge Y_{N_1(i)}) \given X_k,Y_k\Big] &\approx n^{-1} \E\Big[\ind(Y_k \le Y_1 \wedge Y_{N_1(1)}) \given X_k,Y_k\Big]\\
& \approx n^{-1} \E\Big[\ind(Y_k \le Y_1 \wedge \tY_1) \given X_k,Y_k\Big],
\end{align*}
where $\tY_1$ is sampled independently from the conditional distribution of $Y$ given $X_1$. By the definition of the function $h(\cdot)$ in \eqref{eq:h}, 
\[
\E[\ind(Y_k \le Y_1 \wedge \tY_1) \given X_k,Y_k] =  h(Y_k). 
\]
Then using the H\'ajek projection, the difference between 
\[
n^{-2} \sum_{i=1}^n [R_i \wedge R_{N_1(i)}] ~~ {\rm and} ~~
n^{-1} \sum_{i=1}^n [F_Y(Y_i) \wedge F_Y(Y_{N_1(i)})] 
\]
after projection into sums is $n^{-1} \sum_{k=1}^n h(Y_k)$ up to a constant. This gives rise to \eqref{eq:xin*}.

In detail, we have the following theorem.

\begin{theorem}[H\'ajek representation]\label{thm:hayek} It holds true (for any fixed continuous $F_{X,Y}$) that
  \begin{align*}
    \lim_{n \to \infty} \Big\{n\Var[\xi_n - \xi_n^*]\Big\} = 0 ~~~{\rm and}~~~     \lim_{n \to \infty} \Big\{n\Var[\overline{\xi}_n - \overline{\xi}_n^*]\Big\} = 0 ~~({\rm if}~d=1).
  \end{align*}
\end{theorem}

Using Theorem~\ref{thm:hayek}, as long as $n\liminf_{n \to \infty} \Var[\xi_n] > 0$, normalized $\xi_n$ ($\overline\xi_n$) and $\xi_n^*$ ($\overline\xi_n^*$) share the same asymptotic distribution and it suffices to establish the CLT for $\xi_n^*$ ($\overline\xi_n^*$). In the second step, we establish the CLT of $\xi_n^*$ and $\overline\xi_n^*$ by noticing that it merely consists of a linear sum of nearest neighbor statistics. Leveraging the normal approximation theorem under local dependence \citep{MR2435859}, one can then reach the following two CLTs. 

\begin{theorem}\label{thm:clt}
  As long as $Y$ is not a measurable function of $X$ almost surely, it holds true (for any fixed continuous $F_{X,Y}$) that
  \begin{align}\label{eq:new-new-1}
    \big(\xi_n^* - \E[\xi_n^*]\big)/\sqrt{\Var[\xi_n^*]} \longrightarrow N(0,1) ~~ {\rm in~distribution},
  \end{align}
  and
  \begin{align*}
    \big(\overline{\xi}_n^* - \E[\overline{\xi}_n^*]\big)/\sqrt{\Var[\overline{\xi}_n^*]} \longrightarrow N(0,1) ~~ {\rm in~distribution}.
  \end{align*}
\end{theorem}

\begin{remark}\label{remark:sobol}
Of note, in conducting global sensitivity analysis via the first-order Sobol indices, \citet[Theorem 4.1]{gamboa2022global} obtained a CLT similar to \eqref{eq:new-new-1} above. \nb{In another related work, \citet{devroye2018nearest} introduced and analyzed a nearest neighbor statistic for estimating the residual variance in nonparametric regression, and also established its central limit theorem.  All these results,} however, do not have to handle the randomness from ranking $Y_i$'s that we addressed in Theorem \ref{thm:hayek} and is to us the most difficult part.
\end{remark}

Finally, Theorem~\ref{thm:main} is proved by combining Theorems~\ref{thm:hayek} and \ref{thm:clt}.

\begin{remark}\label{eq:hajek}
The H\'ajek representation of $\xi_n$ under independence between $X$ and $Y$ was established in, e.g.,  \citet[Lemma D.1]{deb2020kernel}, \citet[Equ. (4.9)]{cao2020correlations}, \citet[Lemma 7.1]{shi2021ac}, and \citet[Remark 3.2]{lin2021boosting}. See also \citet[Theorem 2.1]{auddy2021exact}. The remaining component there is a U-statistic of the form 
\begin{align}\label{eq:hajak-ind}
-\frac{1}{n(n-1)} \sum_{i \neq j} F_Y(Y_i \wedge Y_j). 
\end{align}
Using standard U-statistic theory \citep[Theorem 12.3]{MR1652247}, the main term of \eqref{eq:hajak-ind} is 
\begin{align}\label{eq:hajak-3}
-n^{-1} \sum_{i=1}^n \Big(2F_Y(Y_i) - F_Y^2(Y_i) - \frac{1}{3} \Big). 
\end{align}
Noticing that $\E[G_X(\cdot)] = 1-F_Y(\cdot)$, we have
\[
h(\cdot) =  \Var [G_X^2(\cdot)] + (\E[G_X(\cdot)])^2 = \Var [G_X^2(\cdot)] - (2F_Y(\cdot) - F_Y^2(\cdot)) + 1. 
\]
Under the null, one is then ready to check $\Var [G_X^2(\cdot)] = 0$, and thus $h(\cdot)$ reduces to \eqref{eq:hajak-3} (up to some constants).
\end{remark}


\section{Proof of the main results}\label{sec:main-proof}

\paragraph*{Notation.}
For any integers $n,d\ge 1$, let $\zahl{n}:= \{1,2,\ldots,n\}$, and $\bR^d$ be the $d$-dimensional real space. A set consisting of distinct elements $x_1,\dots,x_n$ is written as either $\{x_1,\dots,x_n\}$ or $\{x_i\}_{i=1}^{n}$, and its cardinality is written by $\lvert \{x_i\}_{i=1}^n \rvert$. The corresponding sequence is denoted by $[x_1,\dots,x_n]$ or $[x_i]_{i=1}^{n}$. 
For any two real sequences $\{a_n\}$ and $\{b_n\}$, write $a_n \lesssim b_n$ (or equivalently, $b_n \gtrsim a_n$) if there exists a universal constant $C>0$ such that $a_n/b_n \le C$ for all sufficiently large $n$, and write $a_n \prec b_n$ (or equivalently, $b_n \succ a_n$) if $a_n/b_n \to 0$ as $n$ goes to infinity. Write $a_n = O(b_n)$ if $\lvert a_n \rvert \lesssim b_n$ and $a_n = o(b_n)$ if $\lvert a_n \rvert \prec b_n$. We shorthand $(X_1,\ldots,X_n)$ by $\mX$. We use  $\stackrel{\sf d}{\longrightarrow}$ and $\stackrel{\sf p}{\longrightarrow}$ to denote convergences in distribution and in probability, respectively.




\begin{proof}[Proof of Theorem~\ref{thm:main}]
From Proposition~\ref{prop:variance} and Theorem~\ref{thm:hayek},
\begin{align*}
  \limsup_{n \to \infty} \E\Big[\frac{\xi_n^* - \E[\xi_n^*]}{\sqrt{\Var[\xi_n]}} - \frac{\xi_n - \E[\xi_n]}{\sqrt{\Var[\xi_n]}}\Big]^2 &= \limsup_{n \to \infty} \frac{\Var[\xi_n - \xi_n^*]}{\Var[\xi_n]} \\
  &\le \frac{\limsup_{n \to \infty} n\Var[\xi_n - \xi_n^*]}{\liminf_{n \to \infty} n \Var[\xi_n]} = 0,
\end{align*}
and
\begin{align*}
  \limsup_{n \to \infty} \Big\lvert \frac{\Cov[\xi_n,\xi_n - \xi_n^*]}{\Var[\xi_n]} \Big\rvert &\le \limsup_{n \to \infty} \Big(\frac{\Var[\xi_n - \xi_n^*]}{\Var[\xi_n]}\Big)^{\frac{1}{2}} \\
  &\le \Big(\frac{\limsup_{n \to \infty} n\Var[\xi_n - \xi_n^*]}{\liminf_{n \to \infty} n\Var[\xi_n]}\Big)^{\frac{1}{2}} = 0.
\end{align*}
One can then deduce
\[
  \frac{\xi_n^* - \E[\xi_n^*]}{\sqrt{\Var[\xi_n]}} - \frac{\xi_n - \E[\xi_n]}{\sqrt{\Var[\xi_n]}} \stackrel{\sf p}{\longrightarrow} 0~~~{\rm and}~~~\Var[\xi_n^*]/\Var[\xi_n] \longrightarrow 1.
\]
We then complete the proof for $\xi_n$ by using Theorem~\ref{thm:clt}. The proof for $\overline{\xi}_n$ can be established in the same way.
\end{proof}

For better readability, we defer the proof of Theorem~\ref{thm:var} to the end of this section.


\begin{proof}[Proof of Theorem~\ref{thm:hayek}]

We first introduce some necessary notation for the proof.

For any $t \in \bR$, recall $G_X(t)= \P\big(Y \ge t \given X\big)$ and define
\begin{align}
  G(t):= \P\big(Y \ge t\big) = 1-F_Y(t), ~~~g(t):= \Var\Big[G_X(t)\Big] = \E \Big[G_X^2(t)\Big] - G^2(t).
\end{align}
For any $x \in \bR^d$, define
\begin{align}
  h_0(x): = \E[h(Y) \given X=x] = \int \E[G_X^2(t)] \d F_{Y \given X=x}(t),
\end{align}
where $F_{Y \given X=x}$ is the conditional distribution of $Y$ conditional on $X=x$.

We then introduce an intermediate statistic $\widecheck{\xi}_n$ as follows,
\begin{align*}
  \widecheck{\xi}_n := & \frac{6n}{n^2-1} \Big(\sum_{i=1}^n \min\big\{F_Y(Y_i), F_Y(Y_{N_1(i)})\big\} - \frac{1}{n-1} \sum_{\substack{i,j=1\\i \neq j}}^n \min\big\{F_Y(Y_i), F_Y(Y_j)\big\} \yestag\label{eq:checkxin}\\
  &+ \sum_{i=1}^n g(Y_i) + \frac{1}{n-1} \sum_{\substack{i,j=1\\i \neq j}}^n\E \Big[ \min\big\{F_Y(Y_i), F_Y(Y_j)\big\} \Biggiven X_i,X_j \Big] \\
  &- \sum_{i=1}^n \E \Big[ g(Y_i) \Biggiven X_i \Big] +  \sum_{i=1}^n h_0(X_i) \Big).
\end{align*}

Notice that
\begin{align*}
  \Var[\xi_n - \xi_n^*] =& \Var[\xi_n - \widecheck{\xi}_n] + \Var[\widecheck{\xi}_n - \xi_n^*] + 2\Cov[\xi_n - \widecheck{\xi}_n, \widecheck{\xi}_n - \xi_n^*]\\
  \le & \Var[\xi_n - \widecheck{\xi}_n] + \Var[\widecheck{\xi}_n - \xi_n^*] + 2(\Var[\xi_n - \widecheck{\xi}_n]\Var[\widecheck{\xi}_n - \xi_n^*])^{1/2}.
\end{align*}
As long as
\begin{align}\label{eq:han-1}
  \lim_{n \to \infty} n\Var[\xi_n -\widecheck{\xi}_n] = 0 ~~{\rm and}~~ \lim_{n \to \infty} n\Var[\widecheck{\xi}_n - \xi_n^*] = 0,
\end{align}
the proof for $\xi_n$ is complete. The proof for $\overline{\xi}_n$ is similar and accordingly omitted.

For the first equation in \eqref{eq:han-1}, by the law of total variance, one can decompose $\Var[\xi_n - \widecheck{\xi}_n]$ as follows,
\[
  n \Var[\xi_n - \widecheck{\xi}_n] = n \E [\Var[\xi_n - \widecheck{\xi}_n \given \mX]] + n \Var [\E[\xi_n - \widecheck{\xi}_n \given \mX]].
\]


\vspace{0.5cm}

{\bf Step I.} $\lim_{n \to \infty} n\E [\Var[\xi_n - \widecheck{\xi}_n \given \mX]] = 0$.

We decompose $n\E [\Var[\xi_n - \widecheck{\xi}_n \given \mX]]$ as:
\begin{align}\label{eq:hayekmain1}
  n \E [\Var[\xi_n - \widecheck{\xi}_n \given \mX]] = n \E [\Var[\xi_n \given \mX]] + n \E [\Var[\widecheck{\xi}_n \given \mX]] - 2 n \E [\Cov[\xi_n , \widecheck{\xi}_n \given \mX]].
\end{align}
For the first term in \eqref{eq:hayekmain1}, using \eqref{eq:xin}, we have
\begin{align*}
   &n \Var[\xi_n \given \mX] \yestag\label{eq:hayekmain11}\\
 =& \frac{36n}{(n^2-1)^2} \Var \Big[  \sum_{i=1}^n \min\big\{R_i, R_{N_1(i)}\big\} \Biggiven \mX \Big]  \\
  = & \frac{36n^4}{(n^2-1)^2} \Big\{ \frac{1}{n^3} \sum_{i=1}^n \Var \Big[ \min\big\{R_i, R_{N_1(i)}\big\} \Biggiven \mX \Big] \\
  & + \frac{1}{n^3} \sum_{\substack{j=N_1(i),i \neq N_1(j)\\{\rm or}~i=N_1(j),j \neq N_1(i)}} \Cov\Big[\min\big\{R_i, R_{N_1(i)}\big\}, \min\big\{R_j, R_{N_1(j)}\big\} \Biggiven \mX \Big] \\
  & + \frac{1}{n^3} \sum_{\substack{i \neq j \\ N_1(i) = N_1(j)}} \Cov\Big[\min\big\{R_i, R_{N_1(i)}\big\}, \min\big\{R_j, R_{N_1(j)}\big\} \Biggiven \mX \Big] \\
  & + \frac{1}{n^3} \sum_{j = N_1(i), i = N_1(j)} \Cov\Big[\min\big\{R_i, R_{N_1(i)}\big\}, \min\big\{R_j, R_{N_1(j)}\big\} \Biggiven \mX \Big]\\
  & + \frac{1}{n^3} \sum_{i,j,N_1(i),N_1(j)~{\rm distinct}} \Cov\Big[\min\big\{R_i, R_{N_1(i)}\big\}, \min\big\{R_j, R_{N_1(j)}\big\} \Biggiven \mX \Big] \Big\}\\
  =: & \frac{36n^4}{(n^2-1)^2} \Big(T_1 + T_2 + T_3 + T_4 + T_5\Big).
\end{align*}

For the second term in \eqref{eq:hayekmain1}, noticing that the last three terms in \eqref{eq:checkxin} are constants conditional on $\mX$, we have
\begin{align*}
  & n \Var[\widecheck{\xi}_n \given \mX]   \yestag\label{eq:hayekmain12}\\
  = & \frac{36n^3}{(n^2-1)^2} \Var \Big[  \sum_{i=1}^n \min\big\{F_Y(Y_i), F_Y(Y_{N_1(i)})\big\} - \frac{1}{n-1} \sum_{\substack{i,j=1\\i \neq j}}^n \min\big\{F_Y(Y_i), F_Y(Y_j)\big\} \\
  &\quad\quad+ \sum_{i=1}^n g(Y_i) \Biggiven \mX \Big]\\
  = & \frac{36n^4}{(n^2-1)^2} \Big\{ \frac{1}{n} \sum_{i=1}^n \Var \Big[ \min\big\{F_Y(Y_i), F_Y(Y_{N_1(i)})\big\} \Biggiven \mX \Big] \\
  & + \frac{1}{n} \sum_{\substack{j=N_1(i),i \neq N_1(j)\\{\rm or}~i=N_1(j),j \neq N_1(i)}} \Cov\Big[\min\big\{F_Y(Y_i), F_Y(Y_{N_1(i)})\big\}, \min\big\{F_Y(Y_j), F_Y(Y_{N_1(j)})\big\} \Biggiven \mX \Big] \\
  & + \frac{1}{n} \sum_{\substack{i \neq j \\ N_1(i) = N_1(j)}} \Cov\Big[\min\big\{F_Y(Y_i), F_Y(Y_{N_1(i)})\big\}, \min\big\{F_Y(Y_j), F_Y(Y_{N_1(j)})\big\} \Biggiven \mX \Big]\\
  & + \frac{1}{n} \sum_{j = N_1(i), i = N_1(j)} \Cov\Big[\min\big\{F_Y(Y_i), F_Y(Y_{N_1(i)})\big\}, \min\big\{F_Y(Y_j), F_Y(Y_{N_1(j)})\big\}\Biggiven \mX \Big]\\
  & + \frac{1}{n} \sum_{i,j,N_1(i),N_1(j)~{\rm distinct}} \Cov\Big[\min\big\{F_Y(Y_i), F_Y(Y_{N_1(i)})\big\}, \min\big\{F_Y(Y_j), F_Y(Y_{N_1(j)})\big\} \Biggiven \mX \Big] \\
  & - 2 \frac{1}{n(n-1)} \sum_{i=1}^n \Cov\Big[ \min\big\{F_Y(Y_i), F_Y(Y_{N_1(i)})\big\}, \sum_{\substack{i,j=1\\i \neq j}}^n \min\big\{F_Y(Y_i), F_Y(Y_j)\big\} \Biggiven \mX \Big]\\
  & + \frac{1}{n(n-1)^2} \Var\Big[\sum_{\substack{i,j=1\\i \neq j}}^n \min\big\{F_Y(Y_i), F_Y(Y_j)\big\} \Biggiven \mX \Big]\\
  & + 2 \frac{1}{n} \sum_{i=1}^n \Cov\Big[ \min\big\{F_Y(Y_i), F_Y(Y_{N_1(i)})\big\}, \sum_{i=1}^n g(Y_i) \Biggiven \mX \Big]\\
  & - 2 \frac{1}{n(n-1)} \Cov\Big[ \sum_{\substack{i,j=1\\i \neq j}}^n \min\big\{F_Y(Y_i), F_Y(Y_j)\big\} , \sum_{i=1}^n g(Y_i) \Biggiven \mX \Big] \\
  & + \frac{1}{n} \Var\Big[\sum_{i=1}^n g(Y_i) \Biggiven \mX \Big]\Big\}\\
  =: & \frac{36n^4}{(n^2-1)^2} \Big(T_1^* + T_2^* + T_3^* + T_4^* + T_5^* - 2 T_6^* + T_7^* + 2T_8^* - 2T_9^* + T_{10}^*\Big).
\end{align*}

For the third term in \eqref{eq:hayekmain1}, from \eqref{eq:xin} and \eqref{eq:checkxin}, we have
\begin{align*}
  & n \Cov[\xi_n , \widecheck{\xi}_n \given \mX] \yestag\label{eq:hayekmain13}\\
  =& \frac{36n^2}{(n^2-1)^2} \Cov \Big[ \sum_{i=1}^n \min\big\{R_i, R_{N_1(i)}\big\}, \sum_{i=1}^n \min\big\{F_Y(Y_i), F_Y(Y_{N_1(i)})\big\} \\
  & - \frac{1}{n-1} \sum_{\substack{i,j=1\\i \neq j}}^n \min\big\{F_Y(Y_i), F_Y(Y_j)\big\} + \sum_{i=1}^n g(Y_i) \Biggiven \mX \Big]\\
  = & \frac{36n^4}{(n^2-1)^2} \Big\{ \frac{1}{n^2} \sum_{i=1}^n \Cov \Big[ \min\big\{R_i, R_{N_1(i)}\big\}, \min\big\{F_Y(Y_i), F_Y(Y_{N_1(i)})\big\} \Biggiven \mX \Big] \\
  & + \frac{1}{n^2} \sum_{\substack{j=N_1(i),i \neq N_1(j)\\{\rm or}~i=N_1(j),j \neq N_1(i)}} \Cov\Big[ \min\big\{R_i, R_{N_1(i)}\big\}, \min\big\{F_Y(Y_j), F_Y(Y_{N_1(j)})\big\} \Biggiven \mX \Big] \\
  & + \frac{1}{n^2} \sum_{\substack{i \neq j \\ N_1(i) = N_1(j)}} \Cov\Big[ \min\big\{R_i, R_{N_1(i)}\big\}, \min\big\{F_Y(Y_j), F_Y(Y_{N_1(j)})\big\} \Biggiven \mX \Big]\\
  & + \frac{1}{n^2} \sum_{j = N_1(i), i = N_1(j)} \Cov\Big[\min\big\{R_i, R_{N_1(i)}\big\}, \min\big\{F_Y(Y_j), F_Y(Y_{N_1(j)})\big\}\Biggiven \mX \Big]\\
  & + \frac{1}{n^2} \sum_{i,j,N_1(i),N_1(j)~{\rm distinct}} \Cov\Big[ \min\big\{R_i, R_{N_1(i)}\big\}, \min\big\{F_Y(Y_j), F_Y(Y_{N_1(j)})\big\} \Biggiven \mX \Big] \\
  & - \frac{1}{n^2(n-1)} \sum_{i=1}^n \Cov\Big[ \min\big\{R_i, R_{N_1(i)}\big\}, \sum_{\substack{i,j=1\\i \neq j}}^n \min\big\{F_Y(Y_i), F_Y(Y_j)\big\} \Biggiven \mX \Big]\\
  & + \frac{1}{n^2} \sum_{i=1}^n \Cov\Big[ \min\big\{R_i, R_{N_1(i)}\big\}, \sum_{i=1}^n g(Y_i) \Biggiven \mX \Big] \Big\}\\
  =: & \frac{36n^4}{(n^2-1)^2} \Big(T_1' + T_2' + T_3' + T_4' + T_5' - T_6' + T_7'\Big).
\end{align*}

Let $Y, \tY \sim F_Y, \tY_1, \tY_1' \sim F_{Y \given X = X_1}, \tY_2 \sim F_{Y \given X = X_2}$ be mutually independently drawn. We then establish the following five lemmas that control the terms of \eqref{eq:hayekmain11}-\eqref{eq:hayekmain13}.

\begin{lemma}\label{lemma:hayek1}
  For $i=1,2,3,4$,
  \[
    \lim_{n \to \infty} \Big\lvert \E[T_i] - \E[T_i^*] \Big\rvert =0, ~~ \lim_{n \to \infty} \Big\lvert \E[T_i'] - \E[T_i^*] \Big\rvert =0,
  \]
  and
  \begin{align*}
    & \lim_{n \to \infty} \Big\lvert \E\Big[T_1\Big] - \E \Big[ \Var \Big[ F_Y\big(Y_1 \wedge \tY_1\big) \Biggiven X_1 \Big] \Big] \Big\rvert = 0,\\
    & \lim_{n \to \infty} \Big\lvert \E\Big[T_2\Big] - 2\E \Big[\Cov\Big[F_Y\big(Y_1 \wedge \tY_1\big), F_Y\big(\tY_1 \wedge \tY_1' \big) \Biggiven X_1 \Big] \ind \Big( 1 \neq N_1(N_1(1)) \Big) \Big] \Big\rvert = 0,\\
    & \lim_{n \to \infty} \Big\lvert \E\Big[T_3\Big] - \E \Big[\Cov\Big[F_Y\big(Y_1 \wedge \tY_1\big), F_Y\big(\tY_1 \wedge \tY_1' \big) \Biggiven X_1 \Big] \Big\lvert \Big\{j: j \neq 1, N_1(j) = N_1(1)\Big\} \Big\rvert \Big] \Big\rvert = 0,\\
    & \lim_{n \to \infty} \Big\lvert \E\Big[T_4\Big] - \E \Big[ \Var \Big[ F_Y\big(Y_1 \wedge \tY_1\big) \Biggiven X_1 \Big] \ind \Big( 1 = N_1(N_1(1)) \Big) \Big] \Big\rvert = 0.
  \end{align*}
\end{lemma}

\begin{lemma}\label{lemma:hayek2} 
  \begin{align*}
    \lim_{n \to \infty} [\E[T_5] - 2\E[T_5']] = \E \Big[\Cov\Big[\ind\big(Y_3 \le Y_1 \wedge \tY_1\big), \ind\big(Y_3 \le Y_2 \wedge \tY_2\big) \Biggiven X_1,X_2,X_3 \Big] \Big] =: a_1,\\
    \E[T_5^*]=0, ~~{\rm and}~~\lim_{n \to \infty} \Big\lvert \E\Big[T_5'\Big] - 2\E \Big[\Cov\Big[\ind\big(Y_2 \le Y_1 \wedge \tY_1\big), F_Y\big(Y_2 \wedge \tY_2\big) \Biggiven X_1,X_2 \Big] \Big] \Big\rvert = 0.
  \end{align*}
\end{lemma}

\begin{lemma}\label{lemma:hayek3}
  \[
    \lim_{n \to \infty} [\E[T_6'] - \E[T_6^*]] = 2\E \Big[\Cov\Big[\ind\big(Y_2 \le Y_1 \wedge \tY_1\big), F_Y\big(Y_2 \wedge 
    Y \big) \Biggiven X_1,X_2 \Big] \Big] =: 2a_2.
  \]
\end{lemma}

\begin{lemma}\label{lemma:hayek4} 
  \[
    \lim_{n \to \infty} \E[T_7^*] = 4 \E \Big[\Cov\Big[F_Y\big(Y_1 \wedge Y\big), F_Y\big(Y_1 \wedge \tY\big) \Biggiven X_1 \Big] \Big] =: 4a_3.
  \]
\end{lemma}

\begin{lemma}\label{lemma:hayek8}
  \begin{align*}
     \lim_{n \to \infty} \E[T_8^*] &= 2\E \Big[\Cov\Big[F_Y\big(Y_1 \wedge \tY_1\big), g(Y_1) \Biggiven X_1\Big] \Big] =: 2b_1,\\
     \lim_{n \to \infty} \E[T_9^*] &= 2\E \Big[\Cov\Big[F_Y\big(Y_1 \wedge Y \big), g(Y_1) \Biggiven X_1\Big] \Big] =: 2b_2,\\
     \lim_{n \to \infty} \E[T_7'] &= \E \Big[\Cov\Big[\ind\big(Y_2 \le Y_1 \wedge \tY_1\big), g(Y_2) \Biggiven X_1,X_2 \Big] \Big]  \\
    &+ 2\E \Big[\Cov\Big[F_Y\big(Y_1 \wedge \tY_1\big), g(Y_1) \Biggiven X_1\Big] \Big]=: b_3,\\
     \lim_{n \to \infty} \E[T_{10}^*] &= \E \Big[\Var\Big[g(Y_1) \Biggiven X_1\Big]\Big].
  \end{align*}
\end{lemma}

\vspace{0.5cm}

Plugging \eqref{eq:hayekmain11}-\eqref{eq:hayekmain13} to \eqref{eq:hayekmain1} and using Lemmas~\ref{lemma:hayek1}-\ref{lemma:hayek8}, one obtains
\begin{align}\label{eq:hayekmain14}
  &\lim_{n \to \infty} n \E [\Var[\xi_n - \widecheck{\xi}_n \given \mX]] \\
  =& 36 \lim_{n \to \infty} \E \Big[\sum_{i=1}^4 \Big(T_i + T_i^* - 2T_i'\Big) + \Big(T_5 + T_5^* - 2T_5'\Big) - 2\Big(T_6^* - T_6'\Big) \notag\\
  &~~~~~~+ T_7^* + 2T_8^* - 2T_9^* - 2T_7' + T_{10}^* \Big]\notag\\
  = & 36 \Big(a_1 + 4a_2 + 4a_3 - 2\Big(b_3 - 2b_1 + 2b_2 \Big) + \E \Big[\Var\Big[g(Y_1) \Biggiven X_1\Big] \Big] \Big).\notag
\end{align}

For the relationship of $a_1,a_2,a_3$ and $b_1,b_2,b_3$, we establish the following identity.

\begin{lemma}[A key identity]\label{lemma:hayek9} We have
  \[
    a_1 + 4a_2 + 4a_3 = b_3 - 2b_1 + 2b_2 = \E \Big[\Var\Big[g(Y_1) \Biggiven X_1\Big] \Big].
  \]
\end{lemma}

Combining Lemma~\ref{lemma:hayek9} with \eqref{eq:hayekmain14} proves
\begin{align}\label{eq:hayekmain15}
  \lim_{n \to \infty} n \E [\Var[\xi_n - \widecheck{\xi}_n \given \mX]] = 0.
\end{align}

{\bf Step II.} $\lim_{n \to \infty}  n \Var [\E[\xi_n - \widecheck{\xi}_n \given \mX]] = 0$.

Checking \eqref{eq:xin}, one has
\begin{align*}
  \E[\xi_n \given \mX] &= \E \Big[\frac{6}{n^2-1} \sum_{i=1}^n \sum_{k=1}^n \ind\big(Y_k \le Y_i \wedge Y_{N_1(i)}\big) - \frac{2n+1}{n-1} \Biggiven \mX \Big]\\
  & = \frac{6}{n^2-1} \sum_{i=1}^n \sum_{k=1}^n \E \Big[\ind\big(Y_k \le Y_i \wedge Y_{N_1(i)}\big) \Biggiven \mX \Big] - \frac{2n+1}{n-1}.
\end{align*}
Checking \eqref{eq:checkxin}, one has
\[
  \E[\widecheck{\xi}_n \given \mX] = \frac{6n}{n^2-1} \Big(\sum_{i=1}^n \E \Big[ F_Y\big(Y_i \wedge Y_{N_1(i)}\big) \Biggiven \mX \Big] + \sum_{i=1}^n h_0(X_i) \Big).
\]
Consequently, we obtain
\begin{align}\label{eq:hayekmain2}
   \Var [\E[\xi_n - \widecheck{\xi}_n \given \mX]]=&  \frac{36n^2}{(n^2-1)^2} \cdot \Var \Big[\frac{1}{n} \sum_{i=1}^n \sum_{k=1}^n \E \Big[\ind\big(Y_k \le Y_i \wedge Y_{N_1(i)}\big) \Biggiven \mX \Big]\\
   &\quad\quad  - \sum_{i=1}^n \E \Big[ F_Y\big(Y_i \wedge Y_{N_1(i)}\big) \Biggiven \mX \Big] - \sum_{i=1}^n h_0(X_i) \Big].\notag
\end{align}

To apply the Efron-Stein inequality (Theorem 3.1 in \cite{boucheron2013concentration}), recall $\mX = (X_1,\ldots,X_n)$ and define, for any $\ell \in \zahl{n}$,
\[
  \mX_\ell := (X_1,\ldots,X_{\ell-1},\tX_\ell,X_{\ell+1},\ldots,X_n),
\]
where $[\tX_\ell]_{\ell=1}^n$ are independent copies of $[X_\ell]_{\ell=1}^n$.

We fix one $\ell \in \zahl{n}$. For any $i \in \zahl{n}$, let $\tN_1(i)$ be the index of the NN of $i$ in $\mX_\ell$.

For the first term in \eqref{eq:hayekmain2}, we first decompose it as
\begin{align*}
  &\sum_{i=1}^n \sum_{k=1}^n \E \Big[\ind\big(Y_k \le Y_i \wedge Y_{N_1(i)}\big) \Biggiven \mX \Big] \\
  = &\sum_{i=1}^n \E \Big[\ind\big(Y_\ell \le Y_i \wedge Y_{N_1(i)}\big) \Biggiven \mX \Big] + \sum_{i=1}^n \sum_{k=1,k \neq \ell}^n \E \Big[\ind\big(Y_k \le Y_i \wedge Y_{N_1(i)}\big) \Biggiven \mX \Big].
\end{align*}
Notice that $\E \Big[\ind\big(Y_k \le Y_i \wedge Y_{N_1(i)}\big) \Biggiven \mX \Big]$ only depends on $X_k,X_i,X_{N_1(i)}$. Then for any $i \in \zahl{n}$ such that $i \neq \ell, N_1(i) \neq \ell, \tN_1(i) \neq \ell$, we have $N_1(i)=\tN_1(i)$, and then
\[
  \sum_{k=1,k \neq \ell}^n \E \Big[\ind\big(Y_k \le Y_i \wedge Y_{N_1(i)}\big) \Biggiven \mX \Big] - \sum_{k=1,k \neq \ell}^n \E \Big[\ind\big(Y_k \le Y_i \wedge Y_{\tN_1(i)}\big) \Biggiven \mX_\ell \Big] = 0.
\]
One then has
\begin{align}\label{eq:hayekmain21}
  & \sum_{i=1}^n \sum_{k=1}^n \E \Big[\ind\big(Y_k \le Y_i \wedge Y_{N_1(i)}\big) \Biggiven \mX \Big] - \sum_{i=1}^n \sum_{k=1}^n \E \Big[\ind\big(Y_k \le Y_i \wedge Y_{\tN_1(i)}\big) \Biggiven \mX_\ell \Big]\\
  = & \sum_{i=1}^n \E \Big[\ind\big(Y_\ell \le Y_i \wedge Y_{N_1(i)}\big) \Biggiven \mX \Big] - \sum_{i=1}^n \E \Big[\ind\big(Y_\ell \le Y_i \wedge Y_{\tN_1(i)}\big) \Biggiven \mX_\ell \Big]\notag\\
  & + \sum_{k=1,k \neq \ell}^n \E \Big[\ind\big(Y_k \le Y_\ell \wedge Y_{N_1(\ell)}\big) \Biggiven \mX \Big] - \sum_{k=1,k \neq \ell}^n \E \Big[\ind\big(Y_k \le Y_\ell \wedge Y_{\tN_1(\ell)}\big) \Biggiven \mX_\ell \Big]\notag\\
  & + \!\!\!\!\!\!\!\!\!\!\!\!\!\!\!\!\sum_{\substack{i=1\\N_1(i)=\ell~{\rm or}~\tN_1(i)=\ell}}^n \Big[\sum_{k=1,k \neq \ell}^n \E \Big[\ind\big(Y_k \le Y_i \wedge Y_{N_1(i)}\big) \Biggiven \mX \Big] \!-\! \sum_{k=1,k \neq \ell}^n \E \Big[\ind\big(Y_k \le Y_i \wedge Y_{\tN_1(i)}\big) \Biggiven \mX_\ell \Big] \Big].\notag
\end{align}

For the second term in \eqref{eq:hayekmain2}, noticing that $\E \Big[ F_Y\big(Y_i \wedge Y_{N_1(i)}\big) \Biggiven \mX \Big]$ only depends on $X_i,X_{N_1(i)}$, we have
\begin{align}\label{eq:hayekmain22}
  & \sum_{i=1}^n \E \Big[ F_Y\big(Y_i \wedge Y_{N_1(i)}\big) \Biggiven \mX \Big] -  \sum_{i=1}^n \E \Big[ F_Y\big(Y_i \wedge Y_{\tN_1(i)}\big) \Biggiven \mX_\ell \Big]\\
  = & \E \Big[ F_Y\big(Y_\ell \wedge Y_{N_1(\ell)}\big) \Biggiven \mX \Big] - \E \Big[ F_Y\big(Y_\ell \wedge Y_{\tN_1(\ell)}\big) \Biggiven \mX_\ell \Big]\notag\\
  & + \sum_{\substack{i=1\\N_1(i)=\ell~{\rm or}~\tN_1(i)=\ell}}^n \Big[ \E \Big[ F_Y\big(Y_i \wedge Y_{N_1(i)}\big) \Biggiven \mX \Big] - \E \Big[ F_Y\big(Y_i \wedge Y_{\tN_1(i)}\big) \Biggiven \mX_\ell \Big] \Big].\notag
\end{align}

For the third term in \eqref{eq:hayekmain2}, we have
\begin{align*}
  \sum_{i=1}^n h_0(X_i) - \sum_{i=1, i\neq \ell}^n h_0(X_i) - h_0(\tX_\ell) = h_0(X_\ell) - h_0(\tX_\ell).
  \yestag\label{eq:hayekmain23}
\end{align*}

Plugging \eqref{eq:hayekmain21}-\eqref{eq:hayekmain23} to \eqref{eq:hayekmain2} and using the Efron-Stein inequality then yields
\begin{align}\label{eq:hayekmain24}
  &n \Var [\E[\xi_n - \widecheck{\xi}_n \given \mX]] \\
  \le &\frac{18n^3}{(n^2-1)^2} \sum_{\ell=1}^n \E\Big\{\frac{1}{n} \sum_{i=1}^n \E \Big[\ind\big(Y_\ell \le Y_i \wedge Y_{N_1(i)}\big) \Biggiven \mX \Big] - h_0(X_\ell)\notag \\
  & + \frac{1}{n} \sum_{k=1,k \neq \ell}^n \E \Big[\ind\big(Y_k \le Y_\ell \wedge Y_{N_1(\ell)}\big) \Biggiven \mX \Big] - \E \Big[ F_Y\big(Y_\ell \wedge Y_{N_1(\ell)}\big) \Biggiven \mX \Big]\notag\\
  & + \frac{1}{n} \!\!\!\!\!\!\!\!\sum_{\substack{i=1\\N_1(i)=\ell~{\rm or}~\tN_1(i)=\ell}}^n \sum_{k=1,k \neq \ell}^n \E \Big[\ind\big(Y_k \le Y_i \wedge Y_{N_1(i)}\big) \Biggiven \mX \Big] - \!\!\!\!\!\!\!\!\!\!\!\!\!\!\!\! \sum_{\substack{i=1\\N_1(i)=\ell~{\rm or}~\tN_1(i)=\ell}}^n \!\!\!\!\!\!\!\!\E \Big[ F_Y\big(Y_i \wedge Y_{N_1(i)}\big) \Biggiven \mX \Big] \notag\\
  & - \frac{1}{n} \sum_{i=1}^n \E \Big[\ind\big(Y_\ell \le Y_i \wedge Y_{\tN_1(i)}\big) \Biggiven \mX_\ell \Big] + h_0(\tX_\ell) \notag\\
  & - \frac{1}{n} \sum_{k=1,k \neq \ell}^n \E \Big[\ind\big(Y_k \le Y_\ell \wedge Y_{\tN_1(\ell)}\big) \Biggiven \mX_\ell \Big] + \E \Big[ F_Y\big(Y_\ell \wedge Y_{\tN_1(\ell)}\big) \Biggiven \mX_\ell \Big]\notag\\
  & - \frac{1}{n}\!\!\!\!\!\!\!\! \sum_{\substack{i=1\\N_1(i)=\ell~{\rm or}~\tN_1(i)=\ell}}^n \sum_{k=1,k \neq \ell}^n \E \Big[\ind\big(Y_k \le Y_i \wedge Y_{\tN_1(i)}\big) \Biggiven \mX_\ell \Big] +  \!\!\!\!\!\!\!\!\!\!\!\!\!\!\!\!\sum_{\substack{i=1\\N_1(i)=\ell~{\rm or}~\tN_1(i)=\ell}}^n \!\!\!\!\!\!\!\!\!\!\!\!\!\!\!\!\E \Big[ F_Y\big(Y_i \wedge Y_{\tN_1(i)}\big) \Biggiven \mX_\ell \Big] \Big\}^2\notag\\
  \le & \frac{72n^4}{(n^2-1)^2} \E\Big\{\frac{1}{n} \sum_{i=1}^n \E \Big[\ind\big(Y_\ell \le Y_i \wedge Y_{N_1(i)}\big) \Biggiven \mX \Big] - h_0(X_\ell) \notag\\
  & + \frac{1}{n} \sum_{k=1,k \neq \ell}^n \E \Big[\ind\big(Y_k \le Y_\ell \wedge Y_{N_1(\ell)}\big) \Biggiven \mX \Big] - \E \Big[ F_Y\big(Y_\ell \wedge Y_{N_1(\ell)}\big) \Biggiven \mX \Big]\notag\\
  & + \frac{1}{n}\!\!\!\!\!\!\!\! \sum_{\substack{i=1\\N_1(i)=\ell~{\rm or}~\tN_1(i)=\ell}}^n \sum_{k=1,k \neq \ell}^n \E \Big[\ind\big(Y_k \le Y_i \wedge Y_{N_1(i)}\big) \Biggiven \mX \Big] -  \!\!\!\!\!\!\!\!\!\!\!\!\!\!\!\!\sum_{\substack{i=1\\N_1(i)=\ell~{\rm or}~\tN_1(i)=\ell}}^n \!\!\!\!\!\!\!\!\E \Big[ F_Y\big(Y_i \wedge Y_{N_1(i)}\big) \Biggiven \mX \Big] \Big\}^2\notag\\
  \le & \frac{216n^4}{(n^2-1)^2} \Big\{ \E\Big[\frac{1}{n} \sum_{i=1}^n \E \Big[\ind\big(Y_\ell \le Y_i \wedge Y_{N_1(i)}\big) \Biggiven \mX \Big] - h_0(X_\ell) \Big]^2 \notag\\
  & + \E\Big[ \frac{1}{n} \sum_{k=1,k \neq \ell}^n \E \Big[\ind\big(Y_k \le Y_\ell \wedge Y_{N_1(\ell)}\big) \Biggiven \mX \Big] - \E \Big[ F_Y\big(Y_\ell \wedge Y_{N_1(\ell)}\big) \Biggiven \mX \Big] \Big]^2\notag\\
  & + \E\Big[\!\!\!\!\!\!\!\!\sum_{\substack{i=1\\N_1(i)=\ell~{\rm or}~\tN_1(i)=\ell}}^n \!\!\!\!\!\!\!\!\Big(\frac{1}{n} \sum_{k=1,k \neq \ell}^n \E \Big[\ind\big(Y_k \le Y_i \wedge Y_{N_1(i)}\big) \Biggiven \mX \Big] - \E \Big[ F_Y\big(Y_i \wedge Y_{N_1(i)}\big) \Biggiven \mX \Big] \Big) \Big]^2\Big\}\notag\\
  =: & \frac{216n^4}{(n^2-1)^2}\Big(\tT_1 + \tT_2 + \tT_3\Big); \notag
\end{align}
recall that $[X_i]_{i=1}^n$ are independent and identically distributed (i.i.d.), and $[\tX_\ell]_{\ell=1}^n$ are independent copies of $[X_\ell]_{\ell=1}^n$.

We then establish the following three lemmas.

\begin{lemma}\label{lemma:hayek5}
   $ \lim_{n \to \infty} \tT_1 = 0.
  $
\end{lemma}

\begin{lemma}\label{lemma:hayek6}
  $
    \lim_{n \to \infty} \tT_2 = 0.
  $
\end{lemma}

\begin{lemma}\label{lemma:hayek7}
  $
    \lim_{n \to \infty} \tT_3 = 0.
  $
\end{lemma}

Applying Lemmas~\ref{lemma:hayek5}-\ref{lemma:hayek7} to \eqref{eq:hayekmain24} yields
\begin{align}\label{eq:hayekmain25}
  \lim_{n \to \infty} n \Var [\E[\xi_n - \widecheck{\xi}_n \given \mX]] = 0.
\end{align}

\vspace{0.5cm}

{\bf Step III.} $\lim_{n \to \infty} n\Var[\widecheck{\xi}_n - \xi_n^*] = 0$.

By the definition of $\widecheck{\xi}_n$ in \eqref{eq:checkxin}, one has
\begin{align*}
  \widecheck{\xi}_n = & \frac{6n^2}{n^2-1} \Big(\frac{1}{n} \sum_{i=1}^n \min\big\{F_Y(Y_i), F_Y(Y_{N_1(i)})\big\} - \frac{1}{n(n-1)} \sum_{\substack{i,j=1\\i \neq j}}^n \min\big\{F_Y(Y_i), F_Y(Y_j)\big\} \\
  &+ \frac{1}{n}\sum_{i=1}^n g(Y_i)  + \frac{1}{n(n-1)} \sum_{\substack{i,j=1\\i \neq j}}^n\E \Big[ \min\big\{F_Y(Y_i), F_Y(Y_j)\big\} \Biggiven X_i,X_j \Big] \\
  &- \frac{1}{n} \sum_{i=1}^n \E \Big[ g(Y_i) \Biggiven X_i \Big] + \frac{1}{n} \sum_{i=1}^n h_0(X_i) \Big).
\end{align*}

Notice that $\widecheck{\xi}_n$ consists of U-statistic terms. For any $x \in \bR^d$ and $t \in \bR$, define
\begin{align*}
  \tilde{h} (t) := 2\E \Big[\min\big\{F_Y(Y), F_Y(t)\big\}\Big] - \frac{1}{3} ~~{\rm and}~~ \tilde{h}_0(x) := 2\E \Big[\min\big\{F_Y(Y), F_Y(Y_x)\big\}\Big] - \frac{1}{3},
\end{align*}
where $Y \sim F_Y$, $Y_x \sim F_{Y \given X=x}$ and are independent. Using the probability integral transform and the boundedness of $F_Y$,
\begin{align*}
  & \E\Big[\min\big\{F_Y(Y_1), F_Y(Y_2)\big\}\Big] = 1/3, ~~\E\Big[\Big[\min\big\{F_Y(Y_1), F_Y(Y_2)\big\}\Biggiven X_1, X_2 \Big]\Big] = 1/3,\\
  & \E\Big[\min\big\{F_Y(Y_1), F_Y(Y_2)\big\}\Big]^2 \le 1,~~ \E\Big[\E \Big[\min\big\{F_Y(Y_1), F_Y(Y_2)\big\} \Biggiven X_1,X_2 \Big]\Big]^2 \le 1.
\end{align*}
Then the standard U-statistic H\'ajek projection \cite[Theorem 12.3]{MR1652247} gives
\begin{align}\label{eq:hayekmain33}
  \sqrt{n} \widecheck{\xi}_n = & \frac{6n^2}{n^2-1} \Big(\frac{1}{\sqrt{n}} \sum_{i=1}^nF_Y(Y_i \wedge Y_{N_1(i)}) - \frac{1}{\sqrt{n}} \sum_{i=1}^n \tilde{h}(Y_i) + \frac{1}{\sqrt{n}} \sum_{i=1}^n g(Y_i) \\
  & + \frac{1}{\sqrt{n}} \sum_{i=1}^n \tilde{h}_0(X_i) - \frac{1}{\sqrt{n}} \sum_{i=1}^n \E \Big[ g(Y_i) \Biggiven X_i \Big] + \frac{1}{\sqrt{n}} \sum_{i=1}^n h_0(X_i) \Big) + Q,
  \notag
\end{align}
with $\E[Q^2] \lesssim n^{-1}$.

Notice that for $\tilde{h}$ and $\tilde{h}_0$, $F_Y(Y)$ follows a uniform distribution on $[0,1]$ with $Y \sim F_Y$. Then it is ready to check
\[
  \tilde{h}(t) = 2F_Y(t) - F_Y^2(t) - \frac{1}{3} ~~~{\rm and}~~~ \tilde{h}_0(x) = 2\E[F_Y(Y) \given X=x] - \E[F_Y^2(Y) \given X=x] - \frac{1}{3}.
\]
Recall that for any $t \in \bR$, $h(t) = \E [G_X^2(t)]$ and $g(t)= \E [G_X^2(t)] - G^2(t) = h(t) - G^2(t)$. Then
\begin{align*}
   g(t) - \tilde{h}(t) =& h(t) - G^2(t) - \Big[ 2F_Y(t) - F_Y^2(t) - \frac{1}{3} \Big]\yestag\label{eq:hayekmain31}\\
  = & h(t) - \Big(1-F_Y(t)\Big)^2 - \Big[ 2F_Y(t) - F_Y^2(t) - \frac{1}{3} \Big] \\
  =& h(t) - \frac{2}{3}.
\end{align*}
Similarly, recall that $h_0(x) = \E[h(Y) \given X=x]$ and $g(t)=h(t) - G^2(t)$. Then for any $x \in \bR^d$,
\begin{align}\label{eq:hayekmain32}
   &\tilde{h}_0(x) - \E [ g(Y) \given X=x ] + h_0(x) \\
   =& \tilde{h}_0(x) - \E [ g(Y) \given X=x ] + \E[h(Y) \given X=x] \notag\\
  = & \E[G^2(Y) \given X=x] + 2\E[F_Y(Y) \given X=x] - \E[F_Y^2(Y) \given X=x] - \frac{1}{3} \notag\\
  =& 2/3.\notag
\end{align}

Plugging \eqref{eq:hayekmain31} and \eqref{eq:hayekmain32} to \eqref{eq:hayekmain33} yields
\begin{align*}
  \sqrt{n} \widecheck{\xi}_n = & \frac{6n^2}{n^2-1} \Big(\frac{1}{\sqrt{n}} \sum_{i=1}^nF_Y(Y_i \wedge Y_{N_1(i)}) + \frac{1}{\sqrt{n}} \sum_{i=1}^n h(Y_i)\Big)+ Q = \sqrt{n} \xi_n^* + Q.
\end{align*}
Since $\E[Q^2] \lesssim n^{-1}$, we obtain
\begin{align}\label{eq:hayekmain35}
  \lim_{n \to \infty} n\Var[\widecheck{\xi}_n - \xi_n^*] = 0.
\end{align}

\vspace{0.5cm}

Lastly, combining \eqref{eq:hayekmain15}, \eqref{eq:hayekmain25}, and \eqref{eq:hayekmain35} completes the proof.
\end{proof}


\begin{proof}[Proof of Theorem~\ref{thm:clt}]

Let
\[
  W_n:= \frac{1}{\sqrt{n}} \sum_{i=1}^nF_Y(Y_i \wedge Y_{N_1(i)}) + \frac{1}{\sqrt{n}} \sum_{i=1}^n h(Y_i).
\]

Then $\sqrt{n} \xi_n^* = \frac{6n^2}{n^2-1} W_n$, and 
\begin{align}\label{eq:clt1}
  \big(\xi_n^* - \E[\xi_n^*]\big)/\sqrt{\Var[\xi_n^*]} = (W_n-\E[W_n])/\sqrt{\Var[W_n]}.
\end{align}
It suffices to establish the self-normalization central limit theorem for $W_n$.

Let $\delta_n$ be the Kantorovich–Wasserstein distance between the laws of 
\[
(W_n-\E[W_n])/\sqrt{\Var[W_n]}
\]
and the standard Gaussian. Notice that 
\begin{itemize}
\item[(i)] for any $i \in \zahl{n}$, $F_Y(Y_i \wedge Y_{N_1(i)}) + h(Y_i)$ is the function of $(X_i,Y_i)$ and its NN $(X_{N_1(i)},Y_{N_1(i)})$, with NN graph constructed by $\{X_i\}_{i=1}^n$; 
\item[(ii)] both $F_Y$ and $h$ are bounded;
\item[(iii)] by Proposition~\ref{prop:variance} and Theorem~\ref{thm:hayek}, $\liminf_{n \to \infty} \Var[W_n] = \liminf_{n \to \infty} n\Var[\xi_n^*]/36$, which is further equal to $\liminf_{n \to \infty} n\Var[\xi_n]/36 > 0$. 
\end{itemize}
Then using Theorem 3.4 in \cite{MR2435859} with some minor modification since we now consider $[(X_i,Y_i)]_{i=1}^n$ instead of $[X_i]_{i=1}^n$, one can show $\lim_{n \to \infty} \delta_n = 0$. Since Kantorovich–Wasserstein distance is stronger than weak convergence, we obtain 
\begin{align}\label{eq:clt2}
  (W_n-\E[W_n])/\sqrt{\Var[W_n]} \stackrel{\sf d}{\longrightarrow} N(0,1).
\end{align}


Combining \eqref{eq:clt1} and \eqref{eq:clt2} completes the proof for $\xi_n^*$.

For $\overline{\xi}_n^*$, the only difference is that this time we consider the right NN instead of NN. While Theorem 3.4 in \cite{MR2435859} can not be directly applied, we can identify an interaction rule as Step III of the proof of Theorem 3.2 in \cite{lin2021boosting} with the number of right NN to be 1. Then the self-normalization central limit theorem for $\overline{\xi}_n^*$ is followed.
\end{proof}

\begin{proof}[Proof of Theorem~\ref{thm:var}]
Invoking \eqref{eq:hayekmain11} and Lemmas \ref{lemma:hayek1} and \ref{lemma:hayek2}, one has
\begin{align*}
  n \E[\Var[\xi_n \given \mX]] = & 36(1+O(n^{-2})) \Big\{\E \Big[ \Var \Big[ F_Y\big(Y_1 \wedge \tY_1\big) \Biggiven X_1 \Big] \Big] \\
  & + 2\E \Big[\Cov\Big[F_Y\big(Y_1 \wedge \tY_1\big), F_Y\big(\tY_1 \wedge \tY_1' \big) \Biggiven X_1 \Big] \ind \Big( 1 \neq N_1(N_1(1)) \Big) \Big] \\
  & + \E \Big[\Cov\Big[F_Y\big(Y_1 \wedge \tY_1\big), F_Y\big(\tY_1 \wedge \tY_1' \big) \Biggiven X_1 \Big] \Big\lvert \Big\{j: j \neq 1, N_1(j) = N_1(1)\Big\} \Big\rvert \Big] \\
  & + \E \Big[ \Var \Big[ F_Y\big(Y_1 \wedge \tY_1\big) \Biggiven X_1 \Big] \ind \Big( 1 = N_1(N_1(1)) \Big) \Big]\\
  & + 4 \E \Big[\Cov\Big[\ind\big(Y_2 \le Y_1 \wedge \tY_1\big), F_Y\big(Y_2 \wedge \tY_2\big) \Biggiven X_1,X_2 \Big] \Big] \\
  & + \E \Big[\Cov\Big[\ind\big(Y_3 \le Y_1 \wedge \tY_1\big), \ind\big(Y_3 \le Y_2 \wedge \tY_2\big) \Biggiven X_1,X_2,X_3 \Big] \Big] \Big\} + o(1).
\end{align*}

The following lemma establishes approximation for each term above.
\begin{lemma}\label{lemma:var1}
  \begin{align*}
    &\frac{1}{n^3} \sum_{i=1}^n \Big[ \Big(R_i \wedge R_{N_1(i)}\Big)  \Big(R_i \wedge R_{N_1(i)} - R_{N_2(i)} \wedge R_{N_3(i)}\Big)\Big] - \E \Big[ \Var \Big[ F_Y\big(Y_1 \wedge \tY_1\big) \Biggiven X_1 \Big] \Big] \stackrel{\sf p}{\longrightarrow} 0,\\
    &\frac{1}{n^3} \sum_{i=1}^n \Big[\Big(R_i \wedge R_{N_1(i)}\Big)\Big(R_i \wedge R_{N_2(i)} - R_{N_2(i)} \wedge R_{N_3(i)}\Big) \ind \Big( i \neq N_1(N_1(i)) \Big) \Big] \\
    &- \E \Big[\Cov\Big[F_Y\big(Y_1 \wedge \tY_1\big), F_Y\big(\tY_1 \wedge \tY_1' \big) \Biggiven X_1 \Big] \ind \Big( 1 \neq N_1(N_1(1)) \Big) \Big] \stackrel{\sf p}{\longrightarrow} 0,\\
    &\frac{1}{n^3} \sum_{i=1}^n \Big[\Big(R_i \wedge R_{N_1(i)}\Big)\Big(R_i \wedge R_{N_2(i)} - R_{N_2(i)} \wedge R_{N_3(i)}\Big) \Big\lvert \Big\{j: j \neq i, N_1(j) = N_1(i)\Big\} \Big\rvert \Big] \\
    &- \E \Big[\Cov\Big[F_Y\big(Y_1 \wedge \tY_1\big), F_Y\big(\tY_1 \wedge \tY_1' \big) \Biggiven X_1 \Big] \Big\lvert \Big\{j: j \neq 1, N_1(j) = N_1(1)\Big\} \Big\rvert \Big] \stackrel{\sf p}{\longrightarrow} 0,\\
    &\frac{1}{n^3} \sum_{i=1}^n \Big[ \Big(R_i \wedge R_{N_1(i)}\Big)  \Big(R_i \wedge R_{N_1(i)} - R_{N_2(i)} \wedge R_{N_3(i)}\Big) \ind \Big( i = N_1(N_1(i)) \Big) \Big] \\
    &- \E \Big[ \Var \Big[ F_Y\big(Y_1 \wedge \tY_1\big) \Biggiven X_1 \Big] \ind \Big( 1 = N_1(N_1(1)) \Big) \Big] \stackrel{\sf p}{\longrightarrow} 0,\\
    &\frac{1}{n^2(n-1)} \sum_{\substack{i,j=1\\i\neq j}}^n \Big[\ind\Big(R_i \le R_j \wedge R_{N_1(j)}\Big)\Big(R_i \wedge R_{N_1(i)} - R_{N_1(i)} \wedge R_{N_2(i)}\Big)\Big]\\
    &- \E \Big[\Cov\Big[\ind\big(Y_2 \le Y_1 \wedge \tY_1\big), F_Y\big(Y_2 \wedge \tY_2\big) \Biggiven X_1,X_2 \Big] \Big] \stackrel{\sf p}{\longrightarrow} 0,\\
    &\frac{1}{n(n-1)(n-2)} \sum_{\substack{i,j,k=1\\i\neq j \neq k}}^n \Big[\ind\Big(R_i \le R_j \wedge R_{N_1(j)}\Big)\Big(\ind\Big(R_i \le R_k \wedge R_{N_1(k)}\Big) - \ind\Big(R_{N_1(i)} \le R_k \wedge R_{N_1(k)}\Big)\Big)\Big]\\
    &- \E \Big[\Cov\Big[\ind\big(Y_3 \le Y_1 \wedge \tY_1\big), \ind\big(Y_3 \le Y_2 \wedge \tY_2\big) \Biggiven X_1,X_2,X_3 \Big] \Big] \stackrel{\sf p}{\longrightarrow} 0.
  \end{align*}
\end{lemma}

On the other hand, Lemma~\ref{lemma:variance,cond} in the supplement yields
\begin{align*}
  n \Var[\E[\xi_n \given \mX]] = 36(1+O(n^{-2})) \Var\Big[h_1(X_1) + h_0(X_1)\Big] + o(1),
\end{align*}
where we define $h_0(x) = \E[h(Y) \given X=x]$ and $h
_1(x) = \E[F_Y(Y \wedge \tY) \given X=x]$ with $Y,\tY$ independently drawn from $Y \given X=x$.

The following lemma establishes approximation for each term above.
\begin{lemma}\label{lemma:var2}
  \begin{align*}
    &\frac{1}{n(n-1)(n-2)} \sum_{\substack{i,j,k=1\\i\neq j \neq k}}^n \ind\Big(R_i \le R_j \wedge R_{N_1(j)}\Big)\ind\Big(R_{N_1(i)} \le R_k \wedge R_{N_1(k)}\Big) \\
    & - \Big[\frac{1}{n^2} \sum_{i=1}^n \Big(R_i \wedge R_{N_1(i)}\Big)\Big]^2 - \Var \Big[  h_0(X_1) \Big] \stackrel{\sf p}{\longrightarrow} 0,\\
    &\frac{1}{n^2(n-1)} \sum_{\substack{i,j=1\\i\neq j}}^n \ind\Big(R_i \le R_j \wedge R_{N_1(j)}\Big)\Big(R_{N_1(i)} \wedge R_{N_2(i)}\Big) - \Big[\frac{1}{n^2} \sum_{i=1}^n \Big(R_i \wedge R_{N_1(i)}\Big)\Big]^2\\
    &- \Cov \Big[h_0(X_1), h_1(X_1)\Big] \stackrel{\sf p}{\longrightarrow} 0,\\
    &\frac{1}{n^3} \sum_{i=1}^n \Big(R_i \wedge R_{N_1(i)}\Big) \Big(R_{N_2(i)} \wedge R_{N_3(i)} \Big) - \Big[\frac{1}{n^2} \sum_{i=1}^n \Big(R_i \wedge R_{N_1(i)}\Big)\Big]^2 - \Var \Big[  h_1(X_1) \Big] \stackrel{\sf p}{\longrightarrow} 0.
  \end{align*}
\end{lemma}

Combining Lemma~\ref{lemma:var1} with Lemma~\ref{lemma:var2} and from the definition of $\hat{\sigma}^2$, the proof of \eqref{eq:var-est1} is complete.

For $\overline{\xi}_n$, notice that there is only one $i \in \zahl{n}$ such that $i=\overline{N}_1(\overline{N}_1(i))$, and $\lvert \{j: j \neq i, \overline{N}_1(j) = \overline{N}_1(i)\} \rvert = 0$ for all $i \in \zahl{n}$ except two $i$'s such that $X_i$'s are the largest two. Then the variance estimator can be simplified to $\hat{\overline{\sigma}}^2$.
\end{proof}

\newpage{}

\appendix

{\centering\huge Supplement to ``Limit theorems of Chatterjee's rank correlation'' \par}

\nb{

\section{Empirical studies}\label{sec:emp}

\subsection{Simulations}

In this section, we consider the Gaussian rotation model, where $(X,Y)$ are bivariate Gaussian with mean 0 and covariance matrix $\Sigma$, defined as
\[
 \Sigma = \bigg(
  \begin{matrix}
    1 &~~ \rho\\
    \rho &~~ 1
  \end{matrix} \bigg),~~~{\rm with}~\rho\in (
  -1,1).
\]

We consider
\begin{enumerate}[label=(\roman*)]
\item (V-LH) the asymptotic variance estimator described in Theorem~\ref{thm:var};
\item (V-B) the $m$ out of $n$ bootstrap asymptotic variance estimator described in \citet[Theorem 1]{dette2024simple};
\item (D-LH) constructing the confidence interval using the test in \eqref{eq:test};
\item (D-B) constructing the confidence interval using the variance estimator in \citet[Theorem 1]{dette2024simple} given Theorem~\ref{thm:main}.
\end{enumerate}

We compare the performance of the two methods for estimating $\xi_n$'s variance and inferring $\xi$ using various sample sizes $n=1,000, 5,000, 10,000$ and population correlations $\rho=0, 0.3, 0.5, 0.7, 0.9$. For the $m$ out of $n$ bootstrap procedure, we consider $m = \lfloor \sqrt{n} \rfloor$ as \citet[Section 3]{dette2024simple}. We set the bootstrap repetitions to be $5,000$ for each simulation and simulate $5,000$ replications to compute the square roots of the mean squared errors (RMSEs) in estimating $n\Var(\xi_n)$---of limits 0.4,  0.46, 0.51, 0.47, and 0.24 as $\rho$ changes from 0 to 0.9---as well as the empirical coverage probabilities with the nominal level $\alpha=0.05$ or 0.1.

Table \ref{tab:sim} presents the simulation results. While both the variance estimators from Theorem~\ref{thm:var} and $m$ out of $n$ bootstrap are consistent, the bootstrap variance estimator tends to perform well under finite samples, with smaller RMSE and better coverage.

{
\begin{table}
\centering
\caption{\nb{Variance estimation and empirical coverage probability}}{
\nb{\begin{tabular}{lccccccc}
\multirow{2}{*}{$\rho$} & \multirow{2}{*}{$n$}
& \multicolumn{2}{c}{Variance, RMSE} & \multicolumn{2}{c}{Coverage, $\alpha=0.05$}&  \multicolumn{2}{c}{Coverage, $\alpha=0.1$} \\
& &V-LH & V-B & D-LH & D-B & D-LH & D-B\\
0 & 1000 & 0.17 & 0.03 & 0.90 & 0.94 & 0.85 & 0.89 \\ 
& 5000 & 0.08 & 0.02 & 0.94 & 0.94 & 0.89 & 0.89 \\
& 10000 & 0.05 & 0.01 & 0.95 & 0.95 & 0.90 & 0.90 \\
0.3 & 1000 & 0.18 & 0.05 & 0.90 & 0.93 & 0.84 & 0.87 \\
& 5000 & 0.08 & 0.03 & 0.95 & 0.95 & 0.89 & 0.89 \\
& 10000 & 0.05 & 0.02 & 0.95 & 0.95 & 0.90 & 0.90 \\ 
0.5 & 1000 & 0.16 & 0.06 & 0.91 & 0.93 & 0.85 & 0.88 \\
& 5000 & 0.07 & 0.03 & 0.95 & 0.95 & 0.89 & 0.90 \\
& 10000 & 0.05 & 0.02 & 0.95 & 0.95 & 0.90 & 0.90 \\
0.7 & 1000 & 0.15 & 0.04 & 0.91 & 0.94 & 0.85 & 0.89 \\
& 5000 & 0.06 & 0.02 & 0.95 & 0.95 & 0.90 & 0.90 \\
& 10000 & 0.04 & 0.01 & 0.95 & 0.95 & 0.90 & 0.89 \\
0.9 & 1000 & 0.12 & 0.02 & 0.82 & 0.94 & 0.75 & 0.89 \\
& 5000 & 0.04 & 0.02 & 0.94 & 0.95 & 0.89 & 0.91 \\
& 10000 & 0.03 & 0.01 & 0.95 & 0.95 & 0.90 & 0.91 \\
\end{tabular}}}
\label{tab:sim}
\end{table}
}

\subsection{Real data analysis}

In this section, we apply the one-sided test described in \eqref{eq:test} to the gene expression studies. We follow the real data analysis in \citet[Section 5]{chatterjee2020new} to analyze the gene expression data from \cite{reshef2011detecting}. Instead of performing the independence test in \cite{chatterjee2020new}, i.e., testing $\xi = 0$, we identify the genes with oscillatory patterns by considering $H_0: \xi \leq \kappa$ for different $\kappa$. Through this new hypothesis testing framework, we are able to identify genes that are ``practically significant'', borrowing a terminology from \cite{bastian2024testing}.

We vary $\kappa$ from 0 to 1. The p-values of genes are obtained as described in the main text, and we select the significant genes after adjusting the p-values by the Benjamini–Hochberg procedure. Table~\ref{tab:sim2} presents the number of significant genes for different $\kappa$. One can see that, through this process, we are able to identify a substantially smaller number of genes. 

\begin{table}
  \caption{\nb{$\kappa$ values and corresponding number of significant genes}}
  \centering
  \nb{\begin{tabular}{ccccccccccc}
  \textbf{$\kappa$} & 0.00 & 0.05 & 0.10 & 0.15 & 0.20 & 0.25 & 0.30 & 0.35 & 0.40 & 0.45 \\
  {\rm count} & 1187 & 846 & 579 & 350 & 217 & 71 & 13 & 8 & 3 & 0 \\
  \end{tabular}}
  \label{tab:sim2}
\end{table}
}

\nb{
\section{Sobol' indices}\label{sec:sobol}

The Sobol' indices were first introduced in \cite{sobol2001global}. Consider $X=(X_1,\ldots,X_d)$ and $Y=f(X_1,\ldots,X_d)$ for a measurable function $f$. For a subset $\fu \subset \zahl{d}$ and $\fu' = \zahl{d} \setminus \fu$, let $X^\fu=(X_i,i\in\fu)$ and $X^{\fu'}=(X_i,i\in\fu')$. Then the corresponding Sobol' indices are defined as:
\begin{align*}
  S^\fu:= \frac{\Var\{\E[Y\given X^\fu]\}}{\Var[Y]},~~~~S^{\fu'}:=\frac{\Var\{\E[Y\given X^{\fu'}]\}}{\Var[Y]}.
\end{align*}

To estimate the Sobol' indices, we consider the estimator in \cite{gamboa2022global} for the case when $\lvert \fu \rvert=1$,
\begin{align*}
  \overline{\xi}_n^\fu:= \frac{\frac{1}{n} \sum_{i=1}^n Y_i Y_{\overline{N}^\fu_1(i)} - (\frac{1}{n}\sum_{i=1}^n Y_i)^2}{\frac{1}{n}\sum_{i=1}^n Y_i^2 - (\frac{1}{n}\sum_{i=1}^n Y_i)^2},
\end{align*}
where we let $\overline{N}^\fu_1(i)$ index the right NN of $X^\fu_i$ among $\{X^\fu_j\}_{j=1}^n$, with $\overline{N}^\fu_1(i) =k$ if $X_i$ is the largest and $X_k$ is the smallest.

We can generalize the estimator to the case when $\lvert \fu \rvert \ge1$ as in \cite{azadkia2019simple}. The estimator can be defined as
\begin{align*}
  \xi_n^\fu:= \frac{\frac{1}{n} \sum_{i=1}^n Y_i Y_{N^\fu_1(i)} - (\frac{1}{n}\sum_{i=1}^n Y_i)^2}{\frac{1}{n}\sum_{i=1}^n Y_i^2 - (\frac{1}{n}\sum_{i=1}^n Y_i)^2},
\end{align*}
where we let $N^\fu_1(i)$ index the NN of $X^\fu_i$ among $\{X^\fu_j\}_{j=1}^n$.

As both $\xi_n^\fu$ and $\overline{\xi}_n^\fu$ are from the class of nearest neighbor statistics, their asymptotic theories can be established in a similar way as Chatterjee's rank correlation. Here we provide asymptotic theories for $\xi_n^\fu$ estimating $S^\fu$, and those for $\overline{\xi}_n^\fu$ are nearly the same.

To establish the asymptotic theory for $\xi_n^\fu$, we first consider the joint central limit theorem of the numerator and denominator of $\xi_n^\fu$.

\begin{theorem}\label{thm:sobol}
  Assume $f$ is bounded. Then we have
  \begin{align*}
    \sqrt{n} \left(\begin{bmatrix}\frac{1}{n} \sum_{i=1}^n Y_i Y_{N^\fu_1(i)} - \Big(\frac{1}{n}\sum_{i=1}^n Y_i\Big)^2 \\ \frac{1}{n}\sum_{i=1}^n Y_i^2 - (\frac{1}{n}\sum_{i=1}^n Y_i)^2
    \end{bmatrix}
    - \begin{bmatrix}\Var\{\E[Y\given X^\fu]\} + B^\fu\\ \Var[Y]
    \end{bmatrix} \right) \stackrel{\sf d}{\longrightarrow} N(0,\Sigma),
  \end{align*}
  where the explicit form of $\Sigma$ is in the proof of Theorem~\ref{thm:sobol}, and the bias term $B^\fu = \E[Y_1 Y_{N^\fu_1(1)}] - \E\{(\E[Y\given X^\fu])^2\}$. A consistent estimator of $\hat\Sigma$ exists with the explicit form in the proof of Theorem~\ref{thm:sobol}.
\end{theorem}

For the bias term $B^\fu$, we have the following lemma.
\begin{lemma}\label{lemma:sobol}
  Assume $\lvert \fu \rvert = 1$, the support of $X^\fu$ is compact, $f$ and its first derivative with respect to $X^\fu$ are bounded. Then $B^\fu = O(n^{-1})$.
\end{lemma}

Combining Theorem~\ref{thm:sobol} and Lemma~\ref{lemma:sobol} yields the following theorem.
\begin{theorem}\label{thm:sobol0}
Assume $\lvert \fu \rvert = 1$, the support of $X^\fu$ is compact, $f$ and its first derivative with respect to $X^\fu$ are bounded. Then we have
\begin{align*}
  \sqrt{n} (\xi_n^\fu - S^\fu) \stackrel{\sf d}{\longrightarrow} N(0,\sigma^2),
\end{align*}
where $\sigma^2 = (1, -S^\fu)^\top \Sigma (1, -S^\fu)/(\Var[Y])^2$. Let $\hat\sigma^2 = (1, -\xi_n^\fu)^\top \hat\Sigma (1, -\xi_n^\fu)/[\frac{1}{n}\sum_{i=1}^n Y_i^2 - (\frac{1}{n}\sum_{i=1}^n Y_i)^2]^2$. Then $\hat\sigma^2 \stackrel{\sf p}{\longrightarrow} \sigma^2$.
\end{theorem}
}

\section{Proofs of the results in the main paper}\label{sec:proof}

In the sequel, denote the law of $Y$ by $\mu$, and the conditional law of $Y$ given $X=x$ by $\mu_x$. 

\subsection{Proof of Proposition~\ref{prop:bias}}

\begin{proof}[Proof of Proposition~\ref{prop:bias}\ref{prop:bias,null}]
Lemma 6.1 in \cite{lin2021boosting} showed $\E \big[\min\big\{R_1,R_2\big\} \big] = (n+1)/3$. Then if $X$ and $Y$ are independent,
\begin{align*}
  \E[\xi_n] = \frac{6n}{n^2-1} \E \big[\min\big\{R_1, R_{N_1(1)}\big\} \big] - \frac{2n+1}{n-1} = \frac{6n}{n^2-1} \E \big[\min\big\{R_1, R_2\big\} \big] - \frac{2n+1}{n-1} = -\frac{1}{n-1}.
\end{align*}

When $d=1$, there exists only one index $i \in \zahl{n}$ such that $\overline{N}_1(i) = i$. Then
\begin{align*}
  & \E[\overline{\xi}_n] = 1 - \frac{3}{n^2-1} \E \Big[\sum_{i=1}^n \Big\lvert R_{\overline{N}_1(i)} - R_i \Big\rvert \Big] = 1 - \frac{3(n-1)}{n^2-1} \E \Big[\Big\lvert R_2 - R_1 \Big\rvert \Big]\\
  = & 1 - \frac{3(n-1)}{n^2-1} \Big( \E \big[R_1 \big] + \E \big[R_2 \big] - 2 \E \big[\min\big\{R_1, R_2\big\} \big] \Big) = 1 - \frac{3(n-1)}{n^2-1} \frac{(n+1)}{3} = 0.
\end{align*}
This completes the proof.
\end{proof}

\begin{proof}[Proof of Proposition~\ref{prop:bias}\ref{prop:bias,alter}]
Notice that for any $i \in \zahl{n}$, $\min\big\{R_i, R_{N_1(i)}\big\}=\sum_{k=1}^n \ind\big(Y_k \le Y_i \wedge Y_{N_1(i)}\big)$. From \eqref{eq:xin} and since $[(X_i,Y_i)]_{i=1}^n$ are i.i.d., we have
\begin{align*}
   \E [\xi_n] =& \frac{6}{n^2-1} \E \Big[\sum_{i=1}^n \min\big\{R_i, R_{N_1(i)}\big\} \Big] - \frac{2n+1}{n-1} \\
   =& \frac{6n}{n^2-1} \E \Big[ \min\big\{R_1, R_{N_1(1)}\big\} \Big] - \frac{2n+1}{n-1}\\
  = &  \frac{6N_1(N-1)}{n^2-1} \E \Big[\ind\big(Y_2 \le Y_1 \wedge Y_{N_1(1)}\big) \Big] + \frac{6n}{n^2-1}  \E \Big[\ind\big(Y_1 \le Y_1 \wedge Y_{N_1(1)}\big) \Big] - \frac{2n+1}{n-1}\\
  = & 6 \E \Big[\ind\big(Y_2 \le Y_1 \wedge \tY_1\big) \Big] -2 + 6 \Big(\E \Big[\ind\big(Y_2 \le Y_1 \wedge Y_{N_1(1)}\big) \Big] - \E \Big[\ind\big(Y_2 \le Y_1 \wedge \tY_1\big) \Big]\Big) \\
  & - \frac{6}{n+1} \E \Big[\ind\big(Y_2 \le Y_1 \wedge Y_{N_1(1)}\big) \Big] + \frac{6n}{n^2-1}  \E \Big[\ind\big(Y_1 \le Y_1 \wedge Y_{N_1(1)}\big) \Big] - \frac{3}{n-1}\\
  =: & 6 \E \Big[\ind\big(Y_2 \le Y_1 \wedge \tY_1\big) \Big] -2 + Q.
  \yestag\label{eq:bias1}
\end{align*}

For the first term in \eqref{eq:bias1}, 
\[
  \E \Big[\ind\big(Y_2 \le Y_1 \wedge \tY_1\big) \Big] = \E \Big[\int G_{X_1}^2(t) \d \mu_{X_2}(t) \Big] =  \int \E[G_X^2(t)] \d \mu(t).
\]
Noticing that $\int G^2(t) \d \mu(t) = 1/3$, one has
\[
  6 \E \Big[\ind\big(Y_2 \le Y_1 \wedge \tY_1\big) \Big] -2 = 6 \int \Big( \E[G_X^2(t)] - G^2(t) \Big) \d \mu(t).
\]

On the other hand, it is ready to check
\[
  \int \Var\big\{\E\big[\ind\big(Y\geq t\big) \given X \big] \big\} \d \mu(t) = \int \Big( \E[G_X^2(t)] - G^2(t) \Big) \d \mu(t),
\]
and
\[
  \int \Var\big\{\ind\big(Y\geq t\big)\big\}\d \mu(t) = \frac{1}{6}.
\]

Accordingly, combining \eqref{eq:xi} and \eqref{eq:bias1}, we obtain
\[
  \E[\xi_n] - \xi = \E[\xi_n] - 6 \E \Big[\ind\big(Y_2 \le Y_1 \wedge \tY_1\big) \Big] + 2 = Q.
\]

Let $N^{-2}(1)$ index the NN of $X_1$ among $\{X_i\}_{i=1}^n \setminus \{X_2\}$. Using the definition of $Q$ and noticing that the indicator function is bounded by 1, we have
\begin{align*}
  \lvert Q \rvert \lesssim & \Big\lvert \E \Big[\ind\big(Y_2 \le Y_1 \wedge Y_{N_1(1)}\big) \Big] - \E \Big[\ind\big(Y_2 \le Y_1 \wedge \tY_1 \big) \Big] \Big\rvert + \frac{1}{n}\\
  \le &  \Big\lvert \E \Big[\ind\big(Y_2 \le Y_1 \wedge Y_{N^{-2}(1)}\big) \Big] - \E \Big[\ind\big(Y_2 \le Y_1 \wedge \tY_1 \big) \Big] \Big\rvert + 2 \P(N_1(1)=2) + \frac{1}{n}\\
  = & \Big\lvert \E \Big[F_Y\big(Y_1 \wedge Y_{N^{-2}(1)}\big) \Big] - \E \Big[F_Y\big(Y_1 \wedge \tY_1 \big) \Big] \Big\rvert + 2 \P(N_1(1)=2) + \frac{1}{n}\\
  \le & \Big\lvert \E \Big[F_Y\big(Y_1 \wedge Y_{N_1(1)}\big) \Big] - \E \Big[F_Y\big(Y_1 \wedge \tY_1 \big) \Big] \Big\rvert + 4 \P(N_1(1)=2) + \frac{1}{n}.
\end{align*}

For the second term above, $\P(N_1(1)=2) = 1/(n-1)$. For the first term above, recall that $G_X(t)= \P\big(Y \ge t \given X\big)$. Then since $0 \le G_X(t) \le 1$ holds for any $t \in \bR$, one has
\begin{align*}
   \Big\lvert \E \Big[F_Y\big(Y_1 \wedge Y_{N_1(1)}\big) \Big] - \E \Big[F_Y\big(Y_1 \wedge \tY_1 \big) \Big] \Big\rvert =& \Big\lvert \int \Big(\E \Big[G_{X_1}(t) G_{X_{N_1(1)}}(t) \Big] - \E \Big[G_{X_1}^2(t) \Big] \Big) \d \mu(t)\Big\rvert\\
  \le & \int \E \Big\lvert G_{X_{N_1(1)}}(t) - G_{X_1}(t) \Big\rvert \d \mu(t).
\end{align*}

In the same way as the proof of Theorem 4.1 in \cite{azadkia2019simple}, essentially Lemma 14.1 and the proof of Lemma 14.2 therein, and from the assumptions, one could deduce
\[
  \int \E \Big\lvert G_{X_{N_1(1)}}(t) - G_{X_1}(t) \Big\rvert \d \mu(t) \lesssim \frac{(\log n)^{d+\beta+1 + \ind(d=1)}}{n^{1/d}},
\]
and the proof for $\xi_n$ is thus complete.

Similar analyses can be performed for $\overline{\xi}_n$ as well and details are accordingly omitted.
\end{proof}

\subsection{Proof of Proposition~\ref{prop:variance}}

Recall that $h_0(x) = \E[h(Y) \given X=x] = \int \E[G_X^2(t)] \d \mu_x(t)$ and let us further define 
\[
h_1(x):= \E[F_Y(Y \wedge \tY) \given X=x] = \int F_Y(t \wedge t') \d \mu_x(t) \d \mu_x(t'). 
\]
The following lemma about $\E[\xi_n^* \given \mX]$ will be used.

\begin{lemma}\label{lemma:variance,cond} We have
  \begin{align*} 
    \lim_{n \to \infty} \Big\{ n \Var \Big[ \frac{6n}{n^2-1} \sum_{i=1}^n \Big(h_1(X_i) + h_0(X_i)\Big) - \E[\xi_n^* \given \mX] \Big]\Big\} = 0.
  \end{align*}
\end{lemma}

\nb{
\begin{proof}[Proof of Proposition~\ref{prop:variance}~\ref{prop:variance,exist}]
By the proof of Theorem~\ref{thm:var}, we have explicit representations for $n \E[\Var[\xi_n \given \mX]]$ and $n \Var[\E[\xi_n \given \mX]]$, up to a small order term. By Lemma 20.6 in \cite{biau2015lectures} and the dominated convergence theorem, the limits of $n \E[\Var[\xi_n \given \mX]]$ and $n \Var[\E[\xi_n \given \mX]]$ exist, and then the proof is complete. Same results also hold for $n \Var[\overline{\xi}_n]$.
\end{proof}
}

\begin{proof}[Proof of Proposition~\ref{prop:variance}~\ref{prop:variance,lower}]
From \eqref{eq:hayekmain11},
\begin{align*}
  n\Var[\xi_n] \ge n \E[\Var[\xi_n \given \mX]] = \frac{36n^4}{(n^2-1)^2} \Big(\E[T_1] + \E[T_2] + \E[T_3] + \E[T_4] + \E[T_5]\Big).
\end{align*}

Using Lemmas~\ref{lemma:hayek1} and \ref{lemma:hayek2}, and then noticing that for any $X_1 \in \bR^d$, we have 
\[
\Cov[F_Y(Y_1 \wedge \tY_1), F_Y(\tY_1 \wedge \tY_1' ) \given X_1 ] \ge 0,
\]
one can deduce 
\begin{align*}
  n\Var[\xi_n] \ge & 36(1+O(n^{-2})) \Big\{\E \Big[ \Var \Big[ F_Y\big(Y_1 \wedge \tY_1\big) \Biggiven X_1 \Big] \Big] \\
  & + 2\E \Big[\Cov\Big[F_Y\big(Y_1 \wedge \tY_1\big), F_Y\big(\tY_1 \wedge \tY_1' \big) \Biggiven X_1 \Big] \ind \Big( 1 \neq N_1(N_1(1)) \Big) \Big] \\
  & + \E \Big[ \Var \Big[ F_Y\big(Y_1 \wedge \tY_1\big) \Biggiven X_1 \Big] \ind \Big( 1 = N_1(N_1(1)) \Big) \Big]\\
  & + 4 \E \Big[\Cov\Big[\ind\big(Y_2 \le Y_1 \wedge \tY_1\big), F_Y\big(Y_2 \wedge \tY_2\big) \Biggiven X_1,X_2 \Big] \Big] \\
  & + \E \Big[\Cov\Big[\ind\big(Y_3 \le Y_1 \wedge \tY_1\big), \ind\big(Y_3 \le Y_2 \wedge \tY_2\big) \Biggiven X_1,X_2,X_3 \Big] \Big] \Big\} + o(1).
\end{align*}
For the last term above, recalling that $h(t) = \E[G_X^2(t)]$ from \eqref{eq:h}, one has
\begin{align*}
  & \E \Big[\Cov\Big[\ind\big(Y_3 \le Y_1 \wedge \tY_1\big), \ind\big(Y_3 \le Y_2 \wedge \tY_2\big) \Biggiven X_1,X_2,X_3 \Big] \Big]\\
  = & \E \Big[ \int G_{X_1}^2(t) G_{X_2}^2(t) \d \mu_{X_3}(t) - \Big(\int G_{X_1}^2(t) \d \mu_{X_3}(t) \Big) \Big(\int G_{X_2}^2(t) \d \mu_{X_3}(t) \Big) \Big]\\
  = & \E \Big[ \int h^2(t) \d \mu_{X_3}(t) - \Big(\int h(t) \d \mu_{X_3}(t) \Big) \Big(\int h(t) \d \mu_{X_3}(t) \Big) \Big]\\
  = & \E \Big[ \Var \Big[ h(Y_1) \Biggiven X_1 \Big] \Big].
\end{align*}
For the second last term, 
\begin{align*}
  &\E \Big[\Cov\Big[\ind\big(Y_2 \le Y_1 \wedge \tY_1\big), F_Y\big(Y_2 \wedge \tY_2\big) \Biggiven X_1,X_2 \Big] \Big]\\
  = & \E \Big[ \int G_{X_1}^2(t) F_Y(t \wedge t') \d \mu_{X_2}(t) \d \mu_{X_2}(t') - \Big(\int G_{X_1}^2(t) \d \mu_{X_2}(t) \Big) \Big(\int F_Y(t \wedge t') \d \mu_{X_2}(t) \d \mu_{X_2}(t') \Big) \Big]\\
  = & \E \Big[ \int h(t) F_Y(t \wedge t') \d \mu_{X_2}(t) \d \mu_{X_2}(t') - \Big(\int h(t) \d \mu_{X_2}(t) \Big) \Big(\int F_Y(t \wedge t') \d \mu_{X_2}(t) \d \mu_{X_2}(t') \Big) \Big]\\
  = & \E \Big[\Cov\Big[h\big(Y_1\big), F_Y\big(Y_1 \wedge \tY_1\big) \Biggiven X_1\Big] \Big].
\end{align*}
We then have
\begin{align*}
  n\Var[\xi_n] \ge & 36(1+O(n^{-2})) \Big\{\E \Big[ \Var \Big[ F_Y\big(Y_1 \wedge \tY_1\big) \Biggiven X_1 \Big] \Big] \\
  & + \E \Big[\Big(2\Cov\Big[F_Y\big(Y_1 \wedge \tY_1\big), F_Y\big(\tY_1 \wedge \tY_1' \big) \Biggiven X_1 \Big] \Big) \bigwedge \Var \Big[ F_Y\big(Y_1 \wedge \tY_1\big) \Biggiven X_1 \Big] \Big] \\
  & + 4 \E \Big[\Cov\Big[h\big(Y_1\big), F_Y\big(Y_1 \wedge \tY_1\big) \Biggiven X_1\Big] \Big] + \E \Big[ \Var \Big[ h(Y_1) \Biggiven X_1 \Big] \Big] \Big\} + o(1).
  \yestag\label{eq:variance12}
\end{align*}

Notice that
\begin{align*}
  & 2 \Var \Big[ F_Y\big(Y_1 \wedge \tY_1\big) \Biggiven X_1 \Big] + 4 \Cov\Big[h\big(Y_1\big), F_Y\big(Y_1 \wedge \tY_1\big) \Biggiven X_1\Big] + \Var \Big[ h(Y_1) \Biggiven X_1 \Big]\\
  = & 2 \Var \Big[ F_Y\big(Y_1 \wedge \tY_1\big) + \frac{1}{2} h(Y_1) + \frac{1}{2} h(\tY_1) \Biggiven X_1 \Big],
  \yestag\label{eq:variance1}
\end{align*}
and
\begin{align*}
  & \Var \Big[ F_Y\big(Y_1 \wedge \tY_1\big) \Biggiven X_1 \Big] + 2\Cov\Big[F_Y\big(Y_1 \wedge \tY_1\big), F_Y\big(\tY_1 \wedge \tY_1' \big) \Biggiven X_1 \Big] \\
  & + 4 \Cov\Big[h\big(Y_1\big), F_Y\big(Y_1 \wedge \tY_1\big) \Biggiven X_1\Big] + \Var \Big[ h(Y_1) \Biggiven X_1 \Big]\\
  = & \frac{1}{3} \Var \Big[ F_Y\big(Y_1 \wedge \tY_1\big) + F_Y\big(Y_1 \wedge \tY_1'\big) + F_Y\big(\tY_1 \wedge \tY_1'\big) + h(Y_1) + h(\tY_1) + h(\tY_1')\Biggiven X_1 \Big].
  \yestag\label{eq:variance2}
\end{align*}

{\bf Case I.} If $Y$ is not a measurable function of $X$ almost surely, then
\[
  \E \Big[ \Var \Big[ F_Y\big(Y_1 \wedge \tY_1\big) + \frac{1}{2} h(Y_1) + \frac{1}{2} h(\tY_1) \Biggiven X_1 \Big] \Big] > 0,
\]
and
\[
  \E \Big[ \Var \Big[ F_Y\big(Y_1 \wedge \tY_1\big) + F_Y\big(Y_1 \wedge \tY_1'\big) + F_Y\big(\tY_1 \wedge \tY_1'\big) + h(Y_1) + h(\tY_1) + h(\tY_1')\Biggiven X_1 \Big] \Big] > 0.
\]
Combining the above two bounds with \eqref{eq:variance12}, \eqref{eq:variance1}, and \eqref{eq:variance2}  then yields
\[
  \liminf_{n \to \infty} \Big\{n \Var[\xi_n]\Big\} > 0.
\]

{\bf Case II.} If $Y$ is a measurable function of $X$ almost surely, it is ready to check that
\[
\lim_{n \to \infty} \E[T_1] = \lim_{n \to \infty} \E[T_2] = \lim_{n \to \infty} \E[T_3] = \lim_{n \to \infty} \E[T_4] = \lim_{n \to \infty} \E[T_5] = 0 
\]
using Lemmas~\ref{lemma:hayek1} and \ref{lemma:hayek2} since the variance and the covariance terms there are zero conditional on $\mX$. Accordingly, one has 
\[
\lim_{n \to \infty} n\E[\Var[\xi_n \given \mX]] = 0
\]
invoking \eqref{eq:hayekmain11}. 

It remains to establish $\lim_{n \to \infty} n\Var[\E[\xi_n \given \mX]] = 0$. From Theorem~\ref{thm:hayek}, we have 
\[
\limsup_{n \to \infty} n\Var[\E[\xi_n - \xi_n^* \given \mX]] \le \limsup_{n \to \infty} n\Var[\xi_n - \xi_n^*] = 0. 
\]
Then it suffices to establish $\lim_{n \to \infty} n\Var[\E[\xi_n^* \given \mX]] = 0$.

From Lemma~\ref{lemma:variance,cond}, we consider $\Var [ \sum_{i=1}^n (h_1(X_i) + h_0(X_i))]$. Let $Y = \phi(X)$ almost surely with $\phi$ to be a measurable function. Then 
\[
h_1(X_i) = \E[F_Y(Y \wedge \tY) \given X=X_i] = F_Y(\phi(X_i)) 
\]
and 
\[
h_0(X_i) = \E[h(Y) \given X=X_i] = h(\phi(X_i)). 
\]
Notice that for any $t \in \bR$, 
\[
h(t) = \E [G_X^2(t)] = \E [\P(Y \ge t \given X)]^2 = \E[\ind(\phi(X) \ge t)] = \P(\phi(X) \ge t), 
\]
and 
\[
F_Y(t) = \P(Y \le t) = \P(\phi(X) \le t). 
\]
We then have 
\begin{align*}
h_1(X_i) + h_0(X_i) &= F_Y(\phi(X_i)) + h(\phi(X_i)) = \P(\phi(X) \le \phi(X_i)) + \P(\phi(X) \ge \phi(X_i)) \\
&= 1+\P(\phi(X) = \phi(X_i)) = 1+\P(Y=\phi(X_i)) = 1
\end{align*}
from the continuity of $F_Y$. Then $\Var [ \sum_{i=1}^n (h_1(X_i) + h_0(X_i))]=0$ and then $\lim_{n \to \infty} n\Var[\E[\xi_n^* \given \mX]] = 0$ from Lemma~\ref{lemma:variance,cond}.

The two claims for $\overline{\xi}_n$ can be established in the same way by simply replacing $N_1(\cdot)$ by $\overline{N}_1(\cdot)$.
\end{proof}

\begin{proof}[Proof of Proposition~\ref{prop:variance}~\ref{prop:variance,upper}]

Invoking \eqref{eq:hayekmain11} and Lemmas \ref{lemma:hayek1} and \ref{lemma:hayek2},
\begin{align*}
  n \E[\Var[\xi_n \given \mX]] = & 36(1+O(n^{-2})) \Big\{\E \Big[ \Var \Big[ F_Y\big(Y_1 \wedge \tY_1\big) \Biggiven X_1 \Big] \Big] \\
  & + 2\E \Big[\Cov\Big[F_Y\big(Y_1 \wedge \tY_1\big), F_Y\big(\tY_1 \wedge \tY_1' \big) \Biggiven X_1 \Big] \ind \Big( 1 \neq N_1(N_1(1)) \Big) \Big] \\
  & + \E \Big[\Cov\Big[F_Y\big(Y_1 \wedge \tY_1\big), F_Y\big(\tY_1 \wedge \tY_1' \big) \Biggiven X_1 \Big] \Big\lvert \Big\{j: j \neq 1, N_1(j) = N_1(1)\Big\} \Big\rvert \Big] \\
  & + \E \Big[ \Var \Big[ F_Y\big(Y_1 \wedge \tY_1\big) \Biggiven X_1 \Big] \ind \Big( 1 = N_1(N_1(1)) \Big) \Big]\\
  & + 4 \E \Big[\Cov\Big[\ind\big(Y_2 \le Y_1 \wedge \tY_1\big), F_Y\big(Y_2 \wedge \tY_2\big) \Biggiven X_1,X_2 \Big] \Big] \\
  & + \E \Big[\Cov\Big[\ind\big(Y_3 \le Y_1 \wedge \tY_1\big), \ind\big(Y_3 \le Y_2 \wedge \tY_2\big) \Biggiven X_1,X_2,X_3 \Big] \Big] \Big\} + o(1).
\end{align*}

From \eqref{eq:variance1} and \eqref{eq:variance2}, one deduces
\begin{align*}
  & n \E[\Var[\xi_n \given \mX]] \\
  = & 36(1+O(n^{-2})) \Big\{2 \E \Big[ \Var \Big[ F_Y\big(Y_1 \wedge \tY_1\big) + \frac{1}{2} h(Y_1) + \frac{1}{2} h(\tY_1) \Biggiven X_1 \Big] \ind \Big( 1 = N_1(N_1(1)) \Big) \Big] \\
  & + 3\E \Big[\Var \Big[ \frac{1}{3}F_Y\big(Y_1 \wedge \tY_1\big) + \frac{1}{3}F_Y\big(Y_1 \wedge \tY_1'\big) + \frac{1}{3}F_Y\big(\tY_1 \wedge \tY_1'\big) + \frac{1}{3}h(Y_1) + \frac{1}{3}h(\tY_1) + \frac{1}{3}h(\tY_1')\Biggiven X_1 \Big] \\
  & \ind \Big( 1 \neq N_1(N_1(1)) \Big) \Big] + \E \Big[\Cov\Big[F_Y\big(Y_1 \wedge \tY_1\big), F_Y\big(\tY_1 \wedge \tY_1' \big) \Biggiven X_1 \Big] \Big] \Big\lvert \Big\{j: j \neq 1, N_1(j) = N_1(1)\Big\} \Big\rvert \Big\} + o(1).
\end{align*}

Notice that for any $t, t' \in \bR$, $F_Y(t \wedge t') \le (F_Y(t)+F_Y(t'))/2$. In addition, we have 
\[
h(t) = \E[G_X^2(t)] \le \E[G_X(t)] = 1-F_Y(t). 
\]
Then for any $Y_1,\tY_1,\tY_1' \in \bR$,
\begin{align*}
  0 \le F_Y\big(Y_1 \wedge \tY_1\big) + \frac{1}{2} h(Y_1) + \frac{1}{2} h(\tY_1) \le 1,
\end{align*}
and
\begin{align*}
  0 \le  \frac{1}{3}F_Y\big(Y_1 \wedge \tY_1\big) + \frac{1}{3}F_Y\big(Y_1 \wedge \tY_1'\big) + \frac{1}{3}F_Y\big(\tY_1 \wedge \tY_1'\big) + \frac{1}{3}h(Y_1) + \frac{1}{3}h(\tY_1) + \frac{1}{3}h(\tY_1') \le 1.
\end{align*}

Leveraging Popoviciu's inequality, for any $X_1 \in \bR$, we deduce
\begin{align*}
  & \Var \Big[ F_Y\big(Y_1 \wedge \tY_1\big) + \frac{1}{2} h(Y_1) + \frac{1}{2} h(\tY_1) \Biggiven X_1 \Big] \le \frac{1}{4},\\
  & \Var \Big[ \frac{1}{3}F_Y\big(Y_1 \wedge \tY_1\big) + \frac{1}{3}F_Y\big(Y_1 \wedge \tY_1'\big) + \frac{1}{3}F_Y\big(\tY_1 \wedge \tY_1'\big) + \frac{1}{3}h(Y_1) + \frac{1}{3}h(\tY_1) + \frac{1}{3}h(\tY_1')\Biggiven X_1 \Big] \le \frac{1}{4},\\
  & \Cov\Big[F_Y\big(Y_1 \wedge \tY_1\big), F_Y\big(\tY_1 \wedge \tY_1' \big) \Biggiven X_1 \Big] \le \Var\Big[F_Y\big(Y_1 \wedge \tY_1\big) \Biggiven X_1 \Big] \le \frac{1}{4}.
\end{align*}

Then we have
\begin{align*}
  & n \E[\Var[\xi_n \given \mX]] \\
  \le & 36(1+O(n^{-2})) \Big[\frac{1}{2} \P \Big( 1 = N_1(N_1(1)) \Big) + \frac{3}{4} \P \Big( 1 \neq N_1(N_1(1)) \Big) + \frac{1}{4} \E \Big[ \Big\lvert \Big\{j: j \neq 1, N_1(j) = N_1(1)\Big\} \Big\rvert \Big] \Big] + o(1).
\end{align*}

From Lemma 20.6 together with Theorem 20.16 in \cite{biau2015lectures}, the size of the set 
\[
\Big\lvert \Big\{j: j \neq 1, N_1(j) = N_1(1)\Big\} \Big\rvert 
\]
is always bounded by a constant that only depends on $d$. Accordingly, we have
\begin{align}\label{eq:variance5}
  \limsup_{n \to \infty} n \E[\Var[\xi_n \given \mX]]  <\infty.
\end{align}

If we further assume $F_X$ to be absolutely continuous, then Lemmas 3.2 and 3.3 in \cite{shi2021ac} show
\[
  \lim_{n \to \infty} \P \Big( 1 = N_1(N_1(1)) \Big) = \kq_d,~~ \lim_{n \to \infty} \E \Big[ \Big\lvert \Big\{j: j \neq 1, N_1(j) = N_1(1)\Big\} \Big\rvert \Big] = \ko_d.
\]
It then holds true that
\begin{align}\label{eq:variance3}
  \limsup_{n \to \infty} n \E[\Var[\xi_n \given \mX]] \le 27 - 9 \kq_d + 9 \ko_d.
\end{align}
On the other hand, Lemma~\ref{lemma:variance,cond} yields
\begin{align*}
  n \Var[\E[\xi_n \given \mX]] = 36(1+O(n^{-2})) \Var\Big[h_1(X_1) + h_0(X_1)\Big] + o(1).
\end{align*}
Using the definition of $h_0$ and $h_1$,
\begin{align*}
  0 \le h_1(X_1) + h_0(X_1) =  \E \Big[ F_Y\big(Y_1 \wedge \tY_1\big) + \frac{1}{2} h(Y_1) + \frac{1}{2} h(\tY_1) \Biggiven X_1 \Big] \le 1.
\end{align*}
Then Popoviciu's inequality implies
\begin{align}\label{eq:variance4}
  \limsup_{n \to \infty} n \Var[\E[\xi_n \given \mX]] \le 9.
\end{align}

Combining \eqref{eq:variance5}, \eqref{eq:variance3}, \eqref{eq:variance4} completes the proof for $\xi_n$.

\vspace{0.5cm}

For $\overline{\xi}_n$, the only difference is that we have
\[
  \lim_{n \to \infty} \P \Big( 1 = \overline{N}_1(\overline{N}_1(1)) \Big) = \lim_{n \to \infty} \E \Big[ \Big\lvert \Big\{j: j \neq 1, \overline{N}_1(j) = \overline{N}_1(1)\Big\} \Big\rvert \Big] = 0,
\]
and thusly one can replace the bound \eqref{eq:variance3} by
\[
  \limsup_{n \to \infty} n \E[\Var[\overline{\xi}_n \given \mX]] \le 27.
\]
We thus complete the proof. 
\end{proof}

\nb{
\subsection{Proof of Proposition~\ref{prop:test}}

Combining Theorem~\ref{thm:main}, Theorem~\ref{thm:var}, Proposition~\ref{prop:bias} and Proposition~\ref{prop:variance} using Slutsky's theorem, we have
\begin{align*}
  \sqrt{n}\big(\overline\xi_n - \xi \big)/\hat{\overline\sigma}\longrightarrow N(0,1) ~~ {\rm in~distribution}.
\end{align*}

\begin{proof}[Proof of Proposition~\ref{prop:test}~\ref{prop:size}]

For any fix probability measure satisfying $H_0$, we have $\xi \leq \kappa$, and then
\begin{align*}
  \P(T=1) = \P(\overline\xi_n>\kappa+z_{1-\alpha}\hat{\overline\sigma}/\sqrt{n}) \le \P(\overline\xi_n - \xi > z_{1-\alpha}\hat{\overline\sigma}/\sqrt{n}) = \P(\sqrt{n}(\overline\xi_n - \xi)/\hat{\overline\sigma} > z_{1-\alpha}).
\end{align*}

Then we have
\begin{align*}
  \limsup_{n \to \infty} \P(T=1) \le \limsup_{n \to \infty} \P(\sqrt{n}(\overline\xi_n - \xi)/\hat{\overline\sigma} > z_{1-\alpha}) = \alpha.
\end{align*}

\end{proof}

\begin{proof}[Proof of Proposition~\ref{prop:test}~\ref{prop:power}]
For any fix probability measure violating $H_0$, we have $\xi>\kappa$, and then
\begin{align*}
  &\P(T=1) = \P(\overline\xi_n>\kappa+z_{1-\alpha}\hat{\overline\sigma}/\sqrt{n}) = \P(\overline\xi_n - \xi >\kappa - \xi +z_{1-\alpha}\hat{\overline\sigma}/\sqrt{n}) \\
  =& \P(\sqrt{n}(\overline\xi_n - \xi)/\hat{\overline\sigma} > z_{1-\alpha} - \sqrt{n}(\xi-\kappa)/\hat{\overline\sigma}).
\end{align*}

By the central limit theorem of $\xi$ and that $\xi-\kappa>0$, we have
\begin{align*}
  \liminf_{n \to \infty} \P(T=1) = \liminf_{n \to \infty} \P(\sqrt{n}(\overline\xi_n - \xi)/\hat{\overline\sigma} > z_{1-\alpha} - \sqrt{n}(\kappa-\xi)/\hat{\overline\sigma}) = 1.
\end{align*}

\end{proof}

\begin{proof}[Proof of Proposition~\ref{prop:test}~\ref{prop:local}]
Recall that $\overline{\xi}_n^*$ is the H\'ajek representations of $\overline{\xi}_n$. Let $\mu_n^*$ be the law of $\big(\overline\xi_n^* - \E[\overline\xi_n^*]\big)/\sqrt{\Var[\overline\xi_n^*]}$ and $\nu$ be the law of the standard normal distribution. By the proof of Theorem~\ref{thm:clt}, we have $\lim_{n \to \infty} \cW(\mu_n^*,\nu) =0$, where $\cW$ is the Wasserstein-1 distance.

Let $\mu_n$ be the law of $\sqrt{n}\big(\overline\xi_n - \E[\overline\xi_n]\big)/\hat{\overline{\sigma}}$. From Proposition~\ref{prop:variance} and Theorem~\ref{thm:hayek}, we have $\limsup_{n \to \infty} \cW(\mu_n,\mu_n^*) \le \limsup_{n \to \infty} \cW_2(\mu_n,\mu_n^*) = 0$, where $\cW_2$ is the Wasserstein-2 distance. 

Then we have $\limsup_{n \to \infty} \cW(\mu_n,\nu) \le \limsup_{n \to \infty} \cW(\mu_n,\mu_n^*) + \limsup_{n \to \infty} \cW(\mu_n^*,\nu)=0$, which yields
\begin{align*}
  \sqrt{n}\big(\overline\xi_n - \E[\overline\xi_n]\big)/\hat{\overline{\sigma}} \longrightarrow N(0,1) ~~ {\rm in~distribution}.
\end{align*}

By Proposition~\ref{prop:bias}, we have
\begin{align*}
  \sqrt{n}\big(\overline\xi_n - \xi^{(n)}\big)/\hat{\overline{\sigma}} \longrightarrow N(0,1) ~~ {\rm in~distribution}.
\end{align*}

For a sequence of probability measures with $\xi^{(n)} = \kappa + n^{-1/2}h$, we have
\begin{align*}
  &\P(T=1) = \P(\overline\xi_n>\kappa+z_{1-\alpha}\hat{\overline\sigma}/\sqrt{n}) = \P(\overline\xi_n - \xi^{(n)} >\kappa - \xi^{(n)} +z_{1-\alpha}\hat{\overline\sigma}/\sqrt{n}) \\
  =& \P(\sqrt{n}(\overline\xi_n - \xi^{(n)})/\hat{\overline\sigma} > z_{1-\alpha} - \sqrt{n}(\xi^{(n)}-\kappa)/\hat{\overline\sigma}) = \P(\sqrt{n}(\overline\xi_n - \xi^{(n)})/\hat{\overline\sigma} > z_{1-\alpha} - h/\hat{\overline\sigma}).
\end{align*}

By the central limit theorem above, we have
\begin{align*}
  \lim_{n \to \infty} \P(T=1) = 1 - \phi(z_{1-\alpha} - h/\overline\sigma).
\end{align*}

\end{proof}
}

\subsection{Proof of Lemma~\ref{lemma:hayek1}}

\begin{proof}[Proof of Lemma~\ref{lemma:hayek1}]

We establish the two claims for $i=1,2,3,4$ seperately.

{\bf Part I.} $i=1$.

Since $[(X_i,Y_i)]_{i=1}^n$ are i.i.d., we have 
\begin{align*}
   \E [T_1] =& \E \Big[\frac{1}{n^3} \sum_{i=1}^n \Var \Big[ \min\big\{R_i, R_{N_1(i)}\big\} \Biggiven \mX \Big] \Big]\\
    =& \frac{1}{n^2} \E\Big[\Var \Big[ \min\big\{R_1, R_{N_1(1)}\big\} \Biggiven \mX \Big] \Big]\\
  =& \frac{1}{n^2} \E\Big[\Var \Big[ \sum_{k=1}^n \ind\big(Y_k \le Y_1 \wedge Y_{N_1(1)}\big) \Biggiven \mX \Big] \Big]\\
  =& \frac{(n-1)(n-2)}{n^2} \E\Big[\Cov \Big[ \ind\big(Y_2 \le Y_1 \wedge Y_{N_1(1)}\big), \ind\big(Y_3 \le Y_1 \wedge Y_{N_1(1)}\big) \Biggiven \mX \Big] \Big]\\
  & + \frac{1}{n^2} \E\Big[\Var \Big[ \ind\big(Y_1 \le Y_1 \wedge Y_{N_1(1)}\big) \Biggiven \mX \Big] \Big] + \frac{n-1}{n^2} \E\Big[\Var \Big[ \ind\big(Y_2 \le Y_1 \wedge Y_{N_1(1)}\big) \Biggiven \mX \Big] \Big]\\
  & + \frac{2(n-1)}{n^2} \E\Big[\Cov \Big[ \ind\big(Y_1 \le Y_1 \wedge Y_{N_1(1)}\big), \ind\big(Y_2 \le Y_1 \wedge Y_{N_1(1)}\big) \Biggiven \mX \Big] \Big]\\
  =: & \frac{(n-1)(n-2)}{n^2} \E\Big[\Cov \Big[ \ind\big(Y_2 \le Y_1 \wedge Y_{N_1(1)}\big), \ind\big(Y_3 \le Y_1 \wedge Y_{N_1(1)}\big) \Biggiven \mX \Big] \Big] + S_1\\
  = & (1+O(n^{-1})) \E\Big[\Cov \Big[ \ind\big(Y_2 \le Y_1 \wedge \tY_1 \big), \ind\big(Y_3 \le Y_1 \wedge \tY_1 \big) \Biggiven \mX \Big] \Big] + S_1\\
  & + (1+O(n^{-1})) \Big\{\E\Big[\Cov \Big[ \ind\big(Y_2 \le Y_1 \wedge Y_{N_1(1)}\big), \ind\big(Y_3 \le Y_1 \wedge Y_{N_1(1)}\big) \Biggiven \mX \Big] \Big] \\
  & - \E\Big[\Cov \Big[ \ind\big(Y_2 \le Y_1 \wedge \tY_1 \big), \ind\big(Y_3 \le Y_1 \wedge \tY_1 \big) \Biggiven \mX \Big] \Big] \Big\}\\
  =: & (1+O(n^{-1})) \E\Big[\Cov \Big[ \ind\big(Y_2 \le Y_1 \wedge \tY_1 \big), \ind\big(Y_3 \le Y_1 \wedge \tY_1 \big) \Biggiven \mX \Big] \Big] + S_1 + (1+O(n^{-1})) S_2,
  \yestag\label{eq:hayek2}
\end{align*}
where $\tY_1$ is sampled from $F_{Y \given X=X_1}$ independent of the data.

For $S_1$ in \eqref{eq:hayek2}, noticing that the variance of the indicator function is bounded by 1 and then invoking the Cauchy–Schwarz inequality yields
\begin{align}\label{eq:hayek3}
  \lvert S_1 \rvert \le \frac{3n-2}{n^2} = O(n^{-1}).
\end{align}

For $S_2$ in \eqref{eq:hayek2}, we first have 
\begin{align*}
  & \E \Big[\ind\big(Y_2 \le Y_1 \wedge Y_{N_1(1)}\big) \ind\big(Y_3 \le Y_1 \wedge Y_{N_1(1)}\big) \Biggiven \mX \Big]\\
  =& \int \ind\big(y_2 \le y_1 \wedge y_4\big) \ind\big(y_3 \le y_1 \wedge y_4\big) \d \mu_{X_1}(y_1) \d  \mu_{X_2}(y_2) \d  \mu_{X_3}(y_3) \d  \mu_{X_{N_1(1)}}(y_4) \ind(N_1(1) \neq 2,3)\\
  & + \int \ind\big(y_2 \le y_1 \wedge y_2\big) \ind\big(y_3 \le y_1 \wedge y_2\big) \d \mu_{X_1}(y_1) \d  \mu_{X_2}(y_2) \d  \mu_{X_3}(y_3) \ind(N_1(1) = 2)\\
  & + \int \ind\big(y_2 \le y_1 \wedge y_3\big) \ind\big(y_3 \le y_1 \wedge y_3\big) \d \mu_{X_1}(y_1) \d  \mu_{X_2}(y_2) \d  \mu_{X_3}(y_3) \ind(N_1(1) = 3)\\
  =:& \int \ind\big(y_2 \le y_1 \wedge y_4\big) \ind\big(y_3 \le y_1 \wedge y_4\big) \d \mu_{X_1}(y_1) \d  \mu_{X_2}(y_2) \d  \mu_{X_3}(y_3) \d  \mu_{X_{N_1(1)}}(y_4) + Q_1\\
  = & \int G_{X_1}(y_2 \vee y_3) G_{X_{N_1(1)}}(y_2 \vee y_3) \d \mu_{X_2}(y_2) \d  \mu_{X_3}(y_3) + Q_1.
\end{align*}

From the boundedness of the indicator function and $\P(N_1(1)=2) =\P(N_1(1)=3) = 1/(n-1)$, we then have $\E[\lvert Q_1 \rvert] = O(n^{-1})$.

We can establish in the same way that
\begin{align*}
  & \E \Big[\ind\big(Y_2 \le Y_1 \wedge Y_{N_1(1)}\big) \Biggiven \mX \Big] \E \Big[\ind\big(Y_3 \le Y_1 \wedge Y_{N_1(1)}\big) \Biggiven \mX \Big]\\
  =& \int \ind\big(y_2 \le y_1 \wedge y_4\big) \ind\big(y_3 \le y_5 \wedge y_6\big) \d \mu_{X_1}(y_1) \d  \mu_{X_2}(y_2) \d \mu_{X_3}(y_3) \d \mu_{X_{N_1(1)}}(y_4) \d \mu_{X_1}(y_5) \d \mu_{X_{N_1(1)}}(y_6) + Q_2\\
  =& \int G_{X_1}(y_2) G_{X_{N_1(1)}}(y_2) G_{X_1}(y_3) G_{X_{N_1(1)}}(y_3) \d \mu_{X_2}(y_2) \d  \mu_{X_3}(y_3) + Q_2,
\end{align*}
with $\E[\lvert Q_2 \rvert] = O(n^{-1})$.

On the other hand,
\begin{align*}
  & \E \Big[ \ind\big(Y_2 \le Y_1 \wedge \tY_1 \big) \ind\big(Y_3 \le Y_1 \wedge \tY_1 \big) \Biggiven \mX \Big]\\
  = & \int \ind\big(y_2 \le y_1 \wedge y_4\big) \ind\big(y_3 \le y_1 \wedge y_4\big) \d \mu_{X_1}(y_1) \d  \mu_{X_2}(y_2) \d  \mu_{X_3}(y_3) \d  \mu_{X_1}(y_4)\\
  = & \int G^2_{X_1}(y_2 \vee y_3) \d \mu_{X_2}(y_2) \d  \mu_{X_3}(y_3),
\end{align*}
and
\begin{align*}
  & \E \Big[\ind\big(Y_2 \le Y_1 \wedge \tY_1\big) \Biggiven \mX \Big] \E \Big[\ind\big(Y_3 \le Y_1 \wedge \tY_1\big) \Biggiven \mX \Big]\\
  =& \int \ind\big(y_2 \le y_1 \wedge y_4\big) \ind\big(y_3 \le y_5 \wedge y_6\big) \d \mu_{X_1}(y_1) \d  \mu_{X_2}(y_2) \d \mu_{X_3}(y_3) \d \mu_{X_1}(y_4) \d \mu_{X_1}(y_5) \d \mu_{X_1}(y_6)\\
  = & \int G^2_{X_1}(y_2) G^2_{X_1}(y_3) \d \mu_{X_2}(y_2) \d  \mu_{X_3}(y_3).
\end{align*}

Then, since $G_x$ is uniformly bounded by 1 for any $x \in \bR^d$,
\begin{align*}
  & \Big\lvert \E \Big[\ind\big(Y_2 \le Y_1 \wedge Y_{N_1(1)}\big) \ind\big(Y_3 \le Y_1 \wedge Y_{N_1(1)}\big) \Biggiven \mX \Big] - \E \Big[\ind\big(Y_2 \le Y_1 \wedge \tY_1 \big) \ind\big(Y_3 \le Y_1 \wedge \tY_1 \big) \Biggiven \mX \Big] \Big\rvert\\
  = & \Big\lvert \int G_{X_1}(y_2 \vee y_3) \Big(G_{X_{N_1(1)}}(y_2 \vee y_3) - G_{X_1}(y_2 \vee y_3) \Big) \d \mu_{X_2}(y_2) \d  \mu_{X_3}(y_3) + Q_1 \Big\rvert\\
  \le & \int \Big\lvert G_{X_{N_1(1)}}(y_2 \vee y_3) - G_{X_1}(y_2 \vee y_3) \Big\rvert \d \mu_{X_2}(y_2) \d  \mu_{X_3}(y_3) + \lvert Q_1 \rvert,
\end{align*}
and
\begin{align*}
  & \Big\lvert \E \Big[\ind\big(Y_2 \le Y_1 \wedge Y_{N_1(1)}\big) \Biggiven \mX \Big] \E \Big[\ind\big(Y_3 \le Y_1 \wedge Y_{N_1(1)}\big) \Biggiven \mX \Big]  - \E \Big[\ind\big(Y_2 \le Y_1 \wedge \tY_1 \big) \Biggiven \mX \Big] \E \Big[\ind\big(Y_3 \le Y_1 \wedge \tY_1 \big) \Biggiven \mX \Big] \Big\rvert\\
  = & \Big\lvert \int G_{X_1}(y_2) G_{X_1}(y_3) \Big( G_{X_{N_1(1)}}(y_2) G_{X_{N_1(1)}}(y_3) - G_{X_1}(y_2) G_{X_1}(y_3) \Big) \d \mu_{X_2}(y_2) \d  \mu_{X_3}(y_3) + Q_2 \Big\rvert\\
  \le & \int \Big\lvert G_{X_{N_1(1)}}(y_2) - G_{X_1}(y_2) \Big\rvert \d \mu_{X_2}(y_2) \d  \mu_{X_3}(y_3) + \int \Big\lvert G_{X_{N_1(1)}}(y_3) - G_{X_1}(y_3) \Big\rvert \d \mu_{X_2}(y_2) \d \mu_{X_3}(y_3) + \lvert Q_2 \rvert.
\end{align*}

We then have
\begin{align*}
  \lvert S_2 \rvert = & \Big\lvert \E\Big[\Cov \Big[ \ind\big(Y_2 \le Y_1 \wedge Y_{N_1(1)}\big), \ind\big(Y_3 \le Y_1 \wedge Y_{N_1(1)}\big) \Biggiven \mX \Big] \Big] \\
  & - \E\Big[\Cov \Big[ \ind\big(Y_2 \le Y_1 \wedge \tY_1 \big), \ind\big(Y_3 \le Y_1 \wedge \tY_1 \big) \Biggiven \mX \Big] \Big] \Big\rvert\\
  \le & \E \Big[\int \Big\lvert G_{X_{N_1(1)}}(y_2 \vee y_3) - G_{X_1}(y_2 \vee y_3) \Big\rvert \d \mu_{X_2}(y_2) \d  \mu_{X_3}(y_3) \Big] \\
  & + 2 \E \Big[\int \Big\lvert G_{X_{N_1(1)}}(y_2) - G_{X_1}(y_2) \Big\rvert \d \mu_{X_2}(y_2) \d  \mu_{X_3}(y_3) \Big] + \E[\lvert Q_1 \rvert] + \E[\lvert Q_2 \rvert].
\end{align*}

For the first term above, since $G_x$ is uniformly bounded by 1 for $x \in \bR^d$, we have
\begin{align*}
  \int \Big\lvert G_{X_{N_1(1)}}(y_2 \vee y_3) - G_{X_1}(y_2 \vee y_3) \Big\rvert \d \mu_{X_2}(y_2) \d  \mu_{X_3}(y_3) \le 2\int  \mu_{X_2}(y_2) \d  \mu_{X_3}(y_3) = 2.
\end{align*}
Invoking Fatou's lemma then yields
\begin{align*}
  & \limsup_{n \to \infty} \E \Big[\int \Big\lvert G_{X_{N_1(1)}}(y_2 \vee y_3) - G_{X_1}(y_2 \vee y_3) \Big\rvert \d \mu_{X_2}(y_2) \d  \mu_{X_3}(y_3) \Big]\\ 
  = &  \limsup_{n \to \infty} \E \Big[ \E \Big[ \int \Big\lvert G_{X_{N_1(1)}}(y_2 \vee y_3) - G_{X_1}(y_2 \vee y_3) \Big\rvert \d \mu_{X_2}(y_2) \d  \mu_{X_3}(y_3) \Biggiven X_2,X_3\Big] \Big]\\
  = & \limsup_{n \to \infty} \E \Big[  \int \E \Big[ \Big\lvert G_{X_{N_1(1)}}(y_2 \vee y_3) - G_{X_1}(y_2 \vee y_3) \Big\rvert \Biggiven X_2,X_3\Big] \d \mu_{X_2}(y_2) \d  \mu_{X_3}(y_3)  \Big]\\
  \le & \E \Big[  \int  \limsup_{n \to \infty} \E \Big[ \Big\lvert G_{X_{N_1(1)}}(y_2 \vee y_3) - G_{X_1}(y_2 \vee y_3) \Big\rvert \Biggiven X_2,X_3\Big] \d \mu_{X_2}(y_2) \d  \mu_{X_3}(y_3)  \Big].
\end{align*}

Notice that for any $t \in \bR$, the map $x \to G_x(t)$ is a measurable function. Then from Lemma 11.7 in \cite{azadkia2019simple}, $G_{X_{N_1(1)}}(t) - G_{X_1}(t)\stackrel{\sf p}{\to}0$. Then for all $t \in \bR$ and almost all $X_2,X_3 \in \bR^d$, 
\begin{align*}
  \limsup_{n \to \infty} \E \Big[ \Big\lvert G_{X_{N_1(1)}}(t) - G_{X_1}(t) \Big\rvert \Biggiven X_2,X_3\Big] = 0,
\end{align*}
and accordingly
\[
  \lim_{n \to \infty} \E \Big[\int \Big\lvert G_{X_{N_1(1)}}(y_2 \vee y_3) - G_{X_1}(y_2 \vee y_3) \Big\rvert \d \mu_{X_2}(y_2) \d  \mu_{X_3}(y_3) \Big] = 0.
\]

We can handle the second term in the upper bound of $\lvert S_2 \rvert$ in the same way. Recall that $\E[\lvert Q_1 \rvert],\E[\lvert Q_2 \rvert] = O(n^{-1})$. We then obtain
\begin{align}\label{eq:hayek4}
  \lvert S_2 \rvert = o(1).
\end{align}

In the end, let's study the first term in \eqref{eq:hayek2}. Notice that
\begin{align*}
  & \E\Big[\E \Big[ \ind\big(Y_2 \le Y_1 \wedge \tY_1 \big) \ind\big(Y_3 \le Y_1 \wedge \tY_1 \big) \Biggiven \mX \Big] \Big]\\
  = & \E\Big[ \int \ind\big(y_2 \le y_1 \wedge y_4\big) \ind\big(y_3 \le y_1 \wedge y_4\big) \d \mu_{X_1}(y_1) \d  \mu_{X_2}(y_2) \d  \mu_{X_3}(y_3) \d  \mu_{X_1}(y_4) \Big]\\
  = & \E\Big[ \E\Big[\int \ind\big(y_2 \le y_1 \wedge y_4\big) \ind\big(y_3 \le y_1 \wedge y_4\big) \d \mu_{X_1}(y_1) \d  \mu_{X_2}(y_2) \d  \mu_{X_3}(y_3) \d  \mu_{X_1}(y_4) \Biggiven X_1 \Big]\Big]\\
  = & \E\Big[ \int \ind\big(y_2 \le y_1 \wedge y_4\big) \ind\big(y_3 \le y_1 \wedge y_4\big) \d \mu_{X_1}(y_1) \d  \mu(y_2) \d  \mu(y_3) \d  \mu_{X_1}(y_4) \Big]\\
  = & \E\Big[ \int F_Y^2\big(y_1 \wedge y_4\big)  \d \mu_{X_1}(y_1) \d \mu_{X_1}(y_4) \Big] = \E\Big[\E \Big[ F_Y^2\big(Y_1 \wedge \tY_1 \big) \Biggiven X_1 \Big] \Big].
\end{align*}

We can establish
\begin{align*}
  \E\Big[\E \Big[ \ind\big(Y_2 \le Y_1 \wedge \tY_1 \big) \Biggiven \mX \Big] \E \Big[ \ind\big(Y_3 \le Y_1 \wedge \tY_1 \big) \Biggiven \mX \Big] \Big] = \E\Big[ \Big(\E \Big[ F_Y\big(Y_1 \wedge \tY_1 \big) \Biggiven X_1 \Big] \Big)^2 \Big].
\end{align*}

Then
\begin{align}\label{eq:hayek5}
  \E\Big[\Cov \Big[ \ind\big(Y_2 \le Y_1 \wedge \tY_1 \big), \ind\big(Y_3 \le Y_1 \wedge \tY_1 \big) \Biggiven \mX \Big] \Big] = \E\Big[\Var \Big[ F_Y\big(Y_1 \wedge \tY_1 \big) \Biggiven X_1 \Big] \Big].
\end{align}

Plugging \eqref{eq:hayek3}-\eqref{eq:hayek5} to \eqref{eq:hayek2} yields
\begin{align*}
  \E [T_1] = (1+O(n^{-1})) \E\Big[\Var \Big[ F_Y\big(Y_1 \wedge \tY_1 \big) \Biggiven X_1 \Big] \Big] + o(1).
\end{align*}

Similar to \eqref{eq:hayek4}, we also have
\begin{align*}
   \E[T_1^*] =& \E\Big[ \frac{1}{n} \sum_{i=1}^n \Var \Big[ \min\big\{F_Y(Y_i), F_Y(Y_{N_1(i)})\big\} \Biggiven \mX \Big] \Big]\\
  = & \E \Big[\Var \Big[ F_Y(Y_1) \wedge F_Y(Y_{N_1(1)}) \Biggiven \mX \Big] \Big] = \E \Big[ \Var \Big[ F_Y\big(Y_1 \wedge Y_{N_1(1)}\big) \Biggiven \mX \Big] \Big] \\
  = & \E \Big[ \Var \Big[ F_Y\big(Y_1 \wedge \tY_1\big) \Biggiven \mX \Big] \Big] + o(1) \\
  = & \E \Big[ \Var \Big[ F_Y\big(Y_1 \wedge \tY_1\big) \Biggiven X_1 \Big] \Big] + o(1).
\end{align*}
Using the fact that $F_Y\leq 1$, we complete the proof of the first claim, and the second claim can be established in the same way.

\vspace{0.2cm}

{\bf Part II.} $i=2$.

Since $[(X_i,Y_i)]_{i=1}^n$ are i.i.d. and the indicator function is bounded, we have
\begin{align*}
   \E[T_2] =& \frac{1}{n^3} \E \Big[\sum_{\substack{j=N_1(i),i \neq N_1(j)\\{\rm or}~i=N_1(j),j \neq N_1(i)}} \Cov\Big[\min\big\{R_i, R_{N_1(i)}\big\}, \min\big\{R_j, R_{N_1(j)}\big\} \Biggiven \mX \Big] \Big]\\
  = & \frac{2}{n^3} \E \Big[\sum_{j=N_1(i),i \neq N_1(j)} \Cov\Big[\min\big\{R_i, R_{N_1(i)}\big\}, \min\big\{R_j, R_{N_1(j)}\big\} \Biggiven \mX \Big] \Big]\\
  = & \frac{2}{n^3} \E \Big[\sum_{i=1}^n \Cov\Big[\min\big\{R_i, R_{N_1(i)}\big\}, \min\big\{R_{N_1(i)}, R_{N_1(N_1(i))}\big\} \Biggiven \mX \Big] \ind \Big( i \neq N_1(N_1(i)) \Big) \Big]\\
  = & \frac{2}{n^2} \E \Big[\Cov\Big[\min\big\{R_1, R_{N_1(1)}\big\}, \min\big\{R_{N_1(1)}, R_{N_1(N_1(1))}\big\} \Biggiven \mX \Big] \ind \Big( 1 \neq N_1(N_1(1)) \Big) \Big]\\
  = & \frac{2}{n^2} \E \Big[\Cov\Big[\sum_{k=1}^n \ind\big(Y_k \le Y_1 \wedge Y_{N_1(1)}\big), \sum_{\ell=1}^n \ind\big(Y_\ell \le Y_{N_1(1)} \wedge Y_{N_1(N_1(1))}\big)  \Biggiven \mX \Big] \ind \Big( 1 \neq N_1(N_1(1)) \Big) \Big]\\
  = & \frac{2(n-1)(n-2)}{n^2} \E \Big[\Cov\Big[\ind\big(Y_2 \le Y_1 \wedge Y_{N_1(1)}\big), \ind\big(Y_3 \le Y_{N_1(1)} \wedge Y_{N_1(N_1(1))}\big) \Biggiven \mX \Big] \ind \Big( 1 \neq N_1(N_1(1)) \Big) \Big] + O(n^{-1}).
\end{align*}

Lemma 11.3 in \cite{azadkia2019simple} shows $X_{N_1(1)} \to X_1$ almost surely. Notice that 
\[
\lVert X_{N_1(N_1(1))} - X_{1} \rVert \le \lVert X_{N_1(1)} - X_{1} \rVert + \lVert X_{N_1(N_1(1))} - X_{N_1(1)} \rVert \le 2\lVert X_{N_1(1)} - X_{1} \rVert. 
\]
Then $X_{N_1(N_1(1))} \to X_1$ almost surely. Similar to the proof of Lemma 11.7 in \cite{azadkia2019simple}, for any $t \in \bR$, one can prove
\[
G_{X_{N_1(N_1(1))}}(t) - G_{X_1}(t) \stackrel{\sf p}{\longrightarrow} 0.
\]
Notice that $\P(N_1(1)=2,3) = 2/(n-1)$ and $\P(N_1(N_1(1))=2,3) \le 2/(n-1)$. Then, similar to the proof of \eqref{eq:hayek4},
\begin{align*}
  & \E \Big[\Cov\Big[\ind\big(Y_2 \le Y_1 \wedge Y_{N_1(1)}\big), \ind\big(Y_3 \le Y_{N_1(1)} \wedge Y_{N_1(N_1(1))}\big) \Biggiven \mX \Big] \ind \Big( 1 \neq N_1(N_1(1)) \Big) \Big]\\
  = & \E \Big[\Cov\Big[\ind\big(Y_2 \le Y_1 \wedge \tY_1\big), \ind\big(Y_3 \le \tY_1 \wedge \tY_1' \big) \Biggiven \mX \Big] \ind \Big( 1 \neq N_1(N_1(1)) \Big) \Big] + o(1).
\end{align*}

Let $\mX^{-2,3}:=\mX \setminus \{X_2,X_3\}$, and let $N^{-2,3}(j)$ index the NN of $X_j$ in $\mX^{-2,3}$ for $j \in \zahl{n}$ and $j \neq 2,3$. If $N_1(1) \neq 2,3$ and $N_1(N_1(1)) \neq 2,3$, then $N_1(1) = N^{-2,3}(1)$ and $N_1(N_1(1)) = N^{-2,3}(N_1(1))$. Then $N^{-2,3}(N^{-2,3}(1)) = N_1(N_1(1))$. Notice that $\P(N_1(1)=2,3), \P(N_1(N_1(1))=2,3) = O(n^{-1})$ and the event $\{1 \neq N^{-2,3}(N^{-2,3}(1))\}$ is a function of $\mX^{-2,3}$. From the boundedness of the indicator function and $F_Y$,
\begin{align*}
  & \E \Big[\Cov\Big[\ind\big(Y_2 \le Y_1 \wedge \tY_1\big), \ind\big(Y_3 \le \tY_1 \wedge \tY_1' \big) \Biggiven \mX \Big] \ind \Big( 1 \neq N_1(N_1(1)) \Big) \Big]\\
  = & \E \Big[\Cov\Big[\ind\big(Y_2 \le Y_1 \wedge \tY_1\big), \ind\big(Y_3 \le \tY_1 \wedge \tY_1' \big) \Biggiven \mX \Big] \ind \Big( 1 \neq N^{-2,3}(N^{-2,3}(1)) \Big) \Big] + O(n^{-1})\\
  = & \E \Big[\E \Big[\Cov\Big[\ind\big(Y_2 \le Y_1 \wedge \tY_1\big), \ind\big(Y_3 \le \tY_1 \wedge \tY_1' \big) \Biggiven \mX \Big] \Biggiven \mX^{-2,3} \Big] \ind \Big( 1 \neq N^{-2,3}(N^{-2,3}(1)) \Big) \Big] + O(n^{-1})\\
  = & \E \Big[\Cov\Big[F_Y\big(Y_1 \wedge \tY_1\big), F_Y\big(\tY_1 \wedge \tY_1' \big) \Biggiven X_1 \Big] \ind \Big( 1 \neq N^{-2,3}(N^{-2,3}(1)) \Big) \Big] + O(n^{-1})\\
  = & \E \Big[\Cov\Big[F_Y\big(Y_1 \wedge \tY_1\big), F_Y\big(\tY_1 \wedge \tY_1' \big) \Biggiven X_1 \Big] \ind \Big( 1 \neq N_1(N_1(1)) \Big) \Big] + O(n^{-1}).
\end{align*}

We then obtain
\begin{align*}
  \E[T_2] = 2(1+O(n^{-1})) \E \Big[\Cov\Big[F_Y\big(Y_1 \wedge \tY_1\big), F_Y\big(\tY_1 \wedge \tY_1' \big) \Biggiven X_1 \Big] \ind \Big( 1 \neq N_1(N_1(1)) \Big) \Big] + o(1).
\end{align*}

For $T_2^*$, we have
\begin{align*}
  \E[T_2^*] =& \frac{1}{n} \E \Big[ \sum_{\substack{j=N_1(i),i \neq N_1(j)\\{\rm or}~i=N_1(j),j \neq N_1(i)}} \Cov\Big[\min\big\{F_Y(Y_i), F_Y(Y_{N_1(i)})\big\}, \min\big\{F_Y(Y_j), F_Y(Y_{N_1(j)})\big\} \Biggiven \mX \Big] \Big]\\
  = & \frac{2}{n} \E \Big[\sum_{j=N_1(i),i \neq N_1(j)} \Cov\Big[\min\big\{F_Y(Y_i), F_Y(Y_{N_1(i)})\big\}, \min\big\{F_Y(Y_j), F_Y(Y_{N_1(j)})\big\} \Biggiven \mX \Big] \Big]\\
  = & 2 \E \Big[\Cov\Big[F_Y\big(Y_1 \wedge Y_{N_1(1)}\big), F_Y\big(Y_{N_1(1)} \wedge Y_{N_1(N_1(1))}\big) \Biggiven \mX \Big] \ind \Big( 1 \neq N_1(N_1(1)) \Big) \Big]\\
  = & 2\E \Big[\Cov\Big[F_Y\big(Y_1 \wedge \tY_1\big), F_Y\big(\tY_1 \wedge \tY_1' \big) \Biggiven X_1 \Big] \ind \Big( 1 \neq N_1(N_1(1)) \Big) \Big] + o(1).
\end{align*}

From the boundedness of $F_Y$, we complete the proof of the first claim.

The second claim can be established in the same way. Both claims for $i=4$ can be established in the same way by replacing the event $\{1 \neq N_1(N_1(1))\}$ by $\{1 = N_1(N_1(1))\}$. We can obtain
\begin{align*}
  \E[T_4] = (1+O(n^{-1})) \E \Big[\Var \Big[ F_Y\big(Y_1 \wedge \tY_1\big) \Biggiven X_1 \Big]  \ind \Big( 1 = N_1(N_1(1)) \Big) \Big] + o(1),
\end{align*}
and
\begin{align*}
  \E[T_4^*] = \E \Big[\Var \Big[ F_Y\big(Y_1 \wedge \tY_1\big) \Biggiven X_1 \Big] \ind \Big( 1 = N_1(N_1(1)) \Big) \Big] + o(1).
\end{align*}

\vspace{0.2cm}

{\bf Part III.} $i=3$.

Conditional on $\mX$, let $A_1 = A_1(\mX) := \{j: j \neq 1, N_1(j) = N_1(1)\}$
, i.e., the set of all indices $j$ such that $X_j$ and $X_1$ share the same NN. Let $\pi(1)$ be the random variable that assigns the same probability mass on the elements of $A_1$, and are independent of $\mY$ conditional on $\mX$, i.e., for any $j \in A_1$, $\P(\pi(1) = j) = 1/\lvert A_1 \rvert$. Then
\begin{align*}
  & \E [T_3] = \frac{1}{n^3} \E \Big[\sum_{\substack{i \neq j \\ N_1(i) = N_1(j)}} \Cov\Big[\min\big\{R_i, R_{N_1(i)}\big\}, \min\big\{R_j, R_{N_1(j)}\big\} \Biggiven \mX \Big] \Big]\\
  = & \frac{1}{n^3} \E \Big[\sum_{i=1}^n \sum_{j:j\neq i, N_1(i) = N_1(j)} \Cov\Big[\min\big\{R_i, R_{N_1(i)}\big\}, \min\big\{R_j, R_{N_1(j)}\big\} \Biggiven \mX \Big] \Big]\\
  = & \frac{1}{n^2} \E \Big[\sum_{j \in A_1} \Cov\Big[\min\big\{R_1, R_{N_1(1)}\big\}, \min\big\{R_j, R_{N_1(j)}\big\} \Biggiven \mX \Big] \Big]\\
  = & \frac{1}{n^2} \E \Big[\lvert A_1 \rvert \Cov\Big[\min\big\{R_1, R_{N_1(1)}\big\}, \min\big\{R_{\pi(1)}, R_{N_1(1)}\big\} \Biggiven \mX \Big] \Big]\\
  = & \frac{(n-1)(n-2)}{n^2} \E \Big[\lvert A_1 \rvert \Cov\Big[\ind\big(Y_2 \le Y_1 \wedge Y_{N_1(1)}\big), \ind\big(Y_3 \le Y_{\pi(1)} \wedge Y_{N_1(1)}\big) \Biggiven \mX \Big] \Big] + O\Big(\frac{\E[\lvert A_1 \rvert]}{n}\Big).
\end{align*}

From Lemma 20.6 together with Theorem 20.16 in \cite{biau2015lectures}, $\lvert A_1 \rvert$ is always bounded by a constant only depending on $d$. Then
\begin{align*}
  \E [T_3] = (1+O(n^{-1})) \E \Big[\lvert A_1 \rvert \Cov\Big[\ind\big(Y_2 \le Y_1 \wedge \tY_1\big), \ind\big(Y_3 \le \tY_1' \wedge \tY_1\big) \Biggiven \mX \Big] \Big] + o(1).
\end{align*}

Recall the definition of $\mX^{-2,3}$ and $N^{-2,3}(\cdot)$ in the second part. Let 
\[
  A_1^{-2,3} = A_1^{-2,3}(\mX^{-2,3}) := \{j: j \neq 1, N^{-2,3}(j) = N^{-2,3}(1)\}. 
\]
We consider the event $N_1(1) \neq 2,3$. For any $j \in A_1$, we have $j \neq 1, N_1(j) = N_1(1)$. If $j \neq 2,3$, then $N^{-2,3}(j) = N^{-2,3}(1)$ from $N_1(1) \neq 2,3$, and then $j \in A_1^{-2,3}$. Then
\[
  \lvert A_1 \setminus A_1^{-2,3} \rvert \le \ind(N_1(2) = N_1(1)) + \ind(N_1(3) = N_1(1)).
\]

On the other hand, for any $j \in A_1^{-2,3}$, we have $N^{-2,3}(j) = N^{-2,3}(1) = N_1(1)$. If $N_1(j) \neq N_1(1)$, then the possible case is $N_1(j)=2,3, N_1(N_1(j))=2,3, N_1(1) = N_1(N_1(N_1(j)))$, or $N_1(j)=2,3, N_1(N_1(j))\neq2,3, N_1(1) = N_1(N_1(j))$. Then
\begin{align*}
   \lvert A_1^{-2,3} \setminus A_1 \rvert \le& \sum_{j:N_1(j)=2,3} \Big(\ind(N_1(1) = N_1(N_1(N_1(j)))) + \ind(N_1(1) = N_1(N_1(j))) \Big)\\
  \le & \sum_{j:N_1(j)=2} \Big(\ind(N_1(1) = N_1(N_1(2))) + \ind(N_1(1) = N_1(2)) \Big) \\
  & + \sum_{j:N_1(j)=3} \Big(\ind(N_1(1) = N_1(N_1(3))) + \ind(N_1(1) = N_1(3)) \Big).
\end{align*}

Notice that for any $i \in \zahl{n}$, the number of $j \in \zahl{n}$ such that $N_1(j) = i$ is always bounded by a constant depending only on $d$. Then $\E[\lvert A_1 \setminus A_1^{-2,3} \rvert], \E[\lvert A_1^{-2,3} \setminus A_1 \rvert] = O(n^{-1})$. Notice that $\P(N_1(1)=2,3) = O(n^{-1})$. Then
\begin{align*}
  &\E \Big[\lvert A_1 \rvert \Cov\Big[\ind\big(Y_2 \le Y_1 \wedge \tY_1\big), \ind\big(Y_3 \le \tY_1' \wedge \tY_1\big) \Biggiven \mX \Big] \Big] \\
  = & \E \Big[\lvert A_1^{-2,3} \rvert \Cov\Big[\ind\big(Y_2 \le Y_1 \wedge \tY_1\big), \ind\big(Y_3 \le \tY_1' \wedge \tY_1\big) \Biggiven \mX \Big] \Big] + O(n^{-1})\\
  = & \E \Big[\lvert A_1^{-2,3} \rvert \E \Big[\Cov\Big[\ind\big(Y_2 \le Y_1 \wedge \tY_1\big), \ind\big(Y_3 \le \tY_1 \wedge \tY_1' \big) \Biggiven \mX \Big] \Biggiven \mX^{-2,3} \Big] \Big] + O(n^{-1})\\
  = & \E \Big[\lvert A_1^{-2,3} \rvert \Cov\Big[F_Y\big(Y_1 \wedge \tY_1\big), F_Y\big(\tY_1 \wedge \tY_1' \big) \Biggiven X_1 \Big] \Big] + O(n^{-1})\\
  = & \E \Big[\lvert A_1 \rvert \Cov\Big[F_Y\big(Y_1 \wedge \tY_1\big), F_Y\big(\tY_1 \wedge \tY_1' \big) \Biggiven X_1 \Big] \Big] + O(n^{-1}).
\end{align*}

We then obtain
\begin{align*}
  \E [T_3] = (1+O(n^{-1})) \E \Big[\lvert A_1 \rvert \Cov\Big[F_Y\big(Y_1 \wedge \tY_1\big), F_Y\big(\tY_1 \wedge \tY_1' \big) \Biggiven X_1 \Big] \Big] + o(1).
\end{align*}

For $T_3^*$,
\begin{align*}
   \E [T_3^*] =& \frac{1}{n} \E \Big[ \sum_{\substack{i \neq j \\ N_1(i) = N_1(j)}} \Cov\Big[\min\big\{F_Y(Y_i), F_Y(Y_{N_1(i)})\big\}, \min\big\{F_Y(Y_j), F_Y(Y_{N_1(j)})\big\} \Biggiven \mX \Big] \Big]\\
  = & \E \Big[\lvert A_1 \rvert \Cov\Big[F_Y\big(Y_1 \wedge Y_{N_1(1)}\big), F_Y\big(Y_{\pi(1)} \wedge Y_{N_1(1)}\big) \Biggiven \mX \Big] \Big]\\
  = &  \E \Big[\lvert A_1 \rvert \Cov\Big[F_Y\big(Y_1 \wedge \tY_1\big), F_Y\big(\tY_1 \wedge \tY_1' \big) \Biggiven X_1 \Big] \Big] + o(1).
\end{align*}

Then we complete the proof of the first claim and the second claim can be similarly derived.
\end{proof}

\subsection{Proof of Lemma~\ref{lemma:hayek2}}

\begin{proof}[Proof of Lemma~\ref{lemma:hayek2}]
Since $[(X_i,Y_i)]_{i=1}^n$ are i.i.d. and $\min\big\{R_i, R_{N_1(i)}\big\}=\sum_{k=1}^n \ind\big(Y_k \le Y_i \wedge Y_{N_1(i)}\big)$ for any $i \in \zahl{n}$, we have
\begin{align*}
  &\E[T_5] = \E \Big[\frac{1}{n^3} \sum_{i,j,N_1(i),N_1(j)~{\rm distinct}} \Cov\Big[\min\big\{R_i, R_{N_1(i)}\big\}, \min\big\{R_j, R_{N_1(j)}\big\} \Biggiven \mX \Big] \Big]\\
  =& \frac{N_1(N-1)}{n^3}  \E \Big[\Cov\Big[\min\big\{R_1, R_{N_1(1)}\big\}, \min\big\{R_2, R_{N_1(2)}\big\} \Biggiven \mX \Big] \ind\Big(1,2,N_1(1),N_1(2)~{\rm distinct}\Big)\Big]\\
  =& \frac{n-1}{n^2}  \E \Big[\Cov\Big[\sum_{k=1}^n \ind\big(Y_k \le Y_1 \wedge Y_{N_1(1)}\big), \sum_{\ell=1}^n \ind\big(Y_\ell \le Y_2 \wedge Y_{N_1(2)}\big) \Biggiven \mX \Big] \ind\Big(1,2,N_1(1),N_1(2)~{\rm distinct}\Big)\Big].
\end{align*}
Notice that for $k,\ell \neq 1,2,N_1(1),N_1(2)$ and $k \neq \ell$, under the event $\{1,2,N_1(1),N_1(2)~{\rm distinct}\}$, we have
\[
  \Cov\Big[\ind\big(Y_k \le Y_1 \wedge Y_{N_1(1)}\big), \ind\big(Y_\ell \le Y_2 \wedge Y_{N_1(2)}\big) \Biggiven \mX \Big] = 0.
\]

Then by the symmetry,
\begin{align*}
  & \E \Big[\Cov\Big[\sum_{k=1}^n \ind\big(Y_k \le Y_1 \wedge Y_{N_1(1)}\big), \sum_{\ell=1}^n \ind\big(Y_\ell \le Y_2 \wedge Y_{N_1(2)}\big) \Biggiven \mX \Big] \ind\Big(1,2,N_1(1),N_1(2)~{\rm distinct}\Big)\Big]\\
  =& (n-2) \Big\{ \E \Big[\Cov\Big[\ind\big(Y_3 \le Y_1 \wedge Y_{N_1(1)}\big), \ind\big(Y_3 \le Y_2 \wedge Y_{N_1(2)}\big) \Biggiven \mX \Big] \ind\Big(1,2,N_1(1),N_1(2)~{\rm distinct}\Big)\Big]\\
  &+ 2\E \Big[\Cov\Big[\ind\big(Y_1 \le Y_1 \wedge Y_{N_1(1)}\big), \ind\big(Y_3 \le Y_2 \wedge Y_{N_1(2)}\big) \Biggiven \mX \Big] \ind\Big(1,2,N_1(1),N_1(2)~{\rm distinct}\Big)\Big]\\
  &+ 2\E \Big[\Cov\Big[\ind\big(Y_{N_1(1)} \le Y_1 \wedge Y_{N_1(1)}\big), \ind\big(Y_3 \le Y_2 \wedge Y_{N_1(2)}\big) \Biggiven \mX \Big] \ind\Big(1,2,N_1(1),N_1(2)~{\rm distinct}\Big)\Big]\\
  &+ 2\E \Big[\Cov\Big[\ind\big(Y_2 \le Y_1 \wedge Y_{N_1(1)}\big), \ind\big(Y_3 \le Y_2 \wedge Y_{N_1(2)}\big) \Biggiven \mX \Big] \ind\Big(1,2,N_1(1),N_1(2)~{\rm distinct}\Big)\Big]\\
  &+ 2\E \Big[\Cov\Big[\ind\big(Y_{N_1(2)} \le Y_1 \wedge Y_{N_1(1)}\big), \ind\big(Y_3 \le Y_2 \wedge Y_{N_1(2)}\big) \Biggiven \mX \Big] \ind\Big(1,2,N_1(1),N_1(2)~{\rm distinct}\Big)\Big] \Big\} + O(1),
\end{align*}
where $O(1)$ is from the boundedness of the indicator function and the number of the remaining terms and the overlap terms is $O(1)$.

Noticing $\P(1,2,N_1(1),N_1(2)~{\rm distinct}) = 1-O(n^{-1})$, we have
\begin{align*}
  & \E \Big[\Cov\Big[\ind\big(Y_3 \le Y_1 \wedge Y_{N_1(1)}\big), \ind\big(Y_3 \le Y_2 \wedge Y_{N_1(2)}\big) \Biggiven \mX \Big] \ind\Big(1,2,N_1(1),N_1(2)~{\rm distinct}\Big)\Big]\\
  = & \E \Big[\Cov\Big[\ind\big(Y_3 \le Y_1 \wedge \tY_1\big), \ind\big(Y_3 \le Y_2 \wedge \tY_2\big) \Biggiven \mX \Big] \ind\Big(1,2,N_1(1),N_1(2)~{\rm distinct}\Big)\Big] + o(1)\\
  = & \E \Big[\Cov\Big[\ind\big(Y_3 \le Y_1 \wedge \tY_1\big), \ind\big(Y_3 \le Y_2 \wedge \tY_2\big) \Biggiven \mX \Big] \Big] + o(1) \\
  = & \E \Big[\Cov\Big[\ind\big(Y_3 \le Y_1 \wedge \tY_1\big), \ind\big(Y_3 \le Y_2 \wedge \tY_2\big) \Biggiven X_1,X_2,X_3 \Big] \Big] + o(1).
\end{align*}

Similarly,
\begin{align*}
  & \E \Big[\Cov\Big[\ind\big(Y_1 \le Y_1 \wedge Y_{N_1(1)}\big), \ind\big(Y_3 \le Y_2 \wedge Y_{N_1(2)}\big) \Biggiven \mX \Big] \ind\Big(1,2,N_1(1),N_1(2)~{\rm distinct}\Big)\Big]\\
  = & \E \Big[\Cov\Big[\ind\big(Y_1 \le Y_1 \wedge \tY_1\big), \ind\big(Y_3 \le Y_2 \wedge \tY_2\big) \Biggiven \mX \Big]\Big] + o(1)\\
  = & \E \Big[\Cov\Big[\ind\big(Y_1 \le Y_1 \wedge \tY_1\big), F_Y\big(Y_2 \wedge \tY_2\big) \Biggiven X_1,X_2 \Big]\Big] + o(1),
\end{align*}
where the last step is by expanding the covariance in the same way as \eqref{eq:hayek5}.

Then it holds true that
\begin{align*}
   \E[T_5] =& (1+O(n^{-1})) \Big\{ \E \Big[\Cov\Big[\ind\big(Y_3 \le Y_1 \wedge \tY_1\big), \ind\big(Y_3 \le Y_2 \wedge \tY_2\big) \Biggiven X_1,X_2,X_3 \Big] \Big]\\
  &+ 2\E \Big[\Cov\Big[\ind\big(Y_1 \le Y_1 \wedge \tY_1\big), F_Y\big(Y_2 \wedge \tY_2\big) \Biggiven X_1,X_2 \Big] \Big]\\
  &+ 2\E \Big[\Cov\Big[\ind\big(\tY_1 \le Y_1 \wedge \tY_1\big), F_Y\big(Y_2 \wedge \tY_2\big) \Biggiven X_1,X_2 \Big] \Big]\\
  &+ 2\E \Big[\Cov\Big[\ind\big(Y_2 \le Y_1 \wedge \tY_1\big), F_Y\big(Y_2 \wedge \tY_2\big) \Biggiven X_1,X_2 \Big] \Big]\\
  &+ 2\E \Big[\Cov\Big[\ind\big(\tY_2 \le Y_1 \wedge \tY_1\big), F_Y\big(Y_2 \wedge \tY_2\big) \Biggiven X_1,X_2 \Big] \Big] \Big\} + o(1)\\
  =& (1+O(n^{-1})) \Big\{ \E \Big[\Cov\Big[\ind\big(Y_3 \le Y_1 \wedge \tY_1\big), \ind\big(Y_3 \le Y_2 \wedge \tY_2\big) \Biggiven X_1,X_2,X_3 \Big] \Big]\\
  &+ 4\E \Big[\Cov\Big[\ind\big(Y_2 \le Y_1 \wedge \tY_1\big), F_Y\big(Y_2 \wedge \tY_2\big) \Biggiven X_1,X_2 \Big] \Big]\Big\} + o(1).
  \yestag\label{eq:hayek21}
\end{align*}

On the other hand,
\begin{align*}
  & \E[T_5'] = \E \Big[\frac{1}{n^2} \sum_{i,j,N_1(i),N_1(j)~{\rm distinct}} \Cov\Big[ \min\big\{R_i, R_{N_1(i)}\big\}, \min\big\{F_Y(Y_j), F_Y(Y_{N_1(j)})\big\} \Biggiven \mX \Big] \Big]\\
  = & \frac{N_1(N-1)}{n^2} \E \Big[ \Cov\Big[ \min\big\{R_1, R_{N_1(1)}\big\}, \min\big\{F_Y(Y_2), F_Y(Y_{N_1(2)})\big\} \Biggiven \mX \Big] \ind\Big(1,2,N_1(1),N_1(2)~{\rm distinct}\Big) \Big]\\
  = & \frac{n-1}{n} \Big\{ \E \Big[\Cov\Big[\ind\big(Y_1 \le Y_1 \wedge Y_{N_1(1)}\big), F_Y\big(Y_2 \wedge Y_{N_1(2)}\big) \Biggiven \mX \Big] \ind\Big(1,2,N_1(1),N_1(2)~{\rm distinct}\Big)\Big]\\
  &+ \E \Big[\Cov\Big[\ind\big(Y_{N_1(1)} \le Y_1 \wedge Y_{N_1(1)}\big), F_Y\big(Y_2 \wedge Y_{N_1(2)}\big) \Biggiven \mX \Big] \ind\Big(1,2,N_1(1),N_1(2)~{\rm distinct}\Big)\Big]\\
  &+ \E \Big[\Cov\Big[\ind\big(Y_2 \le Y_1 \wedge Y_{N_1(1)}\big), F_Y\big(Y_2 \wedge Y_{N_1(2)}\big) \Biggiven \mX \Big] \ind\Big(1,2,N_1(1),N_1(2)~{\rm distinct}\Big)\Big]\\
  &+ \E \Big[\Cov\Big[\ind\big(Y_{N_1(2)} \le Y_1 \wedge Y_{N_1(1)}\big), F_Y\big(Y_2 \wedge Y_{N_1(2)}\big) \Biggiven \mX \Big] \ind\Big(1,2,N_1(1),N_1(2)~{\rm distinct}\Big)\Big] \Big\}\\
  =& (1+O(n^{-1})) \Big\{ 2\E \Big[\Cov\Big[\ind\big(Y_1 \le Y_1 \wedge \tY_1\big), F_Y\big(Y_2 \wedge \tY_2\big) \Biggiven X_1,X_2 \Big] \Big]\\
  &+ 2\E \Big[\Cov\Big[\ind\big(Y_2 \le Y_1 \wedge \tY_1\big), F_Y\big(Y_2 \wedge \tY_2\big) \Biggiven X_1,X_2 \Big] \Big]\Big\} + o(1)\\
  =& 2 (1+O(n^{-1})) \E \Big[\Cov\Big[\ind\big(Y_2 \le Y_1 \wedge \tY_1\big), F_Y\big(Y_2 \wedge \tY_2\big) \Biggiven X_1,X_2 \Big] \Big] + o(1).
  \yestag\label{eq:hayek22}
\end{align*}

Combining \eqref{eq:hayek21} and \eqref{eq:hayek22} completes the proof of the first claim.

The second claim is direct from the definition of $T_5^*$.
\end{proof}

\subsection{Proof of Lemma~\ref{lemma:hayek3}}

\begin{proof}[Proof of Lemma~\ref{lemma:hayek3}]
Since $[(X_i,Y_i)]_{i=1}^n$ are i.i.d. and $\min\big\{R_i, R_{N_1(i)}\big\}=\sum_{k=1}^n \ind\big(Y_k \le Y_i \wedge Y_{N_1(i)}\big)$ for any $i \in \zahl{n}$, we have
\begin{align*}
  & \E [T_6'] = \E \Big[\frac{1}{n^3} \sum_{i=1}^n \Cov\Big[ \min\big\{R_i, R_{N_1(i)}\big\}, \sum_{\substack{i,j=1\\i \neq j}}^n \min\big\{F_Y(Y_i), F_Y(Y_j)\big\} \Biggiven \mX \Big]\Big]\\
  = & \frac{1}{n^2} \E \Big[\Cov\Big[ \min\big\{R_1, R_{N_1(1)}\big\}, \sum_{\substack{i,j=1\\i \neq j}}^n \min\big\{F_Y(Y_i), F_Y(Y_j)\big\} \Biggiven \mX \Big]\Big]\\
  = & \frac{1}{n^2} \E \Big[\Cov\Big[ \sum_{k=1}^n \ind\big(Y_k \le Y_1 \wedge Y_{N_1(1)}\big), \sum_{\substack{i,j=1\\i \neq j}}^n F_Y\big(Y_i \wedge Y_j\big) \Biggiven \mX \Big]\Big]\\
  = & \frac{1}{n^2} \sum_{k=1}^n \E \Big[\Cov\Big[ \ind\big(Y_k \le Y_1 \wedge Y_{N_1(1)}\big), \sum_{\substack{i,j=1, i \neq j\\i=1,N_1(1),k~{\rm or}~j=1,N_1(1),k}}^n F_Y\big(Y_i \wedge Y_j\big) \Biggiven \mX \Big]\Big]\\
  = & \frac{n-1}{n^2} \E \Big[\Cov\Big[ \ind\big(Y_3 \le Y_1 \wedge Y_{N_1(1)}\big), \sum_{\substack{i,j=1, i \neq j\\i=1,N_1(1),3~{\rm or}~j=1,N_1(1),3}}^n F_Y\big(Y_i \wedge Y_j\big) \Biggiven \mX \Big]\Big] + O(n^{-1})\\
  = & \frac{(n-1)(n-2)}{n^2} \Big\{ 2\E\Big[\Cov\Big[ \ind\big(Y_3 \le Y_1 \wedge Y_{N_1(1)}\big), F_Y\big(Y_1 \wedge Y_2\big) \Biggiven \mX \Big]\Big]\\
  & + 2\E\Big[\Cov\Big[ \ind\big(Y_3 \le Y_1 \wedge Y_{N_1(1)}\big), F_Y\big(Y_{N_1(1)} \wedge Y_2\big) \Biggiven \mX \Big]\Big]\\
  & + 2\E\Big[\Cov\Big[ \ind\big(Y_3 \le Y_1 \wedge Y_{N_1(1)}\big), F_Y\big(Y_3 \wedge Y_2\big) \Biggiven \mX \Big]\Big]\Big\} + O(n^{-1})\\
  = & (1+O(n^{-1})) \Big\{ 2\E\Big[\Cov\Big[ \ind\big(Y_3 \le Y_1 \wedge \tY_1\big), F_Y\big(Y_1 \wedge Y_2\big) \Biggiven \mX \Big]\Big]\\
  & + 2\E\Big[\Cov\Big[ \ind\big(Y_3 \le Y_1 \wedge \tY_1\big), F_Y\big(\tY_1 \wedge Y_2\big) \Biggiven \mX \Big]\Big]\\
  & + 2\E\Big[\Cov\Big[ \ind\big(Y_3 \le Y_1 \wedge \tY_1\big), F_Y\big(Y_3 \wedge Y_2\big) \Biggiven \mX \Big]\Big]\Big\} + o(1)\\
  = & (1+O(n^{-1})) \Big\{ 4\E\Big[\Cov\Big[ F_Y\big(Y_1 \wedge \tY_1\big), F_Y\big(Y_1 \wedge Y_2\big) \Biggiven X_1,X_2 \Big]\Big]\\
  & + 2\E\Big[\Cov\Big[ \ind\big(Y_3 \le Y_1 \wedge \tY_1\big), F_Y\big(Y_3 \wedge Y_2\big) \Biggiven \mX \Big]\Big]\Big\} + o(1).
  \yestag\label{eq:hayek31}
\end{align*}

On the other hand,
\begin{align*}
   \E [T_6^*] = & \E \Big[\frac{1}{n^2} \sum_{i=1}^n \Cov\Big[ \min\big\{F_Y(Y_i), F_Y(Y_{N_1(i)})\big\}, \sum_{\substack{i,j=1\\i \neq j}}^n \min\big\{F_Y(Y_i), F_Y(Y_j)\big\} \Biggiven \mX \Big]\Big]\\
  = & \frac{1}{n} \E \Big[\Cov\Big[ \min\big\{F_Y(Y_1), F_Y(Y_{N_1(1)})\big\}, \sum_{\substack{i,j=1\\i \neq j}}^n \min\big\{F_Y(Y_i), F_Y(Y_j)\big\} \Biggiven \mX \Big]\\
  = & \frac{n-1}{n} \Big\{ 2\E\Big[\Cov\Big[ F_Y\big(Y_1 \wedge Y_{N_1(1)}\big), F_Y\big(Y_1 \wedge Y_2\big) \Biggiven \mX \Big]\Big]\\
  & + 2\E\Big[\Cov\Big[ F_Y\big(Y_1 \wedge Y_{N_1(1)}\big), F_Y\big(Y_{N_1(1)} \wedge Y_2\big) \Biggiven \mX \Big]\Big] \Big\} + O(n^{-1})\\
  = & (1+O(n^{-1})) \Big\{ 2\E\Big[\Cov\Big[ F_Y\big(Y_1 \wedge \tY_1\big), F_Y\big(Y_1 \wedge Y_2\big) \Biggiven \mX \Big]\Big]\\
  & + 2\E\Big[\Cov\Big[ F_Y\big(Y_1 \wedge \tY_1\big), F_Y\big(\tY_1 \wedge Y_2\big) \Biggiven \mX \Big]\Big] \Big\} + o(1)\\
  = & (1+O(n^{-1})) 4\E\Big[\Cov\Big[ F_Y\big(Y_1 \wedge \tY_1\big), F_Y\big(Y_1 \wedge Y_2\big) \Biggiven X_1,X_2 \Big]\Big] + o(1).
  \yestag\label{eq:hayek32}
\end{align*}

Combining \eqref{eq:hayek31} and \eqref{eq:hayek32} and expanding the covariance, we obtain
\begin{align*}
   \lim_{n \to \infty} [\E[T_6'] - \E[T_6^*]] =& 2\E\Big[\Cov\Big[ \ind\big(Y_3 \le Y_1 \wedge \tY_1\big), F_Y\big(Y_3 \wedge Y_2\big) \Biggiven \mX \Big]\Big]\\
  = & 2\E \Big[\Cov\Big[\ind\big(Y_2 \le Y_1 \wedge \tY_1\big), F_Y\big(Y_2 \wedge Y \big) \Biggiven X_1,X_2 \Big] \Big],
\end{align*}
and thus complete the proof.
\end{proof}

\subsection{Proof of Lemma~\ref{lemma:hayek4}}

\begin{proof}[Proof of Lemma~\ref{lemma:hayek4}]
Since $[(X_i,Y_i)]_{i=1}^n$ are i.i.d., we have 
\begin{align*}
   \E [T_7^*] =& \E \Big[ \frac{1}{n^3} \Var\Big[\sum_{\substack{i,j=1\\i \neq j}}^n \min\big\{F_Y(Y_i), F_Y(Y_j)\big\} \Biggiven \mX \Big] \Big]\\
  = & \frac{1}{n^3} \E \Big[ \Cov\Big[\sum_{\substack{i,j=1\\i \neq j}}^n F_Y\big(Y_i \wedge Y_j\big), \sum_{\substack{k,\ell=1\\k \neq \ell}}^n F_Y\big(Y_k \wedge Y_\ell\big)  \Biggiven \mX \Big] \Big] \\
  = & \frac{1}{n^3} \E \Big[ \sum_{\substack{i,j=1\\i \neq j}}^n \sum_{\substack{k,\ell=1\\k \neq \ell}}^n \Cov\Big[ F_Y\big(Y_i \wedge Y_j\big),  F_Y\big(Y_k \wedge Y_\ell\big)  \Biggiven \mX \Big] \Big].
\end{align*}
Notice that when $i,j,k,\ell$ are distinct, the covariance is zero. Then
\begin{align*}
  \E [T_7^*] =& \frac{4N_1(N-1)(n-2)}{n^3} \E \Big[ \Cov\Big[ F_Y\big(Y_1 \wedge Y_2\big),  F_Y\big(Y_1 \wedge Y_3\big)  \Biggiven \mX \Big] \Big] + \frac{2N_1(N-1)}{n^3} \E \Big[ \Var\Big[ F_Y\big(Y_1 \wedge Y_2\big)\Biggiven \mX \Big] \Big]\\
  = & (1+O(n^{-1})) 4 \E \Big[ \Cov\Big[ F_Y\big(Y_1 \wedge Y_2\big),  F_Y\big(Y_1 \wedge Y_3\big)  \Biggiven \mX \Big] \Big] + O(n^{-1}).
\end{align*}
Expanding the covariance, we obtain
\begin{align*}
  & \lim_{n \to \infty} \E[T_7^*] = 4 \E \Big[ \Cov\Big[ F_Y\big(Y_1 \wedge Y_2\big),  F_Y\big(Y_1 \wedge Y_3\big)  \Biggiven \mX \Big] \Big]   =  4 \E \Big[\Cov\Big[F_Y\big(Y_1 \wedge Y\big), F_Y\big(Y_1 \wedge \tY\big) \Biggiven X_1 \Big] \Big]
\end{align*}
and thus complete the proof.
\end{proof}

\subsection{Proof of Lemma~\ref{lemma:hayek8}}

\begin{proof}[Proof of Lemma~\ref{lemma:hayek8}]
For $T_8^*$,
\begin{align*}
  \E [T_8^*] =& \E \Big[\frac{1}{n} \sum_{i=1}^n \Cov\Big[ \min\big\{F_Y(Y_i), F_Y(Y_{N_1(i)})\big\}, \sum_{i=1}^n g(Y_i) \Biggiven \mX \Big] \Big]\\
  = & \E \Big[\Cov\Big[ F_Y\big(Y_1 \wedge Y_{N_1(1)}\big), \sum_{i=1}^n g(Y_i) \Biggiven \mX \Big] \Big] \\
  = & \E \Big[\Cov\Big[ F_Y\big(Y_1 \wedge Y_{N_1(1)}\big), g(Y_1) \Biggiven \mX \Big] \Big] + \E \Big[\Cov\Big[ F_Y\big(Y_1 \wedge Y_{N_1(1)}\big), g(Y_{N_1(1)}) \Biggiven \mX \Big] \Big]\\
  = & 2\E \Big[\Cov\Big[F_Y\big(Y_1 \wedge \tY_1\big), g(Y_1) \Biggiven X_1\Big] \Big] + o(1).
\end{align*}

For $T_9^*$,
\begin{align*}
  \E [T_9^*] =& \E \Big[\frac{1}{n^2} \Cov\Big[ \sum_{\substack{i,j=1\\i \neq j}}^n \min\big\{F_Y(Y_i), F_Y(Y_j)\big\} , \sum_{i=1}^n g(Y_i) \Biggiven \mX \Big] \Big]\\
  = & \frac{N_1(N-1)}{n^2} \E \Big[ \Cov\Big[ F_Y\big(Y_1 \wedge Y_2\big), \sum_{i=1}^n g(Y_i) \Biggiven \mX \Big] \Big]\\
  = & \frac{N_1(N-1)}{n^2} \Big[ \E \Big[\Cov\Big[ F_Y\big(Y_1 \wedge Y_2\big), g(Y_1) \Biggiven \mX \Big] \Big] + \E \Big[ \Cov\Big[ F_Y\big(Y_1 \wedge Y_2\big), g(Y_2) \Biggiven \mX \Big] \Big] \Big]\\
  = & 2(1+O(n^{-1})) \E \Big[\Cov\Big[F_Y\big(Y_1 \wedge Y_2\big), g(Y_1) \Biggiven X_1,X_2\Big] \Big]\\
  = & 2(1+O(n^{-1})) \E \Big[\Cov\Big[F_Y\big(Y_1 \wedge Y\big), g(Y_1) \Biggiven X_1\Big] \Big].
\end{align*}

For $T_7'$, we have
\begin{align*}
   &\E [T_7'] \\
   =& \E \Big[ \frac{1}{n^2} \sum_{i=1}^n \Cov\Big[ \min\big\{R_i, R_{N_1(i)}\big\}, \sum_{i=1}^n g(Y_i) \Biggiven \mX \Big] \Big]\\
  = & \frac{1}{n} \E \Big[ \Cov\Big[ R_1 \wedge R_{N_1(1)}, \sum_{i=1}^n g(Y_i) \Biggiven \mX \Big] \Big] = \frac{1}{n} \E \Big[ \Cov\Big[  \sum_{k=1}^n \ind\big(Y_k \le Y_1 \wedge Y_{N_1(1)}\big), \sum_{i=1}^n g(Y_i) \Biggiven \mX \Big] \Big]\\
  = & \frac{n-1}{n}\E \Big[ \Cov\Big[  \ind\big(Y_2 \le Y_1 \wedge Y_{N_1(1)}\big), \sum_{i=1}^n g(Y_i) \Biggiven \mX \Big] \Big] + O(n^{-1})\\
  = & (1+O(n^{-1})) \Big[ \E \Big[ \Cov\Big[  \ind\big(Y_2 \le Y_1 \wedge Y_{N_1(1)}\big), g(Y_1) \Biggiven \mX \Big] \Big] + \E \Big[ \Cov\Big[  \ind\big(Y_2 \le Y_1 \wedge Y_{N_1(1)}\big), g(Y_2) \Biggiven \mX \Big] \Big] \\
  & + \E \Big[ \Cov\Big[  \ind\big(Y_2 \le Y_1 \wedge Y_{N_1(1)}\big), g(Y_{N_1(1)}) \Biggiven \mX \Big] \Big]  \Big] + O(n^{-1})\\
  = & (1+O(n^{-1})) \Big[ 2\E \Big[\Cov\Big[\ind\big(Y_2 \le Y_1 \wedge \tY_1\big), g(Y_1) \Biggiven X_1\Big] \Big] + \E \Big[\Cov\Big[\ind\big(Y_2 \le Y_1 \wedge \tY_1\big), g(Y_2) \Biggiven X_1,X_2\Big] \Big] \Big] + o(1)\\
  = & (1+O(n^{-1})) \Big[ 2\E \Big[\Cov\Big[F_Y\big(Y_1 \wedge \tY_1\big), g(Y_1) \Biggiven X_1\Big] \Big] + \E \Big[\Cov\Big[\ind\big(Y_2 \le Y_1 \wedge \tY_1\big), g(Y_2) \Biggiven X_1,X_2\Big] \Big] \Big] + o(1).
\end{align*}

For $T_{10}^*$, the result is direct from the variance of the sample mean.
\end{proof}

\subsection{Proof of Lemma~\ref{lemma:hayek9}}

\begin{proof}[Proof of Lemma~\ref{lemma:hayek9}]
~\\
{\bf Part I.} $a_1 + 4a_2 + 4a_3$.

Recall that $G_X(t) = \P(Y \ge t \given X)$ and $h(t)= \E[G_X^2(t)]$. Then
\begin{align*}
  a_1 =& \E \Big[\Cov\Big[\ind\big(Y_3 \le Y_1 \wedge \tY_1\big), \ind\big(Y_3 \le Y_2 \wedge \tY_2\big) \Biggiven X_1,X_2,X_3 \Big] \Big]\\
  = & \E \Big[ \int G_{X_1}^2(t) G_{X_2}^2(t) \d \mu_{X_3}(t) - \Big(\int G_{X_1}^2(t) \d \mu_{X_3}(t) \Big) \Big(\int G_{X_2}^2(t) \d \mu_{X_3}(t) \d t \Big) \Big]\\
  = & \E \Big[ \int h^2(t) \d \mu_{X_3}(t) - \Big(\int h(t) \d \mu_{X_3}(t) \Big) \Big(\int h(t) \d \mu_{X_3}(t) \Big) \Big]\\
  = & \E \Big[ \Var \Big[ h(Y_1) \Biggiven X_1 \Big] \Big].
  \yestag\label{eq:hayek91}
\end{align*}

Let $h'(t):= F_Y(t) - F_Y^2(t)/2$. Notice that for $Y \sim F_Y$, we have $F_Y(Y) \sim U(0,1)$ from the probability integral transform. Then
\begin{align*}
   a_3 =& \E \Big[\Cov\Big[F_Y\big(Y_1 \wedge Y\big), F_Y\big(Y_1 \wedge \tY\big) \Biggiven X_1 \Big] \Big]\\
  = & \E \Big[ \int \E \Big[\Big(F_Y(t) \wedge F_Y(Y) \Big) \Big(F_Y(t) \wedge F_Y(\tY) \Big) \Big] \d \mu_{X_1}(t) \\
  & - \Big(\int \E \Big[F_Y(t) \wedge F_Y(Y) \Big] \d \mu_{X_1}(t) \Big) \Big(\int \E \Big[F_Y(t) \wedge F_Y(\tY) \Big] \d \mu_{X_1}(t) \Big) \Big]\\
  = & \E \Big[ \int (F_Y(t) -F_Y^2(t)/2)^2 \d \mu_{X_1}(t) \\
  & - \Big(\int (F_Y(t) -F_Y^2(t)/2) \d \mu_{X_1}(t) \Big) \Big(\int (F_Y(t) -F_Y^2(t)/2) \d \mu_{X_1}(t) \Big) \Big]\\
  = & \E \Big[ \int h'^2(t) \d \mu_{X_1}(t) - \Big(\int h'(t) \d \mu_{X_1}(t) \Big) \Big(\int h'(t) \d \mu_{X_1}(t) \Big) \Big]\\
  = & \E \Big[ \Var \Big[ h'(Y_1) \Biggiven X_1 \Big] \Big].
  \yestag\label{eq:hayek92}
\end{align*}

In the same way as $a_1$ and $a_3$,
\begin{align*}
  a_2 =& \E \Big[\Cov\Big[\ind\big(Y_2 \le Y_1 \wedge \tY_1\big), F_Y\big(Y_2 \wedge Y \big) \Biggiven X_1,X_2 \Big] \Big]\\
  = & \E \Big[ \int G_{X_1}^2(t)(F_Y(t) -F_Y^2(t)/2) \d \mu_{X_2}(t) \\
  & - \Big(\int G_{X_1}^2(t) \d \mu_{X_2}(t) \d t \Big) \Big(\int (F_Y(t) -F_Y^2(t)/2) \d \mu_{X_2}(t) \Big) \Big]\\
  = & \E \Big[ \int h(t)h'(t) \d \mu_{X_2}(t) - \Big(\int h(t) \d \mu_{X_2}(t) \Big) \Big(\int h'(t) \d \mu_{X_2}(t) \Big) \Big]\\
  = & \E \Big[ \Cov \Big[ h(Y_1), h'(Y_1) \Biggiven X_1 \Big] \Big].
  \yestag\label{eq:hayek93}
\end{align*}

Noticing that
\begin{align*}
  & h(t) = \E[G_X^2(t)] = g(t) + G^2(t) = g(t) + (1-F_Y(t))^2 = 1 - 2h'(t) + g(t).
  \yestag\label{eq:hayek96}
\end{align*}
and combining \eqref{eq:hayek91}-\eqref{eq:hayek96} yields
\begin{align*}
   a_1 + 4a_2 + 4a_3 =& \E \Big[ \Var \Big[ h(Y_1) \Biggiven X_1 \Big] \Big] + 4 \E \Big[ \Cov \Big[ h(Y_1), h'(Y_1) \Biggiven X_1 \Big] \Big] + 4 \E \Big[ \Var \Big[ h'(Y_1) \Biggiven X_1 \Big] \Big]\\
  = & \E \Big[ \Var \Big[ h(Y_1) + 2h'(Y_1) \Biggiven X_1 \Big] \Big] \\
  = &\E \Big[ \Var \Big[ g(Y_1) \Biggiven X_1 \Big] \Big].
\end{align*}
The first part's proof is then complete.

\vspace{0.2cm}

{\bf Part II.} $b_3 - 2b_1 + 2b_2$.

In the same way as the first part,
\begin{align*}
   b_3 - 2b_1 =& \E \Big[\Cov\Big[\ind\big(Y_2 \le Y_1 \wedge \tY_1\big), g(Y_2) \Biggiven X_1,X_2 \Big] \Big]\\
  = & \E \Big[ \int G_{X_1}^2(t) g(t) \d \mu_{X_2}(t) - \Big(\int G_{X_1}^2(t) \d \mu_{X_2}(t) \Big) \Big(\int g(t) \d \mu_{X_2}(t) \Big) \Big]\\
  = & \E \Big[ \int h(t) g(t) \d \mu_{X_2}(t) - \Big(\int h(t) \d \mu_{X_2}(t) \Big) \Big(\int g(t) \d \mu_{X_2}(t) \Big) \Big]\\
  = & \E \Big[ \Cov \Big[ h(Y_1), g(Y_1) \Biggiven X_1 \Big] \Big]
  \yestag\label{eq:hayek94}
\end{align*}
and
\begin{align*}
   b_2 =& \E \Big[\Cov\Big[F_Y\big(Y_1 \wedge Y \big), g(Y_1) \Biggiven X_1\Big] \Big]\\
  = & \E \Big[ \int (F_Y(t) -F_Y^2(t)/2)^2 g(t) \d \mu_{X_1}(t) - \Big(\int (F_Y(t) -F_Y^2(t)/2)^2 \d \mu_{X_1}(t) \Big) \Big(\int g(t) \d \mu_{X_1}(t) \Big) \Big]\\
  = & \E \Big[ \int h'(t) g(t) \d \mu_{X_1}(t) - \Big(\int h'(t) \d \mu_{X_1}(t) \Big) \Big(\int g(t) \d \mu_{X_1}(t) \Big) \Big]\\
  = & \E \Big[ \Cov \Big[ h'(Y_1), g(Y_1) \Biggiven X_1 \Big] \Big].
  \yestag\label{eq:hayek95}
\end{align*}

Combining \eqref{eq:hayek96}-\eqref{eq:hayek95} yields
\begin{align*}
   b_3 - 2b_1 + 2b_2 =& \E \Big[ \Cov \Big[ h(Y_1), g(Y_1) \Biggiven X_1 \Big] \Big] + 2 \E \Big[ \Cov \Big[ h'(Y_1), g(Y_1) \Biggiven X_1 \Big] \Big]\\
  = & \E \Big[ \Cov \Big[ h(Y_1) + 2h'(Y_1), g(Y_1) \Biggiven X_1 \Big] \Big] \\
  = &\E \Big[ \Var \Big[ g(Y_1) \Biggiven X_1 \Big] \Big].
\end{align*}
The second part's proof is then complete.
\end{proof}

\subsection{Proof of Lemma~\ref{lemma:hayek5}}

\begin{proof}[Proof of Lemma~\ref{lemma:hayek5}]

From the boundedness of the indicator function and $h$, we have
\begin{align*}
  &\tT_1\\
   =&  \E\Big[\frac{1}{n} \sum_{i=1}^n \Big(\E \Big[\ind\big(Y_\ell \le Y_i \wedge Y_{N_1(i)}\big) \Biggiven \mX \Big] - h(X_\ell) \Big) \Big]^2\\
  = & \E\Big[\frac{1}{n} \sum_{i=1,i\neq \ell}^n \Big(\E \Big[\ind\big(Y_\ell \le Y_i \wedge Y_{N_1(i)}\big) \Biggiven \mX \Big] - h(X_\ell) \Big) \Big]^2 + O(n^{-1})\\
  = & (1+O(n^{-1})) \E\Big[\Big(\E \Big[\ind\big(Y_3 \le Y_1 \wedge Y_{N_1(1)}\big) \Biggiven \mX \Big] - h(X_3) \Big) \Big(\E \Big[\ind\big(Y_3 \le Y_2 \wedge Y_{N_1(2)}\big) \Biggiven \mX \Big] - h(X_3) \Big) \Big] + O(n^{-1})\\
  = & (1+O(n^{-1})) \E\Big[\Big(\E \Big[\ind\big(Y_3 \le Y_1 \wedge \tY_1\big) \Biggiven X_1,X_3 \Big] - h(X_3) \Big) \Big(\E \Big[\ind\big(Y_3 \le Y_2 \wedge \tY_2 \big) \Biggiven X_2,X_3 \Big] - h(X_3) \Big) \Big] + o(1).
\end{align*}
Since $[X_i]_{i=1}^n$ are i.i.d., by definition of $h$, we have
\begin{align*}
  & \E\Big[\Big(\E \Big[\ind\big(Y_3 \le Y_1 \wedge \tY_1\big) \Biggiven X_1,X_3 \Big] - h(X_3) \Big) \Big(\E \Big[\ind\big(Y_3 \le Y_2 \wedge \tY_2 \big) \Biggiven X_2,X_3 \Big] - h(X_3) \Big) \Big]\\
  = & \E\Big[ \E\Big[\Big(\E \Big[\ind\big(Y_3 \le Y_1 \wedge \tY_1\big) \Biggiven X_1,X_3 \Big] - h(X_3) \Big) \Big(\E \Big[\ind\big(Y_3 \le Y_2 \wedge \tY_2 \big) \Biggiven X_2,X_3 \Big] - h(X_3) \Big) \Biggiven X_3 \Big] \Big]\\
  = & \E\Big[ \E\Big[\E \Big[\ind\big(Y_3 \le Y_1 \wedge \tY_1\big) \Biggiven X_1,X_3 \Big] - h(X_3) \Biggiven X_3 \Big] \E \Big[\E \Big[\ind\big(Y_3 \le Y_2 \wedge \tY_2 \big) \Biggiven X_2,X_3 \Big] - h(X_3) \Biggiven X_3 \Big] \Big]\\
  = & \E\Big[ \Big(h(X_3) - h(X_3)\Big)^2 \Big] \\
  = & 0.
\end{align*}
We then complete the proof.
\end{proof}

\subsection{Proof of Lemma~\ref{lemma:hayek6}}

\begin{proof}[Proof of Lemma~\ref{lemma:hayek6}]
Since the indicator function and $F_Y$ are both bounded and $[X_i]_{i=1}^n$ are i.i.d.,
\begin{align*}
  \tT_2 =& \E\Big[ \frac{1}{n} \sum_{k=1,k \neq \ell}^n \Big(\E \Big[\ind\big(Y_k \le Y_\ell \wedge Y_{N_1(\ell)}\big) \Biggiven \mX \Big] - \E \Big[ F_Y\big(Y_\ell \wedge Y_{N_1(\ell)}\big) \Biggiven \mX \Big] \Big) \Big]^2\\
  = & (1+O(n^{-1})) \E\Big[ \Big(\E \Big[\ind\big(Y_2 \le Y_1 \wedge Y_{N_1(1)}\big) \Biggiven \mX \Big] - \E \Big[ F_Y\big(Y_1 \wedge Y_{N_1(1)}\big) \Biggiven \mX \Big] \Big) \\
  & \Big(\E \Big[\ind\big(Y_3 \le Y_1 \wedge Y_{N_1(1)}\big) \Biggiven \mX \Big] - \E \Big[ F_Y\big(Y_1 \wedge Y_{N_1(1)}\big) \Biggiven \mX \Big] \Big) \Big] + O(n^{-1})\\
  = & (1+O(n^{-1})) \E\Big[ \Big(\E \Big[\ind\big(Y_2 \le Y_1 \wedge \tY_1\big) \Biggiven X_1,X_2 \Big] - \E \Big[ F_Y\big(Y_1 \wedge \tY_1\big) \Biggiven X_1 \Big] \Big) \\
  & \Big(\E \Big[\ind\big(Y_3 \le Y_1 \wedge \tY_1\big) \Biggiven X_1,X_3 \Big] - \E \Big[ F_Y\big(Y_1 \wedge \tY_1\big) \Biggiven X_1 \Big] \Big) \Big] + o(1)\\
  = & (1+O(n^{-1})) \E\Big[ \E \Big[\E \Big[\ind\big(Y_2 \le Y_1 \wedge \tY_1\big) \Biggiven X_1,X_2 \Big] - \E \Big[ F_Y\big(Y_1 \wedge \tY_1\big) \Biggiven X_1 \Big] \Biggiven X_1 \Big] \\
  & \E \Big[\E \Big[\ind\big(Y_3 \le Y_1 \wedge \tY_1\big) \Biggiven X_1,X_3 \Big] - \E \Big[ F_Y\big(Y_1 \wedge \tY_1\big) \Biggiven X_1 \Big] \Biggiven X_1 \Big] \Big] + o(1)\\
  = & (1+O(n^{-1})) \E\Big[ \Big(\E \Big[ F_Y\big(Y_1 \wedge \tY_1\big) \Biggiven X_1 \Big] - \E \Big[ F_Y\big(Y_1 \wedge \tY_1\big) \Biggiven X_1 \Big] \Big)^2 \Big] + o(1) \\
  =& o(1). 
\end{align*}
The proof is then complete.
\end{proof}

\subsection{Proof of Lemma~\ref{lemma:hayek7}}

\begin{proof}[Proof of Lemma~\ref{lemma:hayek7}]
Lemma 20.6 together with Theorem 20.16 in \cite{biau2015lectures} show that $\lvert \{i:N_1(i) = \ell\} \rvert$, $\lvert \{i:\tN_1(i) = \ell\} \rvert$ are both bounded by a constant that only depend on $d$. Notice that $\P(N_1(1)=4),\P(\tN_1(1)=4) = O(n^{-1})$. We assume $\ell = 4$ without loss of generality. Then from the Cauchy–Schwarz inequality,
\begin{align*}
   \tT_3 =& \E\Big[\sum_{\substack{i=1\\N_1(i)=\ell~{\rm or}~\tN_1(i)=\ell}}^n \Big(\frac{1}{n} \sum_{k=1,k \neq \ell}^n \E \Big[\ind\big(Y_k \le Y_i \wedge Y_{N_1(i)}\big) \Biggiven \mX \Big] - \E \Big[ F_Y\big(Y_i \wedge Y_{N_1(i)}\big) \Biggiven \mX \Big] \Big) \Big]^2\\ 
  \le& \E\Big[\Big\lvert \Big\{i:N_1(i)=\ell~{\rm or}~\tN_1(i)=\ell\Big\} \Big\rvert \\
  & \sum_{\substack{i=1\\N_1(i)=\ell~{\rm or}~\tN_1(i)=\ell}}^n \Big(\frac{1}{n} \sum_{k=1,k \neq \ell}^n \E \Big[\ind\big(Y_k \le Y_i \wedge Y_{N_1(i)}\big) \Biggiven \mX \Big] - \E \Big[ F_Y\big(Y_i \wedge Y_{N_1(i)}\big) \Biggiven \mX \Big] \Big)^2 \Big]\\
  \lesssim & \E\Big[ \sum_{\substack{i=1\\N_1(i)=\ell~{\rm or}~\tN_1(i)=\ell}}^n \Big(\frac{1}{n} \sum_{k=1,k \neq \ell}^n \E \Big[\ind\big(Y_k \le Y_i \wedge Y_{N_1(i)}\big) \Biggiven \mX \Big] - \E \Big[ F_Y\big(Y_i \wedge Y_{N_1(i)}\big) \Biggiven \mX \Big] \Big)^2 \Big]\\
  = & (n-1) \E\Big[ \Big(\frac{1}{n} \sum_{k=1,k \neq \ell}^n \E \Big[\ind\big(Y_k \le Y_1 \wedge Y_{N_1(1)}\big) \Biggiven \mX \Big] - \E \Big[ F_Y\big(Y_1 \wedge Y_{N_1(1)}\big) \Biggiven \mX \Big] \Big)^2 \\
  & \ind\Big(N_1(1)=\ell~{\rm or}~\tN_1(1)=\ell\Big)\Big]\\
  \le & (n-1) \E\Big[ \Big(\frac{1}{n} \sum_{k=1,k \neq \ell}^n \E \Big[\ind\big(Y_k \le Y_1 \wedge Y_{N_1(1)}\big) \Biggiven \mX \Big] - \E \Big[ F_Y\big(Y_1 \wedge Y_{N_1(1)}\big) \Biggiven \mX \Big] \Big)^2 \\
  & \Big[\ind\Big(N_1(1)=\ell\Big) + \ind\Big(\tN_1(1)=\ell\Big) \Big]\Big]\\
  = & n(1+O(n^{-1})) \E\Big[ \Big(\E \Big[\ind\big(Y_2 \le Y_1 \wedge Y_{N_1(1)}\big) \Biggiven \mX \Big] - \E \Big[ F_Y\big(Y_1 \wedge Y_{N_1(1)}\big) \Biggiven \mX \Big] \Big) \\
  & \Big(\E \Big[\ind\big(Y_3 \le Y_1 \wedge Y_{N_1(1)}\big) \Biggiven \mX \Big] - \E \Big[ F_Y\big(Y_1 \wedge Y_{N_1(1)}\big) \Biggiven \mX \Big] \Big) \Big[\ind\Big(N_1(1)=\ell\Big) + \ind\Big(\tN_1(1)=\ell\Big) \Big] \Big] \\
  & + O\Big(\P\Big(N_1(1)=\ell\Big) + \P\Big(\tN_1(1)=\ell\Big)\Big)\\
  = & 2 (1+O(n^{-1})) \E\Big[ \Big(\E \Big[\ind\big(Y_2 \le Y_1 \wedge Y_{N_1(1)}\big) \Biggiven \mX \Big] - \E \Big[ F_Y\big(Y_1 \wedge Y_{N_1(1)}\big) \Biggiven \mX \Big] \Big) \\
  & \Big(\E \Big[\ind\big(Y_3 \le Y_1 \wedge Y_{N_1(1)}\big) \Biggiven \mX \Big] - \E \Big[ F_Y\big(Y_1 \wedge Y_{N_1(1)}\big) \Biggiven \mX \Big] \Big) \Big] + O(n^{-1}).
\end{align*}
The last step is true since $\sum_{k=4}^n \ind(N_1(1)=k) = 1-\ind(N_1(1)=2,3)$, $\P(N_1(1)=2,3) = O(n^{-1})$, and $[X_i]_{i=1}^n$ are i.i.d..

Invoking the same idea as used in the proof of Lemma~\ref{lemma:hayek6} then completes the proof.
\end{proof}

\subsection{Proof of Lemma~\ref{lemma:var1}}

\begin{proof}[Proof of Lemma~\ref{lemma:var1}]
For the first statement, notice that
\begin{align*}
  &\E \Big[ \Var \Big[ F_Y\big(Y_1 \wedge \tY_1\big) \Biggiven X_1 \Big] \Big] = \E \Big[ \E \Big[ F_Y^2\big(Y_1 \wedge \tY_1\big) \Biggiven X_1 \Big] \Big] - \E \Big[ \Big(\E \Big[ F_Y\big(Y_1 \wedge \tY_1\big) \Biggiven X_1 \Big] \Big)^2 \Big]\\
  =&\E \Big[ \E \Big[ F_Y^2\big(Y_1 \wedge \tY_1\big) \Biggiven X_1 \Big] \Big] - \E \Big[ \E \Big[ F_Y\big(Y_1 \wedge \tY_1\big) F_Y\big(\tY_1' \wedge \tY_1''\big) \Biggiven X_1 \Big] \Big],
\end{align*}
where $\tY_1',\tY_1''$ are independently drawn from $Y \given X_1$ and are further independent of $Y_1, \tY_1$ conditional on $X_1$.

For the first term above, letting $F_Y^{(n)}$ be the empirical distribution of $\{Y_i\}_{i=1}^n$, one then has
\begin{align*}
  &\frac{1}{n^3} \sum_{i=1}^n \Big(R_i \wedge R_{N_1(i)}\Big)^2 = \frac{1}{n} \sum_{i=1}^n \Big(F_Y^{(n)}\big(Y_i \wedge Y_{N_1(i)}\big) \Big)^2\\
  =& \Big[\frac{1}{n} \sum_{i=1}^n \Big(F_Y^{(n)}\big(Y_i \wedge Y_{N_1(i)}\big) \Big)^2 - \frac{1}{n} \sum_{i=1}^n \Big(F_Y\big(Y_i \wedge Y_{N_1(i)}\big) \Big)^2\Big] + \frac{1}{n} \sum_{i=1}^n F_Y^2\big(Y_i \wedge Y_{N_1(i)}\big).
\end{align*}

Using the Glivenko-Cantelli theorem (Theorem 19.1 in \cite{MR1652247}) and that fact that $F_Y,F_Y^{(n)}$ are bounded by 1, one has
\begin{align*}
  \Big\lvert\frac{1}{n} \sum_{i=1}^n \Big(F_Y^{(n)}\big(Y_i \wedge Y_{N_1(i)}\big) \Big)^2 - \frac{1}{n} \sum_{i=1}^n \Big(F_Y\big(Y_i \wedge Y_{N_1(i)}\big) \Big)^2\Big\rvert \le 2\lVert F_Y^{(n)} - F_Y \rVert_\infty \stackrel{\sf a.s.}{\longrightarrow} 0,
\end{align*}
with ``$\stackrel{\sf a.s.}{\longrightarrow}$'' representing strong convergence.

Then it suffices to consider the second term. We use bias-variance decomposition. Notice that
\begin{align*}
  &\E \Big[\frac{1}{n} \sum_{i=1}^n F_Y^2\big(Y_i \wedge Y_{N_1(i)}\big)\Big] = \E \Big[ \E \Big[F_Y^2\big(Y_1 \wedge Y_{N_1(1)}\big) \Biggiven \mX \Big]\Big]\\
  =& \E \Big[ \E \Big[ \int \ind\Big(Y_1 \wedge Y_{N_1(1)} \ge t_1\Big) \ind\Big(Y_1 \wedge Y_{N_1(1)} \ge t_2\Big) \d \mu_Y(t_1) \d \mu_Y(t_2)\Biggiven \mX \Big]\Big]\\
  =& \E \Big[ \E \Big[ \int \ind\Big(Y_1 \ge t_1 \vee t_2\Big) \ind\Big(Y_{N_1(1)} \ge t_1 \vee t_2\Big) \d \mu_Y(t_1) \d \mu_Y(t_2)\Biggiven \mX \Big]\Big]\\
  =& \E \Big[ \int G_{X_1}\big(t_1 \vee t_2\big) G_{X_{N_1(1)}}\big(t_1 \vee t_2\big) \d \mu_Y(t_1) \d \mu_Y(t_2)\Big].
\end{align*}

On the other hand, one can check that
\begin{align*}
  \E \Big[ \E \Big[ F_Y^2\big(Y_1 \wedge \tY_1\big) \Biggiven X_1 \Big] \Big] = \E \Big[ \int G_{X_1}^2\big(t_1 \vee t_2\big) \d \mu_Y(t_1) \d \mu_Y(t_2)\Big].
\end{align*}

Lemma 11.7 in \cite{azadkia2019simple} then implies that the bias is
\begin{align*}
  &\limsup_{n\to\infty} \Big\lvert \E \Big[\frac{1}{n} \sum_{i=1}^n F_Y^2\big(Y_i \wedge Y_{N_1(i)}\big)\Big] - \E \Big[ \E \Big[ F_Y^2\big(Y_1 \wedge \tY_1\big) \Biggiven X_1 \Big] \Big] \Big\rvert\\
  =& \limsup_{n\to\infty} \Big\lvert \E \Big[ \int G_{X_1}\big(t_1 \vee t_2\big) \Big(G_{X_{N_1(1)}}\big(t_1 \vee t_2\big) -  G_{X_1}\big(t_1 \vee t_2\big)\Big) \d \mu_Y(t_1) \d \mu_Y(t_2)\Big] \Big\rvert = 0.
\end{align*}

From the Efron-Stein inequality and the fact that $\lvert \{j:N_1(j)=i\} \rvert$ is always bounded for any $i \in \zahl{n}$, the variance is
\begin{align*}
  \Var \Big[\frac{1}{n} \sum_{i=1}^n F_Y^2\big(Y_i \wedge Y_{N_1(i)}\big)\Big] = O\Big(\frac{1}{n}\Big).
\end{align*}

Combining the bias and the variance yields
\begin{align*}
  \frac{1}{n^3} \sum_{i=1}^n \Big(R_i \wedge R_{N_1(i)}\Big)^2 - \E \Big[ \E \Big[ F_Y^2\big(Y_1 \wedge \tY_1\big) \Biggiven X_1 \Big] \Big] \stackrel{\sf p}{\longrightarrow} 0.
\end{align*}

In the same way and noticing that $i,N_1(i),N_2(i),N_3(i)$ are all different for any $i \in \zahl{n}$,
\begin{align*}
  &\frac{1}{n^3} \sum_{i=1}^n \Big(R_i \wedge R_{N_1(i)}\Big)  \Big(R_{N_2(i)} \wedge R_{N_3(i)}\Big) = \frac{1}{n} \sum_{i=1}^n F_Y^{(n)}\big(Y_i \wedge Y_{N_1(i)}\big)  F_Y^{(n)}\big(Y_{N_2(i)} \wedge Y_{N_3(i)}\big)\\
  =& \frac{1}{n} \sum_{i=1}^n F_Y\big(Y_i \wedge Y_{N_1(i)}\big)  F_Y\big(Y_{N_2(i)} \wedge Y_{N_3(i)}\big) + o_\P(1) \\
  =& \E \Big[ \E \Big[ F_Y\big(Y_1 \wedge \tY_1\big) F_Y\big(\tY_1' \wedge \tY_1''\big) \Biggiven X_1 \Big] \Big] + o_\P(1).
\end{align*}

We then complete the proof of the first statement, and the fourth statement holds in the same way. The second and the third statements can be established similarly by noticing that
\begin{align*}
  \Cov\Big[F_Y\big(Y_1 \wedge \tY_1\big), F_Y\big(\tY_1 \wedge \tY_1' \big) \Biggiven X_1 \Big] = \E \Big[ F_Y\big(Y_1 \wedge \tY_1\big) F_Y\big(\tY_1 \wedge \tY_1' \big) \Biggiven X_1 \Big] - \E \Big[ F_Y\big(Y_1 \wedge \tY_1\big) F_Y\big(\tY_1' \wedge \tY_1''\big) \Biggiven X_1 \Big].
\end{align*}

For the fifth statement,
\begin{align*}
  & \frac{1}{n^2(n-1)} \sum_{\substack{i,j=1\\i\neq j}}^n \ind\big(R_i \le R_j \wedge R_{N_1(j)}\big)R_i \wedge R_{N_1(i)} = \frac{1}{N_1(N-1)} \sum_{\substack{i,j=1\\i\neq j}}^n \ind\big(Y_i \le Y_j \wedge Y_{N_1(j)}\big) F_Y^{(n)}\big(Y_i \wedge Y_{N_1(i)}\big)\\
  =& \frac{1}{N_1(N-1)} \sum_{\substack{i,j=1\\i\neq j}}^n \ind\big(Y_i \le Y_j \wedge Y_{N_1(j)}\big) F_Y\big(Y_i \wedge Y_{N_1(i)}\big) + o_\P(1).
\end{align*}

Notice that $\P(N_1(1)=2)$ and $\P(N_1(1)=N_1(2))$ are both $O(n^{-1})$. Then the expectation is
\begin{align*}
  &\E\Big[\frac{1}{N_1(N-1)} \sum_{\substack{i,j=1\\i\neq j}}^n \ind\big(Y_i \le Y_j \wedge Y_{N_1(j)}\big) F_Y\big(Y_i \wedge Y_{N_1(i)}\big)\Big] = \E\Big[ \ind\big(Y_2 \le Y_1 \wedge Y_{N_1(1)}\big) F_Y\big(Y_2 \wedge Y_{N_1(2)}\big) \Big]\\
  =& \E\Big[ \int \ind\big(Y_2 \le Y_1 \wedge Y_{N_1(1)}\big)\ind\big(Y_2 \ge t\big) \ind\big(Y_{N_1(2)} \ge t\big) \d \mu_Y(t)\Big]\\
  =& \E\Big[ \int G_{X_1}(Y_2) G_{X_{N_1(1)}}(Y_2) \ind\big(Y_2 \ge t\big) \ind\big(Y_{N_1(2)} \ge t\big) \d \mu_Y(t)\Big] + O\Big(\frac{1}{n}\Big)\\
  =& \E\Big[ \int G_{X_1}^2(Y_2) \ind\big(Y_2 \ge t\big) \ind\big(Y_{N_1(2)} \ge t\big) \d \mu_Y(t)\Big] + o(1)\\ 
  =& \E\Big[ \int h(Y_2) \ind\big(Y_2 \ge t\big) \ind\big(Y_{N_1(2)} \ge t\big) \d \mu_Y(t)\Big] + o(1)\\
  =& \E\Big[ \int G_{X_2}^*(t) G_{X_{N_1(2)}}(t) \d \mu_Y(t)\Big] + o(1) = \E\Big[ \int G_{X_2}^*(t) G_{X_2}(t) \d \mu_Y(t)\Big] + o(1)\\
  =& \int \E\Big[ G_{X}^*(t) G_{X}(t) \Big] \d \mu_Y(t) + o(1),
\end{align*}
where $G_{x}^*(t) := \E[h(Y) \ind(Y \ge t) \given X=x]$ for $x \in \bR^d$.

On the other hand, we can check
\begin{align*}
  &\E \Big[\E\Big[\ind\big(Y_2 \le Y_1 \wedge \tY_1\big) F_Y\big(Y_2 \wedge \tY_2\big) \Biggiven X_1,X_2 \Big] \Big] = \int \E\Big[ G_{X}^*(t) G_{X}(t) \Big] \d \mu_Y(t).
\end{align*}

Then the fifth statement is established by using the same argument as before. The sixth statement can also be established in the same way.
\end{proof}

\subsection{Proof of Lemma~\ref{lemma:var2}}

\begin{proof}[Proof of Lemma~\ref{lemma:var2}]
The proof is similar to that of Lemma~\ref{lemma:var1}. The key is to notice that from the definitions of $h_0$ and $h_1$,
\begin{align*}
  &\Var \Big[  h_0(X_1) \Big] = \E \Big[  h_0^2(X_1) \Big] - \Big(\E \Big[  h_0(X_1) \Big]\Big)^2\\
  =& \E \Big[\E\Big[\ind\big(Y_3 \le Y_1 \wedge \tY_1\big) \ind\big(\tY_3 \le Y_2 \wedge \tY_2\big) \Biggiven X_1,X_2,X_3 \Big] \Big] - \Big(\E\Big[ \E\Big[F_Y\big(Y_1 \wedge \tY_1\big) \Biggiven X_1\Big] \Big] \Big)^2,\\
  & \Cov \Big[h_0(X_1), h_1(X_1)\Big] = \E \Big[h_0(X_1) h_1(X_1)\Big] - \E \Big[h_0(X_1)\Big] \E \Big[h_1(X_1)\Big]\\
  =& \E \Big[\E\Big[\ind\big(Y_2 \le Y_1 \wedge \tY_1\big) F_Y\big(\tY_2 \wedge \tY_2'\big) \Biggiven X_1,X_2\Big] \Big] - \Big(\E\Big[ \E\Big[F_Y\big(Y_1 \wedge \tY_1\big) \Biggiven X_1\Big] \Big] \Big)^2,\\
  &\Var \Big[  h_1(X_1) \Big] = \E \Big[  h_1^2(X_1) \Big] - \Big(\E \Big[  h_1(X_1) \Big]\Big)^2\\
  =& \E \Big[\E\Big[F_Y\big(Y_1 \wedge \tY_1\big) F_Y\big(\tY_1' \wedge \tY_1''\big) \Biggiven X_1\Big] \Big] - \Big(\E\Big[ \E\Big[F_Y\big(Y_1 \wedge \tY_1\big) \Biggiven X_1\Big] \Big] \Big)^2.
\end{align*}
All the rest is the same.
\end{proof}

\section{Proofs of the results in the supplement}\label{sec:proof-supp}

\nb{
\subsection{Proof of Theorem~\ref{thm:sobol}}

\begin{proof}[Proof of Theorem~\ref{thm:sobol}]
Let $\mX^\fu = [X^\fu_i]_{i=1}^n$. The joint central limit theorem is similar to the proof of Theorem~\ref{thm:clt} by using \citet[Theorem 3.4]{MR2435859} combining with the Cramér–Wold theorem. We only need to calculate $\Sigma$, i.e., the limits of $n \Var[\frac{1}{n} \sum_{i=1}^n Y_i Y_{N^\fu_1(i)} - (\frac{1}{n}\sum_{i=1}^n Y_i)^2]$, $n \Cov[\frac{1}{n} \sum_{i=1}^n Y_i Y_{N^\fu_1(i)} - (\frac{1}{n}\sum_{i=1}^n Y_i)^2, \frac{1}{n}\sum_{i=1}^n Y_i^2 - (\frac{1}{n}\sum_{i=1}^n Y_i)^2]$ and $n\Var[\frac{1}{n}\sum_{i=1}^n Y_i^2 - (\frac{1}{n}\sum_{i=1}^n Y_i)^2]$.

{\bf Part I.} We decompose $\Var[\frac{1}{n} \sum_{i=1}^n Y_i Y_{N^\fu_1(i)} - (\frac{1}{n}\sum_{i=1}^n Y_i)^2]$ as
\begin{align*}
  & n \Var\Big[\frac{1}{n} \sum_{i=1}^n Y_i Y_{N^\fu_1(i)} - \Big(\frac{1}{n}\sum_{i=1}^n Y_i\Big)^2 \Big]\\
  =& n \E\Big[\Var\Big[\frac{1}{n} \sum_{i=1}^n Y_i Y_{N^\fu_1(i)} \Biggiven \mX^\fu \Big]\Big] + n \Var\Big[\E\Big[\frac{1}{n} \sum_{i=1}^n Y_i Y_{N^\fu_1(i)} \Biggiven \mX^\fu \Big]\Big]\\
  & - 2n \E\Big[\Cov\Big[\frac{1}{n} \sum_{i=1}^n Y_i Y_{N^\fu_1(i)}, \Big(\frac{1}{n}\sum_{i=1}^n Y_i\Big)^2 \Biggiven \mX^\fu \Big]\Big] - 2n \Cov\Big[\E\Big[\frac{1}{n} \sum_{i=1}^n Y_i Y_{N^\fu_1(i)} \Biggiven \mX^\fu \Big], \E\Big[\Big(\frac{1}{n}\sum_{i=1}^n Y_i\Big)^2 \Biggiven \mX^\fu \Big]\Big]\\
  &+ n \Var\Big[\Big(\frac{1}{n}\sum_{i=1}^n Y_i\Big)^2\Big].
  \yestag\label{eq:sobol1}
\end{align*}

For the first term in \eqref{eq:sobol1},
\begin{align*}
  & n \Var\Big[\frac{1}{n} \sum_{i=1}^n Y_i Y_{N^\fu_1(i)} \Biggiven \mX^\fu \Big]\\
  =& \frac{1}{n}\sum_{i=1}^n \Var\Big[Y_i Y_{N^\fu_1(i)} \Biggiven \mX^\fu \Big] + \frac{1}{n} \sum_{\substack{j=N^\fu_1(i),i \neq N^\fu_1(j)\\ {\rm or}~i=N^\fu_1(j),j \neq N^\fu_1(i)}} \Cov\Big[Y_i Y_{N^\fu_1(i)}, Y_j Y_{N^\fu_1(j)} \Biggiven \mX^\fu \Big] \\
  & + \frac{1}{n} \sum_{\substack{i \neq j \\ N^\fu_1(i) = N^\fu_1(j)}} \Cov\Big[Y_i Y_{N^\fu_1(i)}, Y_j Y_{N^\fu_1(j)} \Biggiven \mX^\fu \Big] + \frac{1}{n} \sum_{j = N^\fu_1(i), i = N^\fu_1(j)} \Cov\Big[Y_i Y_{N^\fu_1(i)}, Y_j Y_{N^\fu_1(j)} \Biggiven \mX^\fu \Big] \\
  & + \frac{1}{n} \sum_{i,j,N^\fu_1(i),N^\fu_1(j)~{\rm distinct}} \Cov\Big[Y_i Y_{N^\fu_1(i)}, Y_j Y_{N^\fu_1(j)} \Biggiven \mX^\fu \Big]\\
  :=& T^\fu_1 + T^\fu_2 + T^\fu_3+ T^\fu_4 + T^\fu_5 .
\end{align*}

Note that when $i,j,N^\fu_1(i),N^\fu_1(j)~{\rm distinct}$, we have $\Cov[Y_i Y_{N^\fu_1(i)}, Y_j Y_{N^\fu_1(j)} \given \mX^\fu ] = 0$. Then
\begin{align*}
  T^\fu_5 = \frac{1}{n} \sum_{i,j,N^\fu_1(i),N^\fu_1(j)~{\rm distinct}} \Cov\Big[Y_i Y_{N^\fu_1(i)}, Y_j Y_{N^\fu_1(j)} \Biggiven \mX^\fu \Big] = 0.
\end{align*}

For the first four terms, similar to Lemma~\ref{lemma:hayek1}, we have
\begin{align*}
  & \E\Big[T^\fu_1\Big] - \E \Big[\Var \Big[ Y_1\tY_1 \Biggiven X^\fu_1 \Big] \Big] \longrightarrow 0,\\
  & \E\Big[T^\fu_2\Big] - 2\E \Big[\Cov\Big[Y_1 \tY_1, \tY_1 \tY_1' \Biggiven X^\fu_1 \Big]\ind \Big( 1 \neq N^\fu_1(N^\fu_1(1)) \Big) \Big] \longrightarrow 0,\\
  & \E\Big[T^\fu_3\Big] - \E \Big[\Cov\Big[Y_1 \tY_1, \tY_1 \tY_1' \Biggiven X^\fu_1 \Big] \Big\lvert \Big\{j: j \neq 1, N^\fu_1(j) = N^\fu_1(1)\Big\} \Big\rvert \Big] \longrightarrow 0,\\
  & \E\Big[T^\fu_4\Big] - \E \Big[ \Var \Big[ Y_1\tY_1 \Biggiven X^\fu_1 \Big] \ind \Big( 1 = N^\fu_1(N^\fu_1(1)) \Big) \Big] \longrightarrow 0,
\end{align*}
where $\tY_1, \tY_1'$ are sampled independently from the conditional distribution of $Y_1$ given $X^\fu_1$.

Then we have
\begin{align*}
  &n \E\Big[\Var\Big[\frac{1}{n} \sum_{i=1}^n Y_i Y_{N^\fu_1(i)} \Biggiven \mX^\fu \Big]\Big] - \E \Big[ \Var \Big[ Y_1\tY_1 \Biggiven X^\fu_1 \Big] \Big(1 + \ind \Big( 1 = N^\fu_1(N^\fu_1(1)) \Big)\Big) \Big] \\
  &- \E\Big[\Cov\Big[Y_1 \tY_1, \tY_1 \tY_1' \Biggiven X^\fu_1 \Big] \Big(2\ind \Big( 1 \neq N^\fu_1(N^\fu_1(1)) \Big) + \Big\lvert \Big\{j: j \neq 1, N^\fu_1(j) = N^\fu_1(1)\Big\} \Big\rvert \Big)\Big] \longrightarrow 0.
  \yestag\label{eq:sobol11}
\end{align*}

As in Lemma~\ref{lemma:var1}, the corresponding estimators are
\begin{align*}
  &\frac{1}{n} \sum_{i=1}^n \Big[ \Big(Y_i Y_{N^\fu_1(i)}\Big)  \Big(Y_i  Y_{N^\fu_1(i)} - Y_{N^\fu_2(i)} Y_{N^\fu_3(i)}\Big)\Big] - \E \Big[\Var \Big[ Y_1\tY_1 \Biggiven X^\fu_1 \Big] \Big] \stackrel{\sf p}{\longrightarrow} 0,\\
  &\frac{1}{n} \sum_{i=1}^n \Big[\Big(Y_i Y_{N^\fu_1(i)}\Big)\Big(Y_i Y_{N^\fu_2(i)} - Y_{N^\fu_2(i)} Y_{N^\fu_3(i)}\Big) \ind \Big( i \neq N^\fu_1(N^\fu_1(i)) \Big) \Big] \\
  &- \E \Big[\Cov\Big[Y_1 \tY_1, \tY_1 \tY_1' \Biggiven X^\fu_1 \Big]\ind \Big( 1 \neq N^\fu_1(N^\fu_1(1)) \Big) \Big] \stackrel{\sf p}{\longrightarrow} 0,\\
  &\frac{1}{n} \sum_{i=1}^n \Big[\Big(Y_i Y_{N^\fu_1(i)}\Big)\Big(Y_i Y_{N^\fu_2(i)} - Y_{N^\fu_2(i)} Y_{N^\fu_3(i)}\Big)\Big\lvert \Big\{j: j \neq i, N^\fu_1(j) = N^\fu_1(i)\Big\} \Big\rvert \Big] \\
  &- \E \Big[\Cov\Big[Y_1 \tY_1, \tY_1 \tY_1' \Biggiven X^\fu_1 \Big] \Big\lvert \Big\{j: j \neq 1, N^\fu_1(j) = N^\fu_1(1)\Big\} \Big\rvert \Big] \stackrel{\sf p}{\longrightarrow} 0,\\
  &\frac{1}{n} \sum_{i=1}^n \Big[ \Big(Y_i Y_{N^\fu_1(i)}\Big)  \Big(Y_i Y_{N^\fu_1(i)} - Y_{N^\fu_2(i)} Y_{N^\fu_3(i)}\Big) \ind \Big( i = N^\fu_1(N^\fu_1(i)) \Big) \Big] \\
  &- \E \Big[ \Var \Big[ Y_1\tY_1 \Biggiven X^\fu_1 \Big] \ind \Big( 1 = N^\fu_1(N^\fu_1(1)) \Big) \Big] \stackrel{\sf p}{\longrightarrow} 0.
\end{align*}

Then the estimator for the first term in \eqref{eq:sobol1} is
\begin{align*}
  &\frac{1}{n} \sum_{i=1}^n Y_i^2 Y_{N^\fu_1(i)}^2 \Big(1 + \ind \Big( i = N^\fu_1(N^\fu_1(i)) \Big) \Big)\\
  &+ \frac{1}{n} \sum_{i=1}^n Y_i^2 Y_{N^\fu_1(i)} Y_{N^\fu_2(i)} \Big(2\ind \Big( i \neq N^\fu_1(N^\fu_1(i)) \Big) + \Big\lvert \Big\{j: j \neq 1, N^\fu_1(j) = N^\fu_1(i)\Big\} \Big\rvert \Big)\\
  &- \frac{1}{n} \sum_{i=1}^n Y_i Y_{N^\fu_1(i)} Y_{N^\fu_2(i)} Y_{N^\fu_3(i)} \Big(2 + \ind \Big( i \neq N^\fu_1(N^\fu_1(i)) \Big) + \Big\lvert \Big\{j: j \neq 1, N^\fu_1(j) = N^\fu_1(i)\Big\} \Big\rvert \Big)\\
  -& n \E\Big[\Var\Big[\frac{1}{n} \sum_{i=1}^n Y_i Y_{N^\fu_1(i)} \Biggiven \mX^\fu \Big]\Big] \stackrel{\sf p}{\longrightarrow} 0.
  \yestag\label{eq:sobol16}
\end{align*}

For the second term in \eqref{eq:sobol1}, similar to Lemma~\ref{lemma:variance,cond} by using the Efron-Stein inequality, we have
\begin{align*}
  & n \Var\Big[\E\Big[\frac{1}{n} \sum_{i=1}^n Y_i Y_{N^\fu_1(i)} \Biggiven \mX^\fu \Big]\Big] - \Var\Big[\E\Big[ Y_1 \tY_1 \Biggiven X^\fu_1 \Big]\Big] \longrightarrow 0.
  \yestag\label{eq:sobol12}
\end{align*}

The estimator for the second term in \eqref{eq:sobol1} is 
\begin{align*}
  \frac{1}{n} \sum_{i=1}^n Y_i Y_{N^\fu_1(i)}  Y_{N^\fu_2(i)} Y_{N^\fu_3(i)} - \Big(\frac{1}{n} \sum_{i=1}^n Y_i Y_{N^\fu_1(i)}\Big)^2 - \Var\Big[\E\Big[ Y_1 \tY_1 \Biggiven X^\fu_1 \Big]\Big] \stackrel{\sf p}{\longrightarrow} 0
  \yestag\label{eq:sobol17}
\end{align*}

For the third term in \eqref{eq:sobol1},
\begin{align*}
  & n \Cov\Big[\frac{1}{n} \sum_{i=1}^n Y_i Y_{N^\fu_1(i)}, \Big(\frac{1}{n}\sum_{i=1}^n Y_i\Big)^2 \Biggiven \mX^\fu \Big]\\
  =& \frac{1}{n^2} \Cov\Big[\sum_{i=1}^n Y_i Y_{N^\fu_1(i)}, \sum_{i=1}^n Y_i^2 \Biggiven \mX^\fu \Big] + \frac{1}{n^2} \Cov\Big[\sum_{i=1}^n Y_i Y_{N^\fu_1(i)}, \sum_{i\neq j} Y_i Y_j \Biggiven \mX^\fu \Big]\\
  =& \frac{1}{n^2} \sum_{i=1}^n \Cov\Big[ Y_i Y_{N^\fu_1(i)}, Y_i^2 + Y_{N^\fu_1(i)}^2 \Biggiven \mX^\fu \Big] \\
  &+ \frac{2}{n^2} \sum_{i=1}^n \Cov\Big[ Y_i Y_{N^\fu_1(i)}, Y_i \sum_{j \neq i}Y_j + Y_{N^\fu_1(i)} \sum_{j \neq N^\fu_1(i)}Y_j - Y_i Y_{N^\fu_1(i)} \Biggiven \mX^\fu \Big].
\end{align*}

Then we have
\begin{align*}
  & n \Cov\Big[\frac{1}{n} \sum_{i=1}^n Y_i Y_{N^\fu_1(i)}, \Big(\frac{1}{n}\sum_{i=1}^n Y_i\Big)^2 \Biggiven \mX^\fu \Big] - 4 \E \Big[\Cov \Big[ Y_1\tY_1, Y_1 Y_2 \Biggiven X^\fu_1 \Big] \Big] \longrightarrow 0.
  \yestag\label{eq:sobol13}
\end{align*}

The estimator for the third term in \eqref{eq:sobol1} is 
\begin{align*}
  &4 \Big[\Big(\frac{1}{n} \sum_{i=1}^n Y_i^2 Y_{N^\fu_1(i)}\Big) \Big(\frac{1}{n} \sum_{i=1}^n Y_i \Big) - \Big(\frac{1}{n} \sum_{i=1}^n Y_i Y_{N^\fu_1(i)} Y_{N^\fu_2(i)}\Big) \Big(\frac{1}{n} \sum_{i=1}^n Y_i \Big)\Big] \\
  &- n \E\Big[\Cov\Big[\frac{1}{n} \sum_{i=1}^n Y_i Y_{N^\fu_1(i)}, \Big(\frac{1}{n}\sum_{i=1}^n Y_i\Big)^2 \Biggiven \mX^\fu \Big]\Big] \stackrel{\sf p}{\longrightarrow} 0.
  \yestag\label{eq:sobol18}
\end{align*}

For the fourth term in \eqref{eq:sobol1}, again similar to Lemma~\ref{lemma:variance,cond}, we have
\begin{align*}
  & n \Cov\Big[\E\Big[\frac{1}{n} \sum_{i=1}^n Y_i Y_{N^\fu_1(i)} \Biggiven \mX^\fu \Big], \E\Big[\Big(\frac{1}{n}\sum_{i=1}^n Y_i\Big)^2 \Biggiven \mX^\fu \Big]\Big]\\
  =& n \Cov\Big[\frac{1}{n} \sum_{i=1}^n \E\Big[ Y_i \tY_i \Biggiven X^\fu_i \Big], \E\Big[\Big(\frac{1}{n}\sum_{i=1}^n Y_i\Big)^2 \Biggiven \mX^\fu \Big]\Big] + o(1)\\
  =& n \Cov\Big[\E\Big[ Y_1 \tY_1 \Biggiven X^\fu_1 \Big], \E\Big[\Big(\frac{1}{n}\sum_{i=1}^n Y_i\Big)^2 \Biggiven \mX^\fu \Big]\Big] + o(1)\\
  =& \frac{1}{n} \Cov\Big[\E\Big[ Y_1 \tY_1 \Biggiven X^\fu_1 \Big], \E\Big[Y_1^2 \Biggiven X^\fu_1 \Big] \Big] + 2(1-\frac{1}{n}) \Cov\Big[\E\Big[ Y_1 \tY_1 \Biggiven X^\fu_1 \Big], \E\Big[Y_1Y_2 \Biggiven X^\fu_1, X^\fu_2 \Big] \Big] + o(1).
\end{align*}

Then we have
\begin{align*}
  n \Cov\Big[\E\Big[\frac{1}{n} \sum_{i=1}^n Y_i Y_{N^\fu_1(i)} \Biggiven \mX^\fu \Big], \E\Big[\Big(\frac{1}{n}\sum_{i=1}^n Y_i\Big)^2 \Biggiven \mX^\fu \Big]\Big] - 2 \Cov\Big[\E\Big[ Y_1 \tY_1 \Biggiven X^\fu_1 \Big], \E\Big[Y_1Y_2 \Biggiven X^\fu_1, X^\fu_2 \Big] \Big]  \longrightarrow 0.
  \yestag\label{eq:sobol14}
\end{align*}

The estimator for the fourth term in \eqref{eq:sobol1} is
\begin{align*}
    &2\Big[\Big(\frac{1}{n} \sum_{i=1}^n Y_i Y_{N^\fu_1(i)} Y_{N^\fu_2(i)}\Big) \Big(\frac{1}{n} \sum_{i=1}^n Y_i\Big) - \Big(\frac{1}{n} \sum_{i=1}^n Y_i Y_{N^\fu_1(i)}\Big) \Big(\frac{1}{n} \sum_{i=1}^n Y_i\Big)^2\Big] \\
    &- n \Cov\Big[\E\Big[\frac{1}{n} \sum_{i=1}^n Y_i Y_{N^\fu_1(i)} \Biggiven \mX^\fu \Big], \E\Big[\Big(\frac{1}{n}\sum_{i=1}^n Y_i\Big)^2 \Biggiven \mX^\fu \Big]\Big] \stackrel{\sf p}{\longrightarrow} 0.
    \yestag\label{eq:sobol19}
\end{align*}

For the fifth term in \eqref{eq:sobol1}, by the Delta method, we have
\begin{align*}
  n \Var\Big[\Big(\frac{1}{n}\sum_{i=1}^n Y_i\Big)^2\Big] - 4 \Var[Y_1] \{\E[Y_1]\}^2 \longrightarrow 0.
  \yestag\label{eq:sobol15}
\end{align*}

Then the estimator for the fifth term in \eqref{eq:sobol1} is
\begin{align*}
  4\Big[\frac{1}{n} \sum_{i=1}^n Y_i^2  - \Big(\frac{1}{n} \sum_{i=1}^n Y_i\Big)^2\Big]\Big(\frac{1}{n} \sum_{i=1}^n Y_i\Big)^2 - n \Var\Big[\Big(\frac{1}{n}\sum_{i=1}^n Y_i\Big)^2\Big] \stackrel{\sf p}{\longrightarrow} 0.
  \yestag\label{eq:sobol110}
\end{align*}

Combining \eqref{eq:sobol11}, \eqref{eq:sobol12}, \eqref{eq:sobol13}, \eqref{eq:sobol14}, \eqref{eq:sobol15} using \eqref{eq:sobol1} yields the limit of $n \Var[\frac{1}{n} \sum_{i=1}^n Y_i Y_{N^\fu_1(i)} - (\frac{1}{n}\sum_{i=1}^n Y_i)^2]$. Combining \eqref{eq:sobol16}, \eqref{eq:sobol17}, \eqref{eq:sobol18}, \eqref{eq:sobol19}, \eqref{eq:sobol110} using \eqref{eq:sobol1} provides the consistent estimator.

{\bf Part II.} We decompose $\Cov[\frac{1}{n} \sum_{i=1}^n Y_i Y_{N^\fu_1(i)} - (\frac{1}{n}\sum_{i=1}^n Y_i)^2, \frac{1}{n}\sum_{i=1}^n Y_i^2 - (\frac{1}{n}\sum_{i=1}^n Y_i)^2]$ as 
\begin{align*}
  & n \Cov\Big[\frac{1}{n} \sum_{i=1}^n Y_i Y_{N^\fu_1(i)} - (\frac{1}{n}\sum_{i=1}^n Y_i)^2, \frac{1}{n}\sum_{i=1}^n Y_i^2 - (\frac{1}{n}\sum_{i=1}^n Y_i)^2\Big]\\
  =& n \E\Big[\Cov\Big[\frac{1}{n} \sum_{i=1}^n Y_i Y_{N^\fu_1(i)}, \frac{1}{n}\sum_{i=1}^n Y_i^2 \Biggiven \mX^\fu \Big]\Big] + n \Cov\Big[\E\Big[\frac{1}{n} \sum_{i=1}^n Y_i Y_{N^\fu_1(i)} \Biggiven \mX^\fu \Big], \E\Big[\frac{1}{n}\sum_{i=1}^n Y_i^2 \Biggiven \mX^\fu \Big]\Big]\\
  & - n \E\Big[\Cov\Big[\frac{1}{n} \sum_{i=1}^n Y_i Y_{N^\fu_1(i)}, \Big(\frac{1}{n}\sum_{i=1}^n Y_i\Big)^2 \Biggiven \mX^\fu \Big]\Big] - n \Cov\Big[\E\Big[\frac{1}{n} \sum_{i=1}^n Y_i Y_{N^\fu_1(i)} \Biggiven \mX^\fu \Big], \E\Big[\Big(\frac{1}{n}\sum_{i=1}^n Y_i\Big)^2 \Biggiven \mX^\fu \Big]\Big]\\
  &- \Cov\Big[\Big(\frac{1}{n}\sum_{i=1}^n Y_i\Big)^2, \frac{1}{n}\sum_{i=1}^n Y_i^2\Big] + n \Var\Big[\Big(\frac{1}{n}\sum_{i=1}^n Y_i\Big)^2\Big].
  \yestag\label{eq:sobol2}
\end{align*}

For the first term in \eqref{eq:sobol2}, we have
\begin{align*}
  n \Cov\Big[\frac{1}{n} \sum_{i=1}^n Y_i Y_{N^\fu_1(i)}, \frac{1}{n}\sum_{i=1}^n Y_i^2 \Biggiven \mX^\fu \Big] = \frac{1}{n} \sum_{i=1}^n \Cov[Y_i Y_{N^\fu_1(i)}, Y_i^2 + Y_{N^\fu_1(i)}^2 \given \mX^\fu],
\end{align*}
and then
\begin{align*}
  n \E\Big[\Cov\Big[\frac{1}{n} \sum_{i=1}^n Y_i Y_{N^\fu_1(i)}, \frac{1}{n}\sum_{i=1}^n Y_i^2 \Biggiven \mX^\fu \Big]\Big] - 2 \E \Big[\Cov \Big[ Y_1\tY_1, Y_1^2 \Biggiven X^\fu_1 \Big] \Big] \longrightarrow 0.
  \yestag\label{eq:sobol21}
\end{align*}

The estimator for the first term in \eqref{eq:sobol2} is 
\begin{align*}
  &2 \Big[\frac{1}{n} \sum_{i=1}^n Y_i^3 Y_{N^\fu_1(i)} - \frac{1}{n} \sum_{i=1}^n Y_i^2 Y_{N^\fu_1(i)} Y_{N^\fu_2(i)}\Big] - n \E\Big[\Cov\Big[\frac{1}{n} \sum_{i=1}^n Y_i Y_{N^\fu_1(i)}, \frac{1}{n}\sum_{i=1}^n Y_i^2 \Biggiven \mX^\fu \Big]\Big] \stackrel{\sf p}{\longrightarrow} 0.
  \yestag\label{eq:sobol24}
\end{align*}

For the second term in \eqref{eq:sobol2}, we have
\begin{align*}
  n \Cov\Big[\E\Big[\frac{1}{n} \sum_{i=1}^n Y_i Y_{N^\fu_1(i)} \Biggiven \mX^\fu \Big], \E\Big[\frac{1}{n}\sum_{i=1}^n Y_i^2 \Biggiven \mX^\fu \Big]\Big] - \Cov\Big[\E\Big[ Y_1 \tY_1 \Biggiven X^\fu_1 \Big], \E\Big[Y_1^2 \Biggiven X^\fu_1 \Big] \Big] \longrightarrow 0.
  \yestag\label{eq:sobol22}
\end{align*}

The estimator for the second term in \eqref{eq:sobol2} is 
\begin{align*}
  \Big[\frac{1}{n} \sum_{i=1}^n Y_i^2 Y_{N^\fu_1(i)} Y_{N^\fu_2(i)} &- \Big(\frac{1}{n} \sum_{i=1}^n Y_i Y_{N^\fu_1(i)}\Big) \Big(\frac{1}{n} \sum_{i=1}^n Y_i^2\Big)\Big] -\\ &n \Cov\Big[\E\Big[\frac{1}{n} \sum_{i=1}^n Y_i Y_{N^\fu_1(i)} \Biggiven \mX^\fu \Big], \E\Big[\frac{1}{n}\sum_{i=1}^n Y_i^2 \Biggiven \mX^\fu \Big]\Big] \stackrel{\sf p}{\longrightarrow} 0.
  \yestag\label{eq:sobol25}
\end{align*}

The third term and the fourth term in \eqref{eq:sobol2} are the same as third term and the fourth term in \eqref{eq:sobol1}.

For the fifth term in \eqref{eq:sobol2}, we have
\begin{align*}
  n \Cov\Big[\Big(\frac{1}{n}\sum_{i=1}^n Y_i\Big)^2, \frac{1}{n}\sum_{i=1}^n Y_i^2\Big] - 2 \Cov\Big[Y_1^2, Y_1 Y_2\Big] \longrightarrow 0.
  \yestag\label{eq:sobol23}
\end{align*}

The estimator for the fifth term in \eqref{eq:sobol2} is 
\begin{align*}
  &2 \Big[\Big(\frac{1}{n} \sum_{i=1}^n Y_i^3\Big) \Big(\frac{1}{n} \sum_{i=1}^n Y_i\Big) - \Big(\frac{1}{n} \sum_{i=1}^n Y_i^2\Big) \Big(\frac{1}{n} \sum_{i=1}^n Y_i\Big)^2\Big] - n \Cov\Big[(\frac{1}{n}\sum_{i=1}^n Y_i)^2, \frac{1}{n}\sum_{i=1}^n Y_i^2\Big] \stackrel{\sf p}{\longrightarrow} 0.
  \yestag\label{eq:sobol26}
\end{align*}

The sixth term in \eqref{eq:sobol2} is the same as the fifth term in \eqref{eq:sobol1}.

Combining \eqref{eq:sobol21}, \eqref{eq:sobol22}, \eqref{eq:sobol13}, \eqref{eq:sobol14}, \eqref{eq:sobol23}, \eqref{eq:sobol15} using \eqref{eq:sobol2} yields the limit of 
\[
n \Cov[\frac{1}{n} \sum_{i=1}^n Y_i Y_{N^\fu_1(i)} - (\frac{1}{n}\sum_{i=1}^n Y_i)^2, \frac{1}{n}\sum_{i=1}^n Y_i^2 - (\frac{1}{n}\sum_{i=1}^n Y_i)^2]. 
\]
Combining \eqref{eq:sobol24}, \eqref{eq:sobol25}, \eqref{eq:sobol18}, \eqref{eq:sobol19}, \eqref{eq:sobol26}, \eqref{eq:sobol110} using \eqref{eq:sobol2} provides the consistent estimator.

{\bf Part III.} We decompose $\Var[\frac{1}{n}\sum_{i=1}^n Y_i^2 - (\frac{1}{n}\sum_{i=1}^n Y_i)^2]$ as 
\begin{align*}
  n \Var\Big[\frac{1}{n}\sum_{i=1}^n Y_i^2 - \Big(\frac{1}{n}\sum_{i=1}^n Y_i\Big)^2\Big] = n \Var\Big[\frac{1}{n}\sum_{i=1}^n Y_i^2 \Big] - 2n \Cov\Big[\frac{1}{n}\sum_{i=1}^n Y_i^2, \Big(\frac{1}{n}\sum_{i=1}^n Y_i\Big)^2\Big] \\
  + n \Var\Big[\Big(\frac{1}{n}\sum_{i=1}^n Y_i\Big)^2\Big].
  \yestag\label{eq:sobol3}
\end{align*}

For the first term in \eqref{eq:sobol3}, we have
\begin{align*}
  n \Var\Big[\frac{1}{n}\sum_{i=1}^n Y_i^2 \Big] - \Var[Y_1^2] = 0.
  \yestag\label{eq:sobol31}
\end{align*}

The estimator for the first term in \eqref{eq:sobol3} is
\begin{align*}
  \Big[\frac{1}{n} \sum_{i=1}^n Y_i^4  - \Big(\frac{1}{n} \sum_{i=1}^n Y_i^2\Big)^2\Big] - n \Var\Big[\frac{1}{n}\sum_{i=1}^n Y_i^2 \Big] \stackrel{\sf p}{\longrightarrow} 0.
  \yestag\label{eq:sobol32}
\end{align*}

The second term in \eqref{eq:sobol3} is the same as the fifth term in \eqref{eq:sobol2}. The third term in \eqref{eq:sobol3} is the same as the fifth term in \eqref{eq:sobol1}. 

Combining \eqref{eq:sobol31}, \eqref{eq:sobol23}, \eqref{eq:sobol15} using \eqref{eq:sobol3} yields the limit of $n \Var[\frac{1}{n}\sum_{i=1}^n Y_i^2 - (\frac{1}{n}\sum_{i=1}^n Y_i)^2]$. Combining \eqref{eq:sobol32}, \eqref{eq:sobol26}, \eqref{eq:sobol110} using \eqref{eq:sobol3} provides the consistent estimator.
\end{proof}

\subsection{Proof of Theorem~\ref{lemma:sobol}}

\begin{proof}[Proof of Theorem~\ref{lemma:sobol}]
By the assumptions, the properties of the nearest neighbor distance on a compact support in one dimension, and the dominated convergence theorem, we have
\begin{align*}
  \lvert B^\fu \rvert = & \lvert \E[Y_1 Y_{N^\fu_1(1)}] - \E\{(\E[Y\given X^\fu])^2\} \rvert = \lvert \E[\E[Y_1 \given X_1^\fu] \E[Y_{N^\fu_1(1)} \given \mX^\fu]] - \E\{(\E[Y\given X^\fu])^2\} \rvert \\
  =& \lvert \E[\E[Y_1 \given X_1^\fu] (\E[Y_{N^\fu_1(1)} \given \mX^\fu] - \E[Y_1 \given X_1^\fu])] \rvert\\
  \le& \lVert \E[Y \given X^\fu = x] \rVert_\infty \lVert \frac{\d}{\d x} \E[Y \given X^\fu = x] \rVert_\infty\E[\lvert X^\fu_{N^\fu_1(1)} - X^\fu_1 \rvert] = O(n^{-1}).
\end{align*}
This completes the proof.
\end{proof}

\subsection{Proof of Theorem~\ref{thm:sobol0}}

\begin{proof}[Proof of Theorem~\ref{thm:sobol0}]

The proof is direct by applying the Delta method on the bivariate function $f(x,y) = x/y$. $\hat\sigma^2$ is consistent since both $\xi_n^\fu$ and $\hat\Sigma$ are consistent when estimating $S^\fu$ and $\Sigma$, respectively.
\end{proof}

}

\subsection{Proof of Lemma~\ref{lemma:variance,cond}}

\begin{proof}[Proof of Lemma~\ref{lemma:variance,cond}]

For any $x_1,x_2 \in \bR^d$, define $\Phi(x_1,x_2) := \E[F_Y(Y_1 \wedge Y_2) \given X_1=x_1, X_2 = x_2]$. Then by the definition of $\xi_n^*$ in \eqref{eq:xin*},
\begin{align*}
   \E [\xi_n^* \given \mX] =& \frac{6n}{n^2-1} \E \Big[\sum_{i=1}^n \min\big\{F_Y(Y_i), F_Y(Y_{N_1(i)})\big\} + \sum_{i=1}^n h(Y_i) \Biggiven \mX \Big] \\
  = & \frac{6n}{n^2-1} \Big(\sum_{i=1}^n \Phi(X_i,X_{N_1(i)}) + \sum_{i=1}^n h_0(X_i) \Big).
\end{align*}

To apply the Efron-Stein inequality, we implement the same notation as used in the Step II in the proof of Theorem~\ref{thm:hayek}. It is then true that
\begin{align*}
  & n \Var \Big[ \frac{6n}{n^2-1} \sum_{i=1}^n \Big(h_1(X_i) + h_0(X_i)\Big) - \E[\xi_n^* \given \mX] \Big]\\
  = & n \Var \Big[ \frac{6n}{n^2-1} \sum_{i=1}^n \Big(\Phi(X_i,X_{N_1(i)}) - h_1(X_i) \Big) \Big]\\
  = & \frac{36n^3}{(n^2-1)^2} \Var \Big[ \sum_{i=1}^n \Big(\Phi(X_i,X_{N_1(i)}) - h_1(X_i) \Big) \Big]\\
  \le & \frac{18n^3}{(n^2-1)^2} \sum_{\ell=1}^n \E\Big[\Phi(X_\ell,X_{N_1(\ell)}) - h_1(X_\ell) - \Phi(\tX_\ell,X_{\tN_1(\ell)}) + h_1(\tX_\ell) \\
  & + \sum_{\substack{i=1\\N_1(i)=\ell~{\rm or}~\tN_1(i)=\ell}}^n \Big(\Phi(X_i,X_{N_1(i)}) - \Phi(X_i,X'_{\tN_1(i)}) \Big) \Big]^2\\
  = & \frac{18n^4}{(n^2-1)^2} \E\Big[\Phi(X_\ell,X_{N_1(\ell)}) - h_1(X_\ell) - \Phi(\tX_\ell,X_{\tN_1(\ell)}) + h_1(\tX_\ell) \\
  & + \sum_{\substack{i=1\\N_1(i)=\ell~{\rm or}~\tN_1(i)=\ell}}^n \Big(\Phi(X_i,X_{N_1(i)}) - \Phi(X_i,X'_{\tN_1(i)}) \Big) \Big]^2,
\end{align*}
where $X'_{\tN_1(i)} = X_{\tN_1(i)}$ if $\tN_1(i) \neq \ell$ and $X'_{\tN_1(i)} = \tX_{\tN_1(i)}$ if $\tN_1(i) = \ell$.

From Lemma 11.3 in \cite{azadkia2019simple}, $X_{N_1(1)} \to X_1$ almost surely. Then similar to the proof of Lemma 11.7 in \cite{azadkia2019simple}, one can establish $\Phi(X_\ell,X_{N_1(\ell)}) - \Phi(X_\ell,X_\ell)$ converges to zero in probability. Noticing that $\Phi(X_\ell,X_\ell) = h_1(X_\ell)$ from the definition of $h_1$, one deduces
\[
  \lim_{n \to \infty} \E\Big[\Phi(X_\ell,X_{N_1(\ell)}) - h_1(X_\ell)\Big]^2 = 0,~~ \lim_{n \to \infty} \E\Big[\Phi(\tX_\ell,X_{\tN_1(\ell)}) - h_1(\tX_\ell)\Big]^2 = 0.
\]

Similar to the proof of Lemma~\ref{lemma:hayek7}, we then have
\[
  \lim_{n \to \infty} \E\Big[ \sum_{\substack{i=1\\N_1(i)=\ell~{\rm or}~\tN_1(i)=\ell}}^n \Big(\Phi(X_i,X_{N_1(i)}) - \Phi(X_i,X'_{\tN_1(i)}) \Big) \Big]^2 = 0.
\]

Leveraging the Cauchy–Schwarz inequality then completes the proof.
\end{proof}

{
\bibliographystyle{apalike}
\bibliography{AMS}
}

\end{document}